\documentclass[reqno,twoside,a4paper]{amsart}

\makeatletter                                                  %
\def\th@plain{\slshape}                                        %
\makeatother                                                   %
\usepackage[italian,english]{babel} 
\usepackage[utf8]{inputenc}
\usepackage[T1]{fontenc} 
\usepackage{indentfirst} 
\usepackage{amsmath,amsfonts,amssymb,amsthm}
\usepackage[all]{xy}
\usepackage{graphicx}%,float}
\usepackage{caption}
\usepackage{subcaption}
\usepackage[font={small}, labelfont=bf]{caption}[2004/07/16]
\usepackage{url}
\usepackage{layaureo}
\usepackage{color}
\usepackage{tikz}
\usepackage{mathrsfs}

\makeatletter
\def\paragraph{\@startsection{paragraph}{4}%
  \z@\z@{-\fontdimen2\font}%
  {\normalfont\bfseries}}
\makeatother

\numberwithin{equation}{section}
\newcommand*{\diff}{\mathop{}\!\mathrm{d}}

\theoremstyle{plain}
\newtheorem{teo}{Theorem}[section]
\newtheorem{prop}[teo]{Proposition}
\newtheorem{cor}[teo]{Corollary}
\newtheorem{lemma}[teo]{Lemma}

\theoremstyle{definition}
\newtheorem{defin}[teo]{Notation}
\newtheorem{example}[teo]{Example}

\theoremstyle{remark}
\newtheorem{remark}{Remark}

\newcommand{\R}{\mathbb{R}}
\newcommand{\Z}{\mathbb{Z}}
\newcommand{\N}{\mathbb{N}}
\newcommand{\T}{\mathbb{T}}

\newcommand{\bigslant}[2]{{\raisebox{.2em}{$#1$}\left/\raisebox{-.2em}{$#2$}\right.}}

\newcommand{\norma}[1]{\lVert#1\rVert}
\newcommand{\modulo}[1]{\left\lvert#1\right\rvert}
\newcommand{\orbita}[2]{\mathcal{O}_{#1}(#2)}

\newcommand{\newword}[1]{\textsl{#1}}
\newcommand{\sldz}{\SL_{d}(\Z)}

\DeclareMathOperator{\misura}{Leb}
\DeclareMathOperator{\hessiano}{Hes}

\DeclareMathOperator{\dist}{dist}

\DeclareMathOperator{\Id}{Id}
\DeclareMathOperator{\SL}{SL}
\DeclareMathOperator{\const}{const}
\DeclareMathOperator{\rank}{rank}

\begin{document}

\bibliographystyle{plain}

\sloppy

\title[Quantitative mixing for locally Hamiltonian flows]{Quantitative mixing for locally Hamiltonian flows with saddle loops on compact surfaces}

\author[D.~Ravotti]{Davide Ravotti}
\address{School of Mathematics\\
University of Bristol\\
University Walk\\
BS8 1TW Bristol, UK}
\email{davide.ravotti@bristol.ac.uk}

\begin{abstract}
Given a compact surface $\mathcal{M}$ with a smooth area form $\omega$, we consider an open and dense subset of the set of smooth closed 1-forms on $\mathcal{M}$ with isolated zeros which admit at least one saddle loop homologous to zero and we prove that almost every element in the former induces a mixing flow on each minimal component. Moreover, we provide an estimate of the speed of the decay of correlations for smooth functions with compact support on the complement of the set of singularities. This result is achieved by proving a quantitative version for the case of finitely many singularities of a theorem by Ulcigrai (ETDS, 2007), stating that any suspension flow with one asymmetric logarithmic singularity over almost every interval exchange transformation is mixing. In particular, the quantitative mixing estimate we prove applies to asymmetric logarithmic suspension flows over rotations, which were shown to be mixing by Sinai and Khanin. 
\end{abstract}


\thanks{\emph{2010 Math.~Subj.~Class.}: 37A25, 37E35}
%\subjclass{37A25, 37E35} 

\maketitle

\section{Introduction}

Let us consider a smooth compact connected orientable surface $\mathcal{M}$, together with a smooth area form $\omega$. Any smooth closed 1-form induces a smooth area-preserving flow on $\mathcal{M}$, which is given locally by the solution of some Hamiltonian equations (see \S2 for definitions); it is hence called \newword{locally Hamiltonian flow} or \newword{multi-valued Hamiltonian flow}. 

The study of such flows was initiated by Novikov \cite{novikov:hamiltonian}, motivated by some problems in solid-state physics. Orbits of locally Hamiltonian flows can be seen as hyperplane sections of periodic manifolds, as pointed out by Arnold \cite{arnold:torus}, 
who studied the case when $\mathcal{M}$ is the 2-dimensional torus $\T^2$. He proved that $\T^2$ can be decomposed into finitely many regions filled with periodic trajectories and one minimal ergodic component; in the same paper he asked whether the restriction of the flow to this ergodic component is mixing. We recall that a flow $\{\varphi_t\}_{t \in \R}$ on a measure space $(X, \mu)$ is \newword{mixing} if for any measurable sets $A, B \subset X$ we have
$$
\lim_{t \to \infty} \mu (\varphi_t(A) \cap B) = \mu(A) \mu(B),
$$
i.e., if the events $A$ and $B$ become asymptotically independent. By choosing an appropriate Poincaré section, the flow on this ergodic component is isomorphic to a suspension flow over a circle rotation with a roof function with asymmetric logarithmic singularities. The question posed by Arnold was answered by Sinai and Khanin \cite{sinai:mixing}, who proved that, under a full-measure Diophantine condition on the rotation angle, the flow is mixing. This condition was weakened by Kochergin \cite{kochergin:non, kochergin:non2, kochergin:torus, kochergin:torus2}.

The presence of singularities in the roof function is necessary, as well as the asymmetry condition: in this setting, mixing does not occur for functions of bounded variation or, assuming a full-measure Diophantine condition on the rotation angle, for functions with symmetric logarithmic singularities; see the results by Kochergin in \cite{kochergin:absence}  and  \cite{kochergin:absence2} respectively. Indeed, mixing is produced by shearing of transversal segments close to singular points, which is a result of different deceleration rates. 

Similarly, if the genus $g$ of the surface $\mathcal{M}$ is greater than $1$, any locally Hamiltonian flow can be decomposed into \newword{periodic components}, i.e.~regions filled with periodic orbits, and \newword{minimal components}, namely regions which are the closure of a nonperiodic orbit, as it was shown independently by several authors, see Levitt \cite{levitt:feuilletages}, Mayer \cite{mayer:minimal} and Zorich \cite{zorich:hyperplane}. The first return map of a Poincaré section on any of the minimal components is an Interval Exchange Transformation (IET), namely a piecewise orientation-preserving isometry of the interval $I=[0,1]$; in particular, typical (in a measure-theoretic sense) flows on minimal components are ergodic, since almost every IET is ergodic, due to a classical result proved by Masur \cite{masur:ergodic} and Veech \cite{veech:ergodic} independently. 

On the other hand, mixing depends on the type of singularities of the first return time function: Kochergin proved mixing for suspension flows over IETs with roof functions with power-like singularities \cite{kochergin:degenerate}. However, this case corresponds to degenerate zeros of the 1-form defining the locally Hamiltonian flow; the complement of the set of these 1-forms is open and dense in the set of 1-forms with isolated zeros.
Generic flows have logarithmic singularities: in this case, if the surface $\mathcal{M}$ is the closure of a single orbit, i.e.~if the flow is minimal, Ulcigrai proved that \newword{almost every} flow is not mixing \cite{ulcigrai:absence}, but weak mixing \cite{ulcigrai:weakmixing}. Here, we consider the measure class sometimes called \newword{Katok fundamental class}, described in \S2.
An example of an \newword{exceptional} minimal mixing flow in this setup has been constructed recently by Chaika and Wright \cite{chaika:example}, who exhibited a locally Hamiltonian minimal mixing flow with simple saddles on a surface of genus 5. 

In this paper we address the question of mixing when the 1-form has isolated simple zeros and the flow is not minimal; typically, minimal components are bounded by saddle loops homologous to zero (see \S2 for definitions). We prove the following result; a more precise formulation is given in Theorem \ref{th:mix}.
\begin{teo}\label{th:1.1}
There exists an open and dense subset of the set of smooth closed 1-forms on $\mathcal{M}$ with isolated zeros which admit at least one saddle loop homologous to zero such that almost every 1-form in it induces a mixing locally Hamiltonian flow on each minimal component.
\end{teo}

Moreover, we provide an estimate on the decay of correlations for a dense set of smooth functions, namely we prove the following theorem.
\begin{teo}\label{th:1.2}
Let $\{\varphi_t\}_{t \in \R}$ be the locally Hamiltonian flow induced by a smooth 1-form $\eta$ as in Theorem \ref{th:1.1} and let $\mathcal{M}\rq{} \subset \mathcal{M}$ be a minimal component. Consider the set $\mathscr{C}^1_c(\mathcal{M}\rq{})$ of $\mathscr{C}^1$ functions on $\mathcal{M}\rq{}$ with compact support in the complement of the singularities of $\eta$. Then, there exists $0 < \gamma <1$ such that for all $g,h \in \mathscr{C}^1_c(\mathcal{M}\rq{})$ with $\int_{\mathcal{M}\rq{}} g \omega = 0$ we have
$$
\modulo{ \int_{\mathcal{M}\rq{}} (g \circ \varphi_t)h\ \omega } \leq \frac{C_{g,h}}{(\log t)^{\gamma}},
$$
for some constant $C_{g,h}>0$.
\end{teo}

To the best of our knowledge, this is the first quantitative mixing result for locally Hamiltonian flows, apart from a Theorem by Fayad \cite{fayad:quantmix}, which states that a certain class of suspension flows over irrational rotations with roof function with power-like singularities have polynomial speed of mixing. In the genus 1 case, Theorem \ref{th:1.2} provides a quantitative version of the mixing result by Sinai and Khanin in \cite{sinai:mixing}.
We believe that the optimal estimate of the speed of decay has indeed this form, namely a power of $\log t$, although this remains an open question.

The proof of Theorem \ref{th:1.1} consists of two parts: first, we describe the open and dense set of 1-forms we consider (with a measure class defined on it) and we show how to represent the restriction of the induced locally Hamiltonian flows to any of its minimal component as a suspension flow over an interval exchange transformation with roof function with asymmetric logarithmic singularities. Secondly, we show that for almost every IET, every such suspension flow is mixing by proving a version of Theorem \ref{th:1.2} for suspension flows. Ulcigrai \cite{ulcigrai:mixing} treated the special case when the roof function has only one asymmetric logarithmic singularity; in this paper, we show that her techniques can be made quantitative and applied to this more general setting. The first step of the proof is to obtain sharp estimates for the Birkhoff sums of the derivative $f\rq{}$ of the roof function $f$, see Theorem \ref{th:BS}. 
These estimates are also used by Kanigowski, Kulaga and Ulcigrai to prove mixing of all orders for such flows \cite{KKU:multmixing}. In order to deduce the result on the decay of correlations, we apply a \newword{bootstrap trick} analogous to the one used by Forni and Ulcigrai in \cite{forniulcigrai:timechanges} and an estimate on the deviation of ergodic averages for typical IETs by Athreya and Forni~\cite{athreyaforni:iet}.

\subsection{Outline of the paper}
In \S2 we recall the definition of locally Hamiltonian flow induced by a smooth closed 1-form and we focus on the set of closed 1-forms with isolated zeros; we describe some of its topological properties and we equip it with Katok\rq{}s measure class. In \S3 we show how to represent the locally Hamiltonian flows we consider as suspension flows over IETs and we discuss the relation between Katok\rq{}s measure class and the measure on the set of IETs. In \S4 we recall some basic facts about the Rauzy-Veech Induction for IETs (a renormalization algorithm which corresponds to inducing the IET to a neighborhood of zero) and in doing so we introduce some notation for the proof of Theorem \ref{th:BS}; moreover, we state a full-measure Diophantine condition for IETs first used by Ulcigrai in \cite{ulcigrai:mixing} to bound the growth of the Rauzy-Veech cocycle matrices along a subsequence of induction times (see Theorem \ref{th:ulcigrai}). 
We remark that, although in general we have more than one singularity, we do not need to induce at other points by using different renormalization algorithms, but we are able to show that the Diophantine condition in \cite{ulcigrai:mixing} can be used to treat also the case of several singularities.
In \S5 we state the results on the Birkhoff sums of the roof function of the suspension flow and its derivative (Theorem \ref{th:BS}), and the quantitative estimate on the speed of the decay of correlations for a dense set of smooth functions in the language of suspension flows (Theorem \ref{th:decayofcorr}); we also deduce Theorem \ref{th:1.2} and Theorem \ref{th:IETgoal} from it.
Section 6 is devoted to the proof of Theorem \ref{th:decayofcorr}, which is carried out in several steps: we first define partitions of the unit interval analogous to the ones used by Ulcigrai in \cite{ulcigrai:mixing}, with explicit bounds on their size, and then we apply a bootstrap trick to reduce the problem to estimate the deviations of ergodic averages for IETs, for which we apply a result by Athreya and Forni~\cite{athreyaforni:iet}. 
In the Appendix 7 we prove Theorem~\ref{th:BS}. 

\subsection{Acknowledgments}
I would like to thank my supervisor Corinna Ulcigrai for her guidance and support throughout the writing of this paper. I also thank the referees for their attentive readings and helpful comments on previous versions of this paper.
The research leading to these results has received funding from the European Research Council under the European Union Seventh Framework Programme (FP/2007-2013)~/~ERC Grant Agreement n.~335989.

%%%%%%%%%%%%%%%%%%%%
%%%%%%%%%%%%%%%%%%%%
%%%%%%%%%%%%%%%%%%%%

\section{Locally Hamiltonian flows}

Let $\mathcal{M}$ be a smooth compact connected orientable surface of genus $g$ and fix a smooth area form $\omega$ on $\mathcal{M}$. For any point $p \in \mathcal{M}$ and for any choice of local coordinates supported on a neighborhood $\mathcal{U}$ of $p$, we can write $\omega = \omega\negthickspace\upharpoonright_{\mathcal{U}} = V(x,y) \diff x \wedge \diff y$, where $V(x,y)$ is a $\mathscr{C}^{\infty}$ function; moreover $\omega_p \neq 0$. 
Fix a smooth closed 1-form $\eta$ on $\mathcal{M}$; here and henceforth, we only consider 1-forms $\eta$ with isolated zeros (sometimes called singularities). Then $\eta$ determines a flow $\{\varphi_t\}_{t \in \R}$ in the following way: consider the vector field $W$ defined by the relation $W \lrcorner\ \omega = \eta$, where $\lrcorner$ denotes the contraction operator; the point $\varphi_t(p)$ is given by following for time $t$ the smooth integral curve passing through $p$. Explicitly, for any point $p$ there exists a simply connected neighborhood $\mathcal{U}$ of $p$ such that $\eta\negthickspace\upharpoonright_{\mathcal{U}} = \diff H$ for a smooth function $H(x,y)$ defined on $\mathcal{U}$. Clearly, $H$ is uniquely determined up to a constant factor. Then the relation defining $W$ translates as
$$
V(x,y)( W_x \diff y -W_y\diff x) =\partial_xH \diff x + \partial_yH \diff y, 
$$
i.e.~$W\negthickspace\upharpoonright_{\mathcal{U}}= \left((\partial_yH) \partial_x -(\partial_xH) \partial_y\right)/V$. Notice that, since $\mathcal{M}$ is compact, the flow is defined for any $t \in \R$.

The 1-form $\eta$ vanishes along any integral curve, namely denoting by $\varphi(p) \colon t \to \varphi_t(p)$ the integral curve through $p$, we have that $\eta\negthickspace\upharpoonright_{\varphi(p)} = 0$. Indeed, $\frac{\diff}{\diff t} H(\varphi_t(p)) = \nabla H \cdot \dot{\varphi}_t(p) =0$, meaning that $H$ is constant along $\varphi(p)$.  We say that $\varphi(p)$ is a \newword{leaf} of $\eta$ and $\eta$ determines a \newword{foliation} of the surface $\mathcal{M}$.

The function $H$ is globally defined on $\mathcal{M}$ if and only if the 1-form $\eta$ is exact, and, in this case, $H$ is said to be a (global) Hamiltonian of the system. In general, the relation $\eta = \diff H$ holds locally: for this reason $\{\varphi_t\}_{t \in \R}$ is called the \newword{locally Hamiltonian flow associated to} $\eta$.

Let $\pi \colon \widetilde{\mathcal{M}} \to \mathcal{M}$ be the universal cover of $\mathcal{M}$; then the pull-back $\pi^{\ast}\eta$ is a closed 1-form on $\widetilde{\mathcal{M}}$, since $\diff (\pi^{\ast} \eta) = \pi^{\ast} \diff \eta =0$. The fact that $\widetilde{\mathcal{M}}$ is simply connected implies that there exists a global Hamiltonian $\widetilde{H}$ on $\widetilde{\mathcal{M}}$ and the values of $\widetilde{H}$ at different pre-images $p_1,p_2 \in \pi^{-1}(p)$ differ by the \newword{periods}, i.e.~the values of $\widetilde{H}(p_2) - \widetilde{H}(p_1) = \int_{p_1}^{p_2} \pi^{\ast}\eta = \int_{\gamma} \eta$, where $\gamma \in \pi_1(\mathcal{M},p)$ is a loop in $\mathcal{M}$ with base point $p$ which lifts to a path connecting $p_1$ to $p_2$. Therefore, there exists a multi-valued function $H = \widetilde{H} \circ \pi^{-1}$ on $\mathcal{M}$, which is well-defined as a function
$$
H \colon \mathcal{M} \to \bigslant{\R}{ \{ \int_{\gamma} \eta : \gamma \in \pi_1(\mathcal{M})\}},
$$
being a Hamiltonian for $\eta$, since $\eta_p = (\pi^{\ast}\eta)_{\pi^{-1}(p)} \circ \diff \pi^{-1}_p = \diff (\widetilde{H} \circ \pi^{-1})_p = \diff H_p$. For this reason, the flow $\{\varphi_t\}_{t \in \R}$ is also called the \newword{multi-valued Hamiltonian flow associated to} $\eta$.

\begin{remark}\label{rk:areapreserving}
The flow $\{\varphi_t\}_{t \in \R}$ preserves both the area form $\omega$ and the 1-form $\eta$. To see this, it is sufficient to show that the correspondent Lie derivatives $\mathscr{L}_W \omega$ and $\mathscr{L}_W \eta$ w.r.t.~$W$ vanish. Indeed, since by definition $\eta = W \lrcorner\ \omega$ and $\eta$ is closed,
$$
\mathscr{L}_W \omega = W \lrcorner\ (\diff \omega) + \diff(W \lrcorner\ \omega) = \diff \eta =0,
$$
and
$$
\mathscr{L}_W \eta =  W \lrcorner\ (\diff \eta) + \diff (W \lrcorner\ \eta) = \diff (W \lrcorner\ (W\lrcorner\ \omega)) = \diff \omega( W, W) = 0,
$$
since $\omega$ is alternating. 
\end{remark}

%%%%%%%%%%

\subsection{Perturbations of closed 1-forms}

Let $\eta, \eta\rq{}$ be two smooth closed 1-forms. We say that $\eta\rq{}$ is an $\varepsilon$-perturbation of $\eta$ if for any $p \in \mathcal{M}$ and for any coordinates supported on a simply connected neighborhood $\mathcal{U}$ of $p$, we have $\eta\negthickspace\upharpoonright_{\mathcal{U}} = \diff H$ and $(\eta\rq{}-\eta)\negthickspace\upharpoonright_{\mathcal{U}} = \diff f$, with $\norma{f}_{\mathscr{C}^{\infty}} \leq \varepsilon \norma{H}_{\mathscr{C}^{\infty}}$, where $\norma{\cdot}_{\mathscr{C}^{\infty}}$ denotes the $\mathscr{C}^{\infty}$-norm. We want to study the properties of \newword{generic} 1-forms, namely the properties of 1-forms which persist under small perturbations. 

Let $p \in \mathcal{M}$ be a zero of $\eta$, and write in local coordinates $\eta = \diff H$; we say that $p$ is a \newword{simple} zero if $\det \hessiano_{(0,0)}(H) \neq 0$, where $\hessiano_{(0,0)}(H)$ denotes the Hessian matrix of $H$ at $p=(0,0)$. We remark that this condition is independent of the choice of local coordinates. A zero which is not simple is called \newword{degenerate}.

\begin{defin}
We denote by $\mathcal{F}$ the set of smooth closed 1-forms on $\mathcal{M}$ with isolated zeros and by $\mathcal{A} \subset \mathcal{F}$ the subset of 1-forms with simple zeros.
\end{defin}

Let us recall the following result by Morse, see e.g.~\cite[p.~6]{milnor:morse}. 
\begin{teo}\label{th:morse1}
Let $p \in \mathcal{M}$ be a simple zero of $\eta$. There exist local coordinates supported on a simply connected neighborhood $\mathcal{U}$ of $p=(0,0)$ such that either $\eta\negthickspace\upharpoonright_{\mathcal{U}} = x \diff x + y \diff y$, or $\eta\negthickspace\upharpoonright_{\mathcal{U}} = -x \diff x - y \diff y$, or $\eta\negthickspace\upharpoonright_{\mathcal{U}} = y \diff x +x \diff y$. 
\end{teo}
In the first case, $p$ is a local minimum for any local Hamiltonian $H$ and we say that $p$ is a minimum for $\eta$; for the same reason, in the second case we say that $p$ is a maximum for $\eta$ and in the latter case we say that $p$ is a saddle point. With the aid of these coordinates, it is easy to check that the index of the associated vector field at a maximum or minimum is $1$, whence it is $-1$ at a saddle point.
By the Poincaré-Hopf Theorem, if $\eta$ has only simple zeros, then $\# \text{minima} + \#\text{maxima} -  \#\text{saddles} = \chi(\mathcal{M})$, where $\chi(\mathcal{M}) = 2-2g$ is the Euler characteristic of $\mathcal{M}$.

If $p$ is a maximum or a minimum for $\eta$, locally the leaves of $\eta$ are closed curves homologous to zero. Hence, $p$ is the centre of a disk filled with \lq\lq parallel\rq\rq\ leaves; the maximal disk of this type, which will be called an \newword{island} for $\eta$, is bounded by a closed leaf $\gamma_0$ homologous to zero. The closed curve $\gamma_0$ must contain at least one critical point for $\eta$, which has to be a saddle if $\eta$ has only simple zeros. A leaf $\gamma_0$ as above is called a saddle leaf; namely a \newword{saddle leaf} is a leaf $\gamma = \varphi(x)$ such that $\lim_{t \to \infty} \varphi_t(x) = q_1$ and $\lim_{t \to -\infty} \varphi_t(x) = q_2$, where $q_1, q_2$ are a saddle points. If $q_1 = q_2$ we say that $\varphi(x)$ is a \newword{saddle loop}, otherwise we say that $\varphi(x)$ is a \newword{saddle connection}.

We describe some topological properties of the sets $\mathcal{A}$ and $\mathcal{F}$.
\begin{lemma}\label{lemmaadl}
Let $\mathcal{A}_{s,l}$ be the set of 1-forms in $\mathcal{A}$ with $s$ saddle points and $l$ minima or maxima. Then, each $\mathcal{A}_{s,l}$ is open and their union $\mathcal{A}$ is dense in $\mathcal{F}$. 
\end{lemma}
\begin{proof}
The last assertion is classical, see e.g.~\cite[Corollary 1.29]{pajitnov:morse}, but we present a proof for the sake of completeness. 
We first show that $\mathcal{A}$ is open. By contradiction, suppose that there exists a sequence of 1-forms $(\eta_n)$ converging to $\eta \in \mathcal{A}$ such that each $\eta_n$ admits a degenerate zero $p_n$. Since $\mathcal{M}$ is compact, we can assume $p_n \to p$ for some $p \in \mathcal{M}$. Let $\mathcal{U}$ be a simply connected neighborhood of $p$ and consider a sequence of local Hamiltonians $H_n$ for $\eta_n$ on $\mathcal{U}$ which converges in the $\mathscr{C}^{\infty}$-norm to a local Hamiltonian $H$ for $\eta$. Therefore, $0 = \det \hessiano_{p_n}(H_n) \to \det \hessiano_p(H) \neq 0$, which is the desired contradiction.

We now show that the sets $\mathcal{A}_{s,l}$ are open. Consider $\eta \in \mathcal{A}_{s,l}$ with zeros $p_1, \dots, p_{s+l}$. Any sufficiently small perturbation $\eta\rq{}$ of $\eta$ has only simple zeros $p_1\rq{}, \dots, p_{s+l}\rq{}$ with $p_i\rq{}$ close to $p_i$. The type of the zero $p_i\rq{}$ depends on the sign of the trace and of the determinant of the Hessian matrix of a local Hamiltonian at $p_i\rq{}$, which are continuous maps in the $\mathscr{C}^{\infty}$-topology; hence the type of zero of $p_i$ and $p_i\rq{}$ is the same. Thus, each $\mathcal{A}_{s,l}$ is open. 

To prove $\mathcal{A}$ is dense, we show that for all degenerate zeros $p$ of $\eta \in \mathcal{F}$, there exist arbitrarily small perturbations $\eta\rq{}$ which coincide with $\eta$ outside a neighborhood $\mathcal{U}$ of $p$ and have only simple zeros in $\mathcal{U}$. Let $p$ be a degenerate zero of $\eta$ and fix an open simply connected neighborhood $\mathcal{U}$ of $p$.  
Sard\rq{}s Theorem applied to $\eta \colon \mathcal{M} \to T^{\ast}\mathcal{M}$ implies that there exist regular values $\eta_q \in T^{\ast}_q\mathcal{M}$, with $q $ arbitrarily close to $p$. Fix a regular value $\eta_q$ and let $\mathcal{V}$ be a simply connected neighborhood of $p$ containing $q$ compactly contained in $\mathcal{U}$. Any choice of local coordinates on $\mathcal{U}$ gives a trivialization $T^{\ast}\mathcal{M}\negthickspace\upharpoonright_{\mathcal{U}} = \mathcal{U} \times \R^2$, which we implicitly use to extend $\eta_q$ to a constant 1-form on $\mathcal{U}$. Finally, consider a \lq\lq{}bump\rq\rq{}\ function $f \colon \mathcal{M} \to \R$ whose support is contained in $\mathcal{U}$ and such that $f\negthickspace\upharpoonright_{\mathcal{V}} = 1$; the 1-form $\eta\rq{} = \eta - f\eta_q$ satisfies the claim.
\end{proof}

As we just saw in Lemma \ref{lemmaadl}, the number and type of zeros of a 1-form $\eta \in \mathcal{A}$ are invariant under small perturbations; the following lemma ensures that certain closed leaves are stable as well. Let us recall that a loop is homologous to zero in $\mathcal{M}$ if and only if it disconnects the surface.

\begin{lemma}\label{lemma:homtozero}
If a saddle loop $\gamma$ is homologous to zero, then it is stable under small perturbations.
\end{lemma}
\begin{proof}
Let $\gamma$ be a saddle loop homologous to zero passing through a saddle $p$ of $\eta$ and let $\eta\rq{}$ be a $\varepsilon$-perturbation of $\eta$. We consider the connected component $\mathcal{M}\rq{}$ of $\mathcal{M}$ not containing leaves passing through $p$: leaves close to $\gamma$ are homotopic one to the other, hence we have a cylinder (or an island, if $\mathcal{M}\rq{}$ contains only a maximum or minimum for $\eta$) filled with closed \lq\lq{}parallel\rq\rq{}\ leaves, each of which is homologous to zero. 
On this cylinder, the integrals of $\eta$ and $\eta\rq{}$ along any closed curve are zero; thus they admit Hamiltonians $H$ and $H + f$. If $\varepsilon$ is sufficiently small, the level sets for $H+ f$ are again closed curves, hence the cylinder of closed leaves survives under small perturbations.
\end{proof}

In general, saddle connections and saddle loops non-homologous to zero disappear under arbitrarily small perturbations, as shown by the following Example \ref{example1} and \ref{example2} respectively.
\begin{example}\label{example1}
Consider the function $H(x,y) = y(x^2+y^2-1)$ and the standard area form $\omega = \diff x \wedge \diff y$ defined on $\R^2$. There are four critical points for $\diff H$: the saddles $(\pm 1,0)$, the minimum $(0, \sqrt{3}/3)$ and the maximum $(0, -\sqrt{3}/3)$; moreover there is a saddle connection supported on the interval $(-1,1)$. Using bump functions, define a function $f$ equal to $(\varepsilon/4)(1-(x+ 1)^2+y^2)$ if $(x,y)$ is $\varepsilon$-close to $(-1,0)$, and $0$ if the distance between $(x,y)$ and $(-1,0)$ is greater than $2\varepsilon$. Then it is possible to see that the perturbed 1-form $\diff(H + f)$ admits no saddle connections, see Figures 1(a) and 1(b). 
\end{example}

\begin{figure}[h]
\centering
\begin{subfigure}[b]{0.85\textwidth}
                \includegraphics[width=\textwidth]{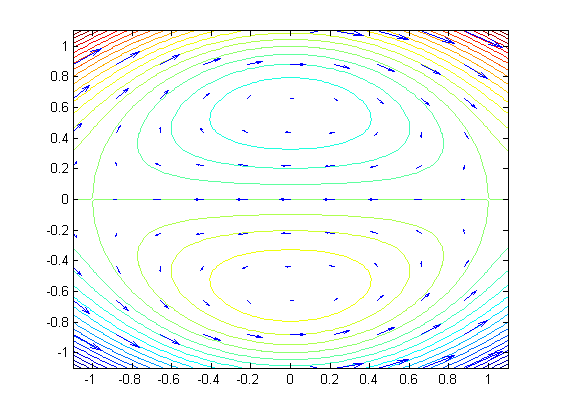}
                \caption{Orbits of the flow given by the Hamiltonian $H(x,y) = y(x^2+y^2-1)$.}
                \label{fig:Example1}
\end{subfigure}

\begin{subfigure}[b]{0.85\textwidth}
                \includegraphics[width=\textwidth]{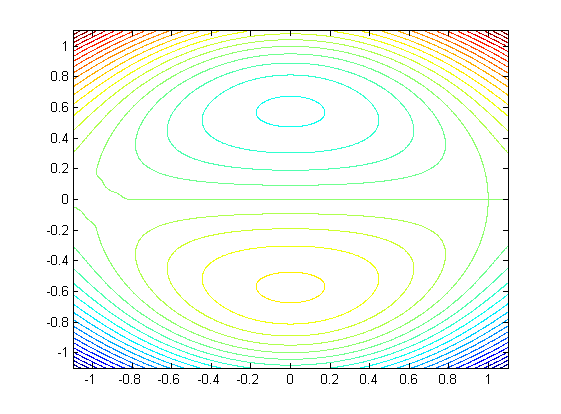}
                \caption{Orbits of the flow given by the perturbed Hamiltonian $H+f$.}
                \label{fig:Example2}
\end{subfigure}
\end{figure}

The following  example uses the dichotomy for the orbits of a linear flow on the torus. 

\begin{example}\label{example2}
Consider the torus $\mathbb{T}^2=\R^2 / \Z^2$ and construct $\eta \in \mathcal{A}_{1,1}$ in the following way. Fix $0<\delta< \frac{1}{8}$ and let $\eta$ be defined in the strip $(2\delta, 1-2\delta) \times (\frac{1}{2} - \delta, \frac{1}{2} +\delta)$ as $(x-\frac{1}{2})(x-\frac{1+\delta}{2}) \diff x + (y-\frac{1}{2}) \diff y$ and outside $(\delta, 1-\delta) \times (\frac{1}{2} - 2 \delta, \frac{1}{2} + 2 \delta)$ as $\diff x$; using a symmetric bump function it is possible to do so in such a way that every orbit is periodic. The 1-form $\eta$ has a minimum in $(\frac{1+\delta}{2},\frac{1}{2})$ and a saddle in $(\frac{1}{2},\frac{1}{2})$, hence a saddle loop not homologous to zero. 
Take a bump function $\varepsilon f(x,y) =\varepsilon f(y)$ depending on $y$ only such that $\varepsilon f(y) = \varepsilon$ for every $y \in [- \delta, \delta] \text{ mod }\Z$ and equal to 0 outside $[-2\delta, 2\delta]  \text{ mod }\Z$.
The perturbed form $\eta + \varepsilon f(y) \diff y$ coincide with $\eta$ in $[0,1)  \times (\frac{1}{2} - 2 \delta, \frac{1}{2} + 2 \delta)$, in which leaves enter vertically. Outside that region, the vector field defining the flow is $ \varepsilon f(y) \partial_x - \partial_y$, thus the displacement of any leaf in the $x$-coordinate after winding once around the torus is given by $\int_{\mathbb{T}^2} \varepsilon f$. Hence, for any $\varepsilon$ such that the previous integral is a rational number, the saddle loop is preserved; otherwise, if $\int \varepsilon f$ is irrational, the saddle loop vanish.
\end{example}
The previous example shows that neither the set of 1-forms in $\mathcal{A}$ with saddle loops non-homologous to zero nor its complement is an open set, and similarly if we consider saddle connections. However both these cases are \newword{exceptional}, as we are going to describe in the next subsection.

%%%%%%%%%%

\subsection{Measure class}\label{sectionmeasureclass}

We want to define a measure class (namely, a notion of null sets and full measure sets) on each open set $\mathcal{A}_{s,l}$; later it will be restricted to an open and dense subset. Let $\Sigma = \Sigma(\eta)$ be the finite set of singular points of a given $\eta \in \mathcal{A}_{s,l}$ and fix a basis $\gamma_1, \dots, \gamma_m$ of the first relative homology group $H_1(\mathcal{M}, \Sigma, \R)$; here $m = 2g+l+s-1$. If $\eta\rq{}$ is a perturbation of $\eta$, we can identify $H_1(\mathcal{M}, \Sigma(\eta), \R)$ with $H_1(\mathcal{M}, \Sigma(\eta\rq{}), \R)$ via the \newword{Gauss-Manin connection}, i.e.~via the identification of the lattices $H_1(\mathcal{M}, \Sigma(\eta), \Z)$ and $H_1(\mathcal{M}, \Sigma(\eta\rq{}), \Z)$. 
Define the \newword{period coordinates} of $\eta$ as
$$
\Theta (\eta) = \left( \int_{\gamma_1} \eta, \dots, \int_{\gamma_m} \eta \right)\in \R^m.
$$
The map $\Theta$ is well-defined in a neighborhood of $\eta$. Moreover, the next proposition, which is a variation of Moser\rq{}s Homotopy Trick \cite{moser:trick}, shows it is a complete invariant for isotopy classes (recall that an \newword{isotopy} between $\eta$ and $\eta\rq{}$ is a family of smooth maps $\{\psi_t \colon \mathcal{M} \to \mathcal{M} \}_{t \in [0,1]}$ such that $\psi_1^{\ast}(\eta\rq{}) = \eta$).

\begin{prop}
Let $\eta \in \mathcal{A}_{s,l}$ be fixed. There exists a neighborhood $\mathcal{U}$ of $\eta$ such that for all $\eta\rq{} \in \mathcal{U}$ there is an isotopy  $\{\psi_t \}_{t \in [0,1]}$ between $\eta$ and $\eta\rq{}$ if and only if $\Theta(\eta)=\Theta(\eta\rq{})$.
\end{prop}
\begin{proof}
If $\eta$ and $\eta\rq{}$ are isotopic, then for any element $\gamma_j$ of the basis of $H_1(\mathcal{M}, \Sigma(\eta), \Z)$ we have
$$
\int_{\gamma_j}\eta = \int_{\gamma_j} \psi_1^{\ast} \eta\rq{} = \int_{\psi_1 \circ \gamma_j} \eta\rq{},
$$
hence the claim.

Conversely, let $\eta\rq{}$ be a small perturbation of $\eta$ and suppose that they have the same period coordinates. Up to an isotopy, we can assume that $\Sigma(\eta) = \Sigma(\eta\rq{})$.

Consider the convex combinations $\eta_t = (1-t)\eta + t \eta\rq{}$ for $t \in [0,1]$. To construct $\{\psi_t\}$ such that $ \psi_t^{\ast}(\eta_t) =\eta_0 = \eta$, we look for a smooth non-autonomous vector field $\{ X_t \}$ such that $\psi_t$ is the flow induced by $\{ X_t \}$. It is enough for $\{ X_t \}$ to satisfy
\begin{equation}
0= \frac{\diff}{\diff t} \psi_t^{\ast}(\eta_t) = \psi^{\ast}_t \left(  \frac{\diff}{\diff t} \eta_t + \mathscr{L}_{X_t}\eta_t \right). \label{eq:Lie1}
\end{equation}
The previous equation holds if $\frac{\diff}{\diff t} \eta_t + \mathscr{L}_{X_t}\eta_t =0$. Notice that $\frac{\diff}{\diff t} \eta_t = \eta\rq{} - \eta$, which, by hypothesis, is cohomologous to zero, since the integral over any closed loop on $\mathcal{M}$ is zero. Hence, there exists a global function $U$ over $\mathcal{M}$ such that $\frac{\diff}{\diff t} \eta_t = \diff U$ and then we can rewrite \eqref{eq:Lie1} as $\diff (U + X_t \lrcorner\ \eta_t) =0$. If $W_t$ denotes the vector field associated to $\eta_t$, i.e.~$W_t \lrcorner\ \omega = \eta_t$, the equation to be solved becomes $-U=  X_t \lrcorner\ \eta_t = \omega(W_t, X_t)$. 
 
 On the set $\Sigma$ of critical points, the vector field $W_t$ vanishes; thus a necessary condition for the existence of a solution is that $U(p)=0$ for any $p \in \Sigma$. It is possible to choose $U$ satisfying this condition: $U$ is defined up to a constant and if $p,q \in \Sigma$, then $U(p) = U(q)$ because 
 $$
 U(p)-U(q) = \int_q^p \diff U = \int_q^p \eta - \int_q^p \eta\rq{} = 0.
 $$
In a neighborhood of any point $q \in \mathcal{M} \setminus \Sigma$, we have $(W_t)_q \neq 0$ since we assumed $\Sigma(\eta) = \Sigma(\eta\rq{})$; by the nondegeneracy of $\omega$, a solution $X_t$ exists. This concludes the proof.
\end{proof}

Notice that if $\gamma$ is a leaf for $\eta$, then $\psi_1  \circ \gamma$ is a leaf for $\eta\rq{}$, since $\eta\rq{}\negthickspace\upharpoonright_{\psi_1 \circ \gamma} = \eta\rq{}((\psi_1)_{\ast}(\dot{\gamma})) = (\psi_1^{\ast}\eta\rq{})(\dot{\gamma}) = \eta\negthickspace\upharpoonright_{\gamma}=0$. 
Therefore, $\psi_1$ realises an orbit equivalence between the locally Hamiltonian flows induced by $\eta$ and $\eta\rq{}$, which is $\mathscr{C}^{\infty}$ away from the critical set.  

\begin{defin}\label{definmeasureclass}
We equip $\mathcal{A}_{s,l}$ with the measure class $\Theta^{\ast}( \misura_{\R^m})$ given by the pull-back of the Lebesgue measure $\misura_{\R^m}$ on $\R^m$ via $\Theta$.
\end{defin}
We want to study the dynamics induced by \newword{typical} 1-forms with respect to this measure class.
We remark that if $\eta$ has a saddle loop non-homologous to zero or a saddle connection, then, up to a change of basis of $H_1(\mathcal{M}, \Sigma(\eta), \R)$, one of the coordinate of $\eta$ is zero, in particular the set of such 1-forms is a null set.

Let us remark that if the locally Hamiltonian flow is minimal, then $l=0$ and $-s = \chi(\mathcal{M})$; in this case, as recalled in the introduction, Ulcigrai in \cite{ulcigrai:absence} and \cite{ulcigrai:weakmixing} proved that \newword{almost every} $\eta$ induces a non-mixing but weakly mixing flow.

%%%%%%%%%%

\section{Suspension flows over IETs}

In this section, we are going to represent the restriction of a locally Hamiltonian flow $\{\varphi_t\}_{t \in \R}$ to a minimal component as a suspension flow over an interval exchange transformation. We recall all the relevant definitions for the reader\rq{}s convenience.

An \newword{Interval Exchange Transformation} $T$ of $d$ intervals (IET for short) is an orientation-preserving piecewise isometry of the unit interval $I = [0,1]$; namely it is the datum of a permutation $\pi$ of $d$ elements and a vector $\underline{\lambda} = (\lambda_i)$ in the standard $d$-simplex $\Delta_{d}$: the interval $I$ is partitioned into the subintervals $I_j=I_j^{(0)} = [a_{j-1}, a_{j})$ of length $\lambda_j$ and the subintervals $I_j^{(0)}$ after applying $T$ are ordered according to the permutation $\pi$. Formally, let $a_j = \sum_{k\leq j} \lambda_k$ and $a_j\rq{} = \sum_{k \leq \pi(j)} \lambda_{\pi^{-1}(k)}$ and define $T(x) = x - a_{j-1} + a_{j-1}\rq{}$ for $x \in [a_{j-1}, a_{j-1} + \lambda_i)$. 
We refer to \cite{viana:iet2} or \cite{viana:iet} for a background on IETs. 

Given a strictly positive function $f \in L^1([0,1])$, a \newword{suspension flow over an IET with roof function $f$} is defined in the following way.
Consider the quotient space
\begin{equation}\label{eq:suspspace}
\mathcal{X}:=  \bigslant{ \{ (x,y) \in [0,1] \times \R: 0 \leq y \leq f(x)\} }{\sim},
\end{equation}
where $\sim$ denotes the equivalence relation generated by the pairs $\{ (x,f(x)),$ $ (T(x),0)\}$. We define the \newword{suspension flow} $\{\phi_t\}_{t \in \R}$ over $([0,1], T, \diff x)$ with roof function $f$ to be the flow on $\mathcal{X}$ given by $\phi_t (x,y) = (x, y+t)$ for $-y \leq t \leq f(x)-y$, and then extended to all times $t \in \R$ via the identification $\sim$. Intuitively, a point $(x,y) \in \mathcal{X}$ under the action of the flow moves vertically with unit speed up to the point $(x, f(x))$, which is identified with $(T(x),0)$; after this \lq\lq jump\rq\rq, it continues in the same way. 

The flow $\{\phi_t\}_{t \in \R}$ can be described explicitly. For any function $g \colon I \to \R$ and for $r \geq 0$, denote by $S_r(g)(x)$ the $r$-th Birkhoff sum of $g$ along the orbit of $x \in I$, i.e.
$$
S_r(g)(x) := \sum_{i=0}^{r-1} g(T^i x);
$$
then, for $t\geq 0$,
\begin{equation}\label{eq:sfj}
\phi_t(x,0) = \left(T^{r(x,t)}x, t-S_{r(x,t)}(f)(x) \right),
\end{equation}
where $r(x,t)$ denotes the maximum $r \geq 0$ such that $S_r(f)(x) \leq t$. 

The set of suspension flows we are going to consider consists of the ones for which the roof function $f$ has \newword{asymmetric logarithmic singularities}, namely it satisfies the following properties: 
\begin{itemize}
\item[(a)] $f$ is not defined on the $d-1$ points $a_1, a_2, \dots, a_{d-1} \in (0,1)$;
\item[(b)] $f \in \mathscr{C}^{\infty} \left( [0,1] \setminus \bigcup_{i=1}^{d-1} \{a_i\} \right)$;
\item[(c)] there exists $\min f(x)>0$, where the minimum is taken over the domain of definition of $f$;
\item[(d)] for each $j = 1, \dots, d-1$ there exist positive constants $C_j^{+}, C_j^{-}$ and a neighborhood $\mathcal{U}_j$ of $a_j$ such that
\begin{equation*}
\begin{split}
&f(x) = C_j^{+} \modulo{\log (x-a_j)} + e(x), \qquad \text{for } x \in \mathcal{U}_j, x>a_j,\\
&f(x) = C_j^{-} \modulo{\log (a_j-x)} +\widetilde{e}(x), \qquad \text{for } x \in \mathcal{U}_j, x< a_j;
\end{split}
\end{equation*}
where $e, \widetilde{e}$ are smooth bounded functions on $[0,1]$.
Moreover, $C^{+} \neq C^{-}$, where $C^{+} := \sum_j C_j^{+}$ and $C^{-}:= \sum_j C_j^{-}$.
\end{itemize}

Our main result is the following; it was proved by Ulcigrai \cite{ulcigrai:mixing} in the case the roof function $f$ has one asymmetric logarithmic singularity at the origin. In this paper, we generalize her techniques to the case of finitely many singularities.

\begin{teo}\label{th:IETgoal}
For almost every IET $T$ and for any $f$ with asymmetric logarithmic singularities, the suspension flow $\{\phi_t\}_{t \in \R}$ over $([0,1], T, \diff x)$ with roof function $f$ is mixing.
\end{teo}
The asymmetry condition in (d) is the key property to produce mixing. From this result, we deduce mixing for typical locally Hamiltonian flows with asymmetric saddle loops, namely the following result.
\begin{teo}\label{th:mix}
There exists an open and dense set $\mathcal{A}_{s,l}\rq{} \subset \mathcal{A}_{s,l}$ of smooth 1-forms with $s$ saddle points and $l$ minima or maxima such that for almost every $\eta \in \mathcal{A}_{s,l}\rq{}$ with at least one saddle loop homologous to zero and for any minimal component $\mathcal{M}\rq{} \subset \mathcal{M}$, the restriction of the induced flow $\{\varphi_t\}_{t \in \R}$ to $\mathcal{M}\rq{}$ is mixing.
\end{teo}
The sets $\mathcal{A}_{s,l}\rq{}$ are the subsets of $\mathcal{A}_{s,l}$ for which the asymmetry condition in (d) is satisfied; we are going to construct them explicitly in the next subsection. Theorem \ref{th:mix} follows from Theorem \ref{th:IETgoal} by constructing an appropriate Poincaré section, showing that the first return map is an IET and, if the locally Hamiltonian flow is induced by a 1-form in $\mathcal{A}_{s,l}\rq{}$, then the first return time function $f$ has asymmetric logarithmic singularities.

\subsection{Proof of Theorem \ref{th:mix}}

Let $\eta \in \mathcal{A}_{s,l}$; as we remarked in \S\ref{sectionmeasureclass}, 1-forms with saddle connections are a zero measure set, therefore we can assume $\eta$ has no saddle connections.
Let $\mathcal{M}_1, \dots \mathcal{M}_k$ be the minimal components and let $\mathcal{M}_{k+1}, \dots, \mathcal{M}_{k+l}$ the islands, i.e.~the periodic components containing a minimum or a maximum of $\eta$ (in addition there can be cylinders of periodic orbits, but we do not label them). Each $\mathcal{M}_i$ is bounded by saddle loops homologous to zero. Denote by $p_{1,i}, \dots, p_{s_i,i}$ the singularities of $\eta$ contained in the closure of $\mathcal{M}_i$, which are saddles, and let $\{ q_1, \dots q_l \} $, with $q_i \in \mathcal{M}_{k+i}$, be the set of maxima or minima of $\eta$, which is possibly empty if $l=0$.

\medskip

\paragraph{Step 1: Poincaré section.}
Let us consider one of the minimal components $\mathcal{M}_i$. We first show that we can find a Poincaré section $I$ so that the first return map $T \colon I \to I$ is an IET of $d_i$ intervals, where
\begin{equation}\label{eq:di}
\left(\sum_{i=1}^k d_i \right) + l + (k-1) = 2g+(l+s) -1 = \rank H_1(\mathcal{M}, \Sigma, \Z).
\end{equation}
Fix a segment $I\rq{} \subset \mathcal{M}_i$ transverse to the flow 
containing no critical points and whose endpoints $a$ and $b$ lie on outgoing saddle leaves. Let $a_1, \dots, a_{d_i-1} \in I\rq{}$ be the the pull-backs of the saddle points via the flow, namely the points $a_j \in I\rq{}$ are such that $ \lim_{t \to \infty} \varphi_t(a_j) = p_{r,i}$ for some $r=1, \dots, s_i$ and $\varphi_t(a_j) \notin I\rq{}$ for any $t > 0$, see Figure \ref{fig:1}. Up to relabelling, we can suppose that the points are labelled in consecutive order, namely the segment $[a, a_j]\subset I\rq{}$ with endpoints $a$ and $a_j$ is contained in $[a,a_{j+1}]$ for all $j = 1, \dots, d_i-2$. Let $a_0$ be the closest point to $a_1$ contained in $[a,a_1]$ which lies in an outgoing saddle leaf and similarly let $a_{d_i}$ be the closest point to $a_{d_1-1}$ contained in $[a_{d_i-1},b]$ which lies in an outgoing saddle leaf.
We consider the segment $I = [a_0, a_{d_i}]$, see Figure \ref{fig:1}. 

\begin{figure}[h!]
\centering
\begin{subfigure}[b]{0.8\textwidth}
 \includegraphics[width=\textwidth]{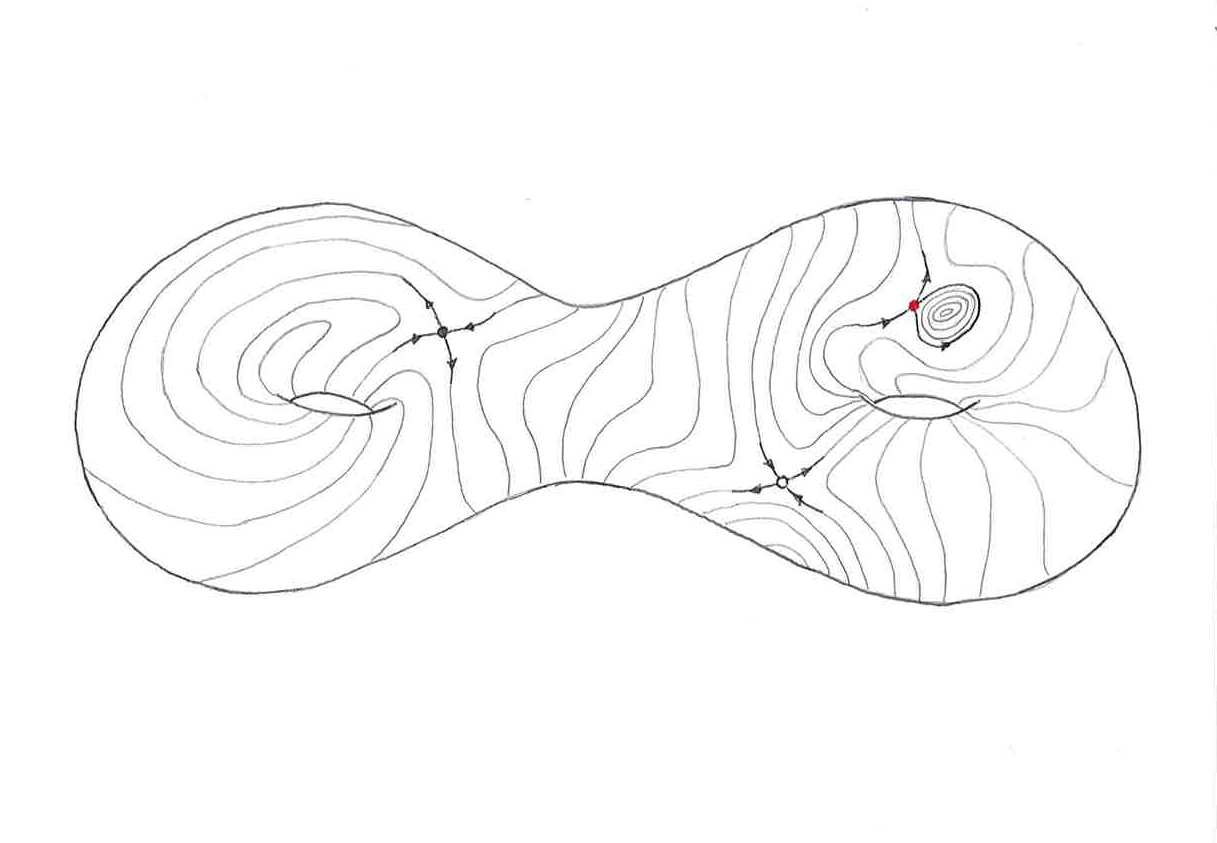}
\end{subfigure}

\begin{subfigure}[b]{0.6\textwidth}
\begin{tikzpicture}[scale=3]
\foreach \Point in {(0.3,-0.5),(0.9,0.5),(2.2,0.6)}{ %, (-1,1), (-2,3), (-1,2.5), (1,3)}{
    \node at \Point {\textbullet};
}
\foreach \Point in {(2.4,-0.7),(0.6,0.7),(1.45,0.8)}{ %, (-1,1), (-2,3), (-1,2.5), (1,3)}{
    \node at \Point {$\circ$};
}
\clip (-0.05,-0.7) rectangle (2.85,1.2);
\draw (0,0) -- (2.8,0);
\draw [->] (0.3,-0.5) -- (0.3,-0.25);
\draw [-] (0.3,-0.25) -- (0.3,0);
\draw [->] (2.4,-0.7) -- (2.4,-0.35);
\draw [-] (2.4,-0.35) -- (2.4,0);
\draw [->] (0.6,0) -- (0.6,0.35);
\draw [-] (0.6,0.35) -- (0.6,0.7);
\draw [->] (0.9,0) -- (0.9,0.25);
\draw [-] (0.9,0.25) -- (0.9,0.5);
\draw [->] (1.45,0) -- (1.45,0.4);
\draw [-] (1.45,0.4) -- (1.45,0.8);
\draw [->] (2.2,0) -- (2.2,0.3);
\draw [-] (2.2,0.3) -- (2.2,0.6);
\draw [->] (1.75,0) -- (1.75,0.32);
\draw [-] (1.75,0.32) -- (1.75,0.65);
\node [red] at (1.75,0.65) {\textbullet};
\draw [->] (1.75,0.65) to [bend right = 100] (2,0.65); 
\draw [-] (2,0.65) to [bend right = 100] (1.75,0.66);
\draw [blue] [->] (1.15,0) -- (1.15,0.75);
\draw [blue] [-] (1.15,0.75) -- (1.15,1.25);
\draw [blue] [->] (2,-0.5) -- (2,-0.3);
\draw [blue] [-] (2,-0.3) -- (2,0);
\draw [blue] [-] (2,0) -- (1.15,0);
\draw [green] [-] (0.9,0.5) -- (1.45,0.8);
\draw (0,0) node [anchor=south] {$a$} ;
\draw (0.3,0) node [anchor=south] {$a_0$} ;
\draw (0.6,0) node [anchor=north] {$a_1$} ;
\draw (1.45,0) node [anchor=north] {$\cdots$ $a_j$ $\cdots$} ;
\draw (2.4,0) node [anchor=south] {$a_{d_i}$} ;
\draw (2.8,0) node [anchor=south] {$b$} ;
\draw (2,-0.3) node [anchor=east] {$\gamma_j$} ;
\draw (1.3,0.7) node [anchor=south] {$\sigma_j$} ;
\end{tikzpicture}
\end{subfigure}%
\caption{Example of the construction of the Poincaré section; in blue one of the curves $\gamma_j$ and in green its dual $\sigma_j$.}
\label{fig:1}
\end{figure}

Let $T \colon I \to I$ be the first return map of $\varphi_t$ to $I$ and $f \colon I \to \R_{>0}$ the first return time function. Clearly, $T$ is not defined on $\{a_1, \dots, a_{d_i-1}\}$, since the return time of those points is infinite. 
Consider the connected component $I_j $ of $I \setminus \{a_1, \dots, a_{d_i-1}\}$ bounded by $a_{j-1}$ and $  a_{j}$. 
For any $z \in I_j$ and for any $0 \leq t \leq f(z)$, by compactness, the point $\varphi_t(z)$ is bounded away from the singularities, thus the map $\varphi_t$ is continuous at $z$. In particular, $T$ is continuous at any $z \in I_j$ and $T(I_j)$ is a connected segment in $I$. 
Since $I$ is transverse to the flow, we have that $\int_I \eta \neq 0$; up to reversing the orientation we can assume that $\int_I \eta >0$. Moreover, since there are no critical points of $\eta$ in the interior of $I$, the integral of $\eta$ is an increasing function, i.e.~$\int_{a_0}^{z_1} \eta < \int_{a_0}^{z_2} \eta$ whenever the segment $[a_0, z_1]$ is strictly contained in $[a_0, z_2]$. 
The 1-form $\eta$ defines a measure on $I$, which it is easy to see it is $T$-invariant. 
By considering the coordinates on $I$ given by $z \mapsto \int_{a_0}^z \eta /(\int_I \eta)$, we can identify $I=[0,1]$ and $\eta\negthickspace\upharpoonright_{I}$ with the Lebesgue measure $\misura$ on $I$. The map $T\negthickspace\upharpoonright_{I_j}$ is an isometry for any $j=1, \dots, d_i$; thus $T$ is an IET of $d_i$ intervals. 

Let us prove \eqref{eq:di}. By construction, $d_i-1$ is the number of pull-backs of the saddle points: each saddle with a saddle loop homologous to zero admits one pull-back, whence the other saddles have two. Each of the former is uniquely paired with a minimum or a maximum or with another minimal component via a cylinder of periodic orbits, hence there are exactly $l+2(k-1)$ of them. We deduce $\sum_{i=1}^k (d_i-1) + l +2k-2 = 2s$; therefore $(\sum_i d_i) + l + (k-1) = 2s+1 = 2g+(s+l)-1 = \rank H_1(\mathcal{M}, \Sigma, \Z)$ by Poincaré-Hopf formula.

\medskip

\paragraph{Step 2: return time function.}
We now investigate the first return time function $f $. Clearly, $f$ is smooth in $I \setminus \{a_1, \dots, a_{d_i-1}\}$ and blows to infinity at the points $a_j$. Since $f \neq 0$ on $I$ by hypothesis, it admits a minimum $\min f(x) >0$. 
In order to understand the type of singularities of $f$, we have to compute the time spent by an orbit travelling close to a saddle point $p$. By Theorem \ref{th:morse1}, we can suppose that a local Hamiltonian at $p = (0,0)$ is $H(x,y) = xy$ and the area form $\omega = V(x,y) \diff x \wedge \diff y$. Let $(x(t),y(t))$ be an orbit of the flow; as we have already remarked, $H$ is constant along it, $H(x(t),y(t)) = c$. The vector field is given by $W = \frac{x}{V(x,y)} \partial_x - \frac{y}{V(x,y)} \partial_y$, so that the time spent for travelling from a point $(z, c/z)$ to $(c/z,z)$ is 
$$
T = \int_0^T \diff t = \int_0^T \frac{V(x,c/x) \dot{x}}{x} \diff t = \int_a^{c/a} \frac{V(x,c/x)}{x} \diff x. 
$$
Lemma A.1 in \cite{fraczek:infinite} yields that $T = -V(0,0) \log c + e(c,a)$, where $e$ is a smooth function of bounded variation. Therefore, when the \lq\lq{}energy level\rq\rq{} $c$ approaches $0$, or equivalently when the leaf gets close to the saddle leaf, the time spent close to $p$ blows up as $\modulo{\log c}$.
Denote by $C_1, \dots, C_{s_i}$ the constants given by $ T(c)/\modulo{\log c}$ as $c \to 0$ for all the saddle points $p_{1,i}, \dots, p_{s_i,i}$.
Suppose that $a_j$ corresponds to a saddle $p_{r,i}$ belonging to a saddle loop homologous to zero. Since there are no saddle connections, there exists a small neighborhood $\mathcal{U} \subset I$ of $a_j$ which contains points that do not come close to any other singularity of $\eta$ before coming back to $I$. 
Because of the saddle loop, the logarithmic singularity of $f$ at $a_j$ has different constants: points in $I \cap \mathcal{U}$ on different sides of $a_j$ travel either once or twice near $p_{r,i}$. Namely, for some smooth bounded functions $e,\widetilde{e}$ we either have
\begin{equation*}
\begin{split}
f(x) = -C_j  \log\modulo{x-a_j} + e(x), \quad &\text{ for $x \in I \cap \mathcal{U}, x > a_j$} \\
 f(x) =- 2C_j \log \modulo{a_j-x} + \widetilde{e}(x),  \quad &\text{ for $x \in I \cap \mathcal{U}, x < a_j$},
\end{split}
\end{equation*}
or similar equalities with the conditions $x>a_j$ and $x<a_j$ reversed. On the other hand, if the point $a_j$ corresponds to a singularity $p_{r,i}$ with no saddle loop, then the constants on different sides of $a_j$ are the same. 
We remark that this phenomenon was discovered by Arnold \cite{arnold:torus} in the genus one case and exploited by Sinai and Khanin \cite{sinai:mixing} to prove mixing.

\medskip

\paragraph{Step 3: asymmetry.}
For property (d) to hold, the sum of the constants on the left side of the singularities has to be different from the one on the right.
\begin{defin}
Let $\mathcal{A}_{s,l}\rq{}$ be the subset of $\mathcal{A}_{s,l}$ of smooth 1-forms such that no linear combination of the $C_j$ with coefficients in $\{-1,0,1\}$ equals zero. 
 \end{defin}
In particular, for all $\eta \in\mathcal{A}_{s,l}\rq{}$, we have that $C^{+} \neq C^{-}$. Let us show that it is an open and dense set. Let $p=p_{j,i}$ be a singularity of $\eta$. For any small perturbation of $\eta$, there exists a change of coordinates $\psi$ close to the identity such that we can write the Hamiltonian for the perturbed 1-form as $H\rq{}=x\rq{}y\rq{}$. Thus the return time is $T(c) = -V(0,0) |\det J(\psi)_p| \log c + \widetilde{e}$, where $J(\psi)_p$ is the Jacobian matrix of $\psi$ at $p$ and $\widetilde{e}$ is another smooth function of bounded variation. If $\eta \notin \mathcal{A}_{s,l}\rq{}$, fix a saddle $p$ and for any $\varepsilon>0$ consider the perturbed local Hamiltonian $H\rq{}=(1-\varepsilon^2)xy$  at $p$; then $\psi(x,y) = ((1-\varepsilon)x, (1+\varepsilon)y)$ so that $|\det J(\psi)_p| = 1-\varepsilon^2$. Since the other constants $C_j$ are the same, it is possible to choose arbitrarily small $\varepsilon$ such that $\eta\rq{} \in \mathcal{A}_{s,l}\rq{}$, which is hence dense. In order to see that $\mathcal{A}_{s,l}\rq{}$ is open, let $xy+f(x,y)$ be the perturbed Hamiltonian at a singularity, with $\norma{f}_{\mathscr{C}^{\infty}} <\varepsilon$ and let $(x\rq{},y\rq{}) = \psi(x,y) = (\psi_1(x,y),\psi_2(x,y))$ the associated change of coordinates as above. Then, $f(x,y) = \psi_1(x,y)\psi_2(x,y) - xy = P \circ(\Id - \psi)(x,y)$, where $P$ denotes the product $P(x,y)=xy$.
Thus, there exists $\varepsilon\rq{}>0$ such that $\norma{\Id - \psi}_{\mathscr{C}^{\infty}} < \varepsilon\rq{}$ on a neighborhood of $p$; hence $|\det J(\psi)_p| \in [1-\varepsilon\rq{}, 1+\varepsilon\rq{}]$. Since this holds for any singularity $p$, the set $\mathcal{A}_{s,l}\rq{}$ is open.

\medskip

\paragraph{Step 4: full measure sets.}
Finally, we have to prove that if a property holds for almost every IET, then it holds for almost every $\eta \in \mathcal{A}_{s,l}\rq{}$ w.r.t.~the measure class defined in Notation \ref{definmeasureclass}. 
Fix the minimal component $\mathcal{M}_i$, let $\widetilde{\mathcal{M}}_i$ be the open neighborhood of $\mathcal{M}_i$ obtained by adding all cylinders or islands of periodic orbits adjacent to $\mathcal{M}_i$. Let $\Sigma_i$ the set of singularities in $\widetilde{\mathcal{M}}_i$, or equivalently in the closure of $\mathcal{M}_i$.

For each interval $I_j$ as above, let $\gamma_j$ be a path starting from a point $x \in I_j$ different from $a_{j-1}, a_j$, moving along the orbit of $x$ up to the first return to $I$ and closing it up in $I$, see Figure \ref{fig:1}. Set $\mathcal{B}_i = \{\gamma_j : 1 \leq j \leq d_i\}$. Let $\{\xi_r\}$ be the set of the boundary components of $\mathcal{M}_i$. By \cite[Lemma 2.17]{viana:iet}, $\mathcal{B}_i  \cup \{\xi_r\}$ is a generating set for $H_1(\widetilde{\mathcal{M}}_i,\Z)$.
Moreover, a proof analogous to \cite[Lemma 2.18]{viana:iet} shows that any loop around a singularity is a linear combination of the $\gamma_j$ (if the singularity is not contained in a saddle loop), and of the $\gamma_j$ and $\xi_r$ (if the singularity $p_{r,i}$ is contained in a saddle loop). In particular, $\mathcal{B}_i  \cup \{\xi_r\} $ is a generating set for $H_1(\widetilde{\mathcal{M}}_i \setminus \Sigma_i,\Z)$. 

\begin{lemma}\label{th:mayerviet}
Let $\mathcal{B}_i$ be as above. There exists a basis $\mathcal{B}$  of $H_1(\mathcal{M} \setminus \Sigma,\Z)$ given by the disjoint union of the $\mathcal{B}_i$ together with the homology classes of the loops $\xi$ bounding the $\widetilde{\mathcal{M}}_i$. 
\end{lemma}
\begin{proof}
Consider two minimal components $\mathcal{M}_a$ and $\mathcal{M}_b$ separated by a cylinder of periodic orbits; the same proof applies if $\mathcal{M}_b$ is an island containing a maximum or a minimum.
Notice that $\widetilde{\mathcal{M}}_a \cap \widetilde{\mathcal{M}}_b$ is a cylinder of periodic orbits containing no singularity. Let $\xi_a \in  H_1( \widetilde{\mathcal{M}}_a \setminus \Sigma_{a}, \Z)$ and $\xi_b \in H_1( \widetilde{\mathcal{M}}_b \setminus \Sigma_{b}, \Z)$ the boundary components in $\widetilde{\mathcal{M}}_a \cap \widetilde{\mathcal{M}}_b$. We remark that $\xi_a$ and $\xi_b$ are homologous.
\begin{figure}[h!]
\centering
 \includegraphics[width=.9\textwidth]{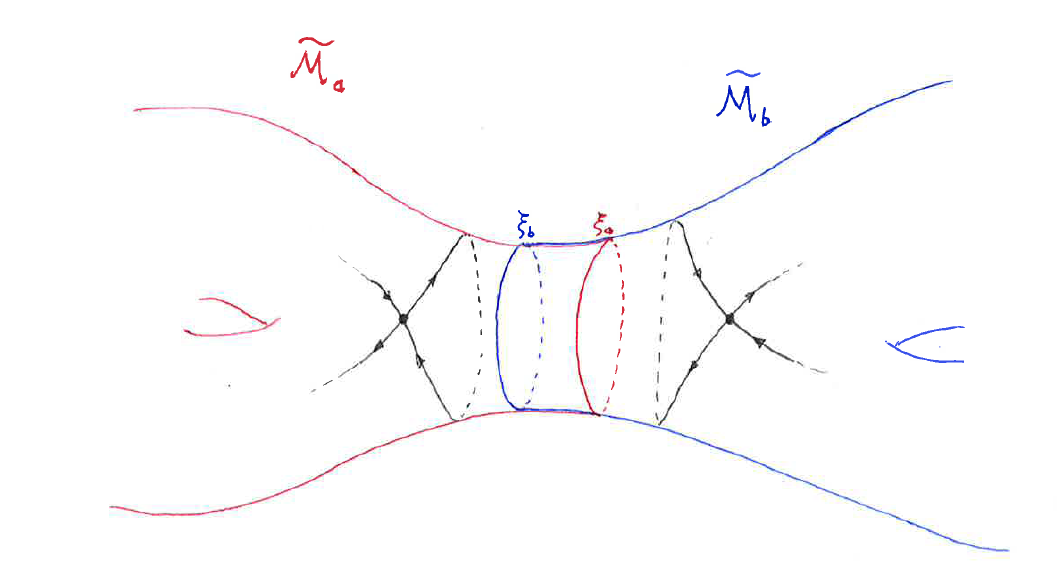}
\end{figure}

 Let $i,j,\widetilde{i},\widetilde{j}$ be the inclusion maps in the following diagram.
\begin{displaymath}
    \xymatrix{\ & \widetilde{\mathcal{M}}_a \cup \widetilde{\mathcal{M}}_b \setminus \Sigma_{a} \cup \Sigma_{b} &\ \\
            \widetilde{\mathcal{M}}_a \setminus \Sigma_{a}  \ar@{->}[ur]^{\widetilde{i}} & \ & \widetilde{\mathcal{M}}_b \setminus  \Sigma_{b} \ar@{->}[ul]_{\widetilde{j}}\\
            \ & \widetilde{\mathcal{M}}_a\cap \widetilde{\mathcal{M}}_b  \ar@{->}[ul]^i \ar@{->}[ur]_j & \ \\
}
\end{displaymath}
The Mayer-Vietoris sequence
\begin{equation*}
\begin{split}
&\cdots \xrightarrow{ } H_1( \widetilde{\mathcal{M}}_a \cap  \widetilde{\mathcal{M}}_b, \Z) \xrightarrow{(i_{\ast},j_{\ast})} H_1( \widetilde{\mathcal{M}}_a \setminus \Sigma_{a}, \Z)  \oplus  H_1( \widetilde{\mathcal{M}}_b \setminus \Sigma_{b}, \Z) \xrightarrow{\widetilde{i}_{\ast}-\widetilde{j}_{\ast}} \\
&  \xrightarrow{\widetilde{i}_{\ast}-\widetilde{j}_{\ast}}  H_1( \widetilde{\mathcal{M}}_a \cup  \widetilde{\mathcal{M}}_b \setminus \Sigma_{a} \cup \Sigma_{b}, \Z)  \xrightarrow{\partial_{\ast}}   H_0( \widetilde{\mathcal{M}}_a \cap  \widetilde{\mathcal{M}}_b , \Z)\xrightarrow{(i_{\ast},j_{\ast})}  \cdots
\end{split}
\end{equation*}
is exact. We have that $H_1( \widetilde{\mathcal{M}}_a \cap  \widetilde{\mathcal{M}}_b , \Z) = \langle \xi \rangle$, where $\xi=\xi_a=\xi_b$, and the image im$(i_{\ast},j_{\ast})$ is equal to $\langle (\xi_a, \xi_b)\rangle$. 
By exactness, it follows that $ H_1( \widetilde{\mathcal{M}}_a \setminus \Sigma_{a} ,\Z)  \oplus  H_1( \widetilde{\mathcal{M}}_b \setminus \Sigma_{b}, \Z) / \langle (\xi_a, \xi_b)\rangle \simeq \text{im}(\widetilde{i}_{\ast}-\widetilde{j}_{\ast})$.   Since $(i_{\ast},j_{\ast})\colon H_0( \widetilde{\mathcal{M}}_a \cap  \widetilde{\mathcal{M}}_b, \Z) \to H_0( \widetilde{\mathcal{M}}_a \setminus \Sigma_{a}, \Z)  \oplus  H_0( \widetilde{\mathcal{M}}_b \setminus \Sigma_{b}, \Z)$ is injective, im$(\partial_{\ast}) = \{0\}$, then ker$(\partial_{\ast}) =  H_1( \widetilde{\mathcal{M}}_a \cup \widetilde{\mathcal{M}}_b \setminus \Sigma_{a} \cup \Sigma_{b}, \Z) =\text{im}(\widetilde{i}_{\ast}-\widetilde{j}_{\ast} )$. We have obtained that 
$$
H_1( \widetilde{\mathcal{M}}_a \setminus \Sigma_{a} ,\Z)  \oplus  H_1( \widetilde{\mathcal{M}}_b \setminus \Sigma_{b}, \Z) / \langle (\xi_a, \xi_b)\rangle \simeq H_1( \widetilde{\mathcal{M}}_a \cup \widetilde{\mathcal{M}}_b \setminus \Sigma_{a} \cup \Sigma_{b}, \Z)
$$
in particular, the set $\mathcal{B}_a \cup \mathcal{B}_b $ is contained in a generating set for $H_1( \widetilde{\mathcal{M}}_a \cup \widetilde{\mathcal{M}}_b \setminus \Sigma_{a} \cup \Sigma_{b}, \Z)$ and the union is disjoint in the image, i.e.~they all give distinct elements.

Iterate this process for all components. The generating set we obtain is the disjoint union of the $\mathcal{B}_i$ together with the homology classes of the loops $\xi$ bounding the $\widetilde{\mathcal{M}}_i$. Since the cardinality of $\mathcal{B}_i $ is $d_i$, the cardinality of the set obtained is 
$
\sum_{i=1}^k d_i + l + (k-1) 
$.
By formula \eqref{eq:di}, it equals the rank of $H_1(\mathcal{M} \setminus \Sigma,\Z)$, hence it is a basis.
\end{proof}

\begin{cor}\label{corcor}
Every full measure set of length vectors $\underline{\lambda} \in \Delta_d$ corresponds to a full measure set of 1-forms $\eta \in \mathcal{A}_{s,l}\rq{}$.
\end{cor}
\begin{proof}
It is sufficient to show that for any fixed $\eta\in \mathcal{A}_{s,l}\rq{}$ we can choose a basis of $H_1(\mathcal{M}, \Sigma,\Z)$ such that the lengths of the subintervals of the induced IETs on all minimal components appear as some of the coordinates of $\Theta(\eta)$. 

Let $\mathcal{B}$ be the basis of $H_1(\mathcal{M} \setminus \Sigma,\Z)$ given by Lemma \ref{th:mayerviet}. Denote by $\widehat{\mathcal{M}}$ the surface obtained from $\mathcal{M}$ by removing a small ball centered at each singularity. By the Excision Theorem, $H_1(\mathcal{M},\Sigma, \Z) \simeq H_1(\widehat{\mathcal{M}}, \partial \widehat{\mathcal{M}}, \Z)$ and the Poincaré-Lefschetz duality implies that the latter is isomorphic to $H^1(\widehat{\mathcal{M}},\Z) \simeq H^1(\mathcal{M}\setminus \Sigma, \Z)$. At the homology level, we then have a perfect pairing given by the intersection form. Consider the basis $\{ \sigma_j \}$, where $\sigma_j \in H_1(\mathcal{M},\Sigma, \Z) $ is the dual path to $\gamma_j$, see Figure \ref{fig:1}.
If $\sigma_j \subset \mathcal{M}_i$, the associated period coordinates are given by $\int_{\sigma_j}\eta = (a_j-a_{j-1})\int_I  \eta$, which are the lengths of the subintervals defining the IET $T$ on $I \subset \mathcal{M}_i$ (up to the constant $\int_I  \eta$).
\end{proof}

Theorem \ref{th:IETgoal} implies that for every permutation $\pi$, for almost every length vector $\underline{\lambda} \in \Delta_d$ and for every function $f$ with asymmetric logarithmic singularities the suspension flow over $T=(\pi, \underline{\lambda})$ with roof function $f$ is mixing. By Corollary \ref{corcor}, consider the correspondent full measure set of 1-forms $\eta \in \mathcal{A}_{s,l}\rq{}$. By the previous steps, the restriction of the induced locally Hamiltonian flow to any minimal component can be represented as a suspension flow over an IET with roof function with asymmetric logarithmic singularities, which is mixing by Theorem \ref{th:IETgoal}. This concludes the proof.

%%%%%%%%%%%%%%%%%%%%
%%%%%%%%%%%%%%%%%%%%
%%%%%%%%%%%%%%%%%%%%

\section{Rauzy-Veech Induction and Diophantine conditions}\label{section:def}

The Rauzy-Veech algorithm is an inducing scheme which produces a sequence of IETs defined on nested subintervals of $[0,1]$ shrinking towards zero. We assume some familiarity with the Rauzy-Veech Induction, referring to \cite{viana:iet} for details. We introduce some notation and terminology that we will use in the proof of Theorem \ref{th:IETgoal}. 

We will denote by $R T$ the IET obtained in one step of the algorithm and, for any $n \geq 0$, we let $T^{(n)} := R^n T$. 
The map $T^{(n)}$ is defined on a subinterval $I^{(n)} \subset I$ of length $\lambda^{(n)}$. Let $\underline{\lambda}^{(n)} \in (\lambda^{(n)})^{-1} \Delta_{d}$ be the vector whose components ${\lambda}_j^{(n)}$ are the lengths of the subintervals $I_j^{(n)} \subset I^{(n)}$ defining $T^{(n)}$; it satisfies the following relation
\begin{equation*}
\underline{\lambda}^{(n)} = (A^{(n)})^{-1} \underline{\lambda}, \quad \text{ with } A^{(n)} \in \sldz.
\end{equation*}
We can write 
$$
A^{(n)} = A_0 \cdots A_{n-1} := A(T) \cdots A(T^{(n-1)}),
$$
where $(A^{(n)})^{-1}$ is a matrix cocycle (sometimes called the \newword{Rauzy-Veech lengths cocycle}).
For $m<n$, define also
$$
A^{(m,n)}= A_m \cdots A_{n-1} =A(T^{(m)}) \cdots A(T^{(n-1)}),
$$
so that 
\begin{equation}\label{eq:cocycleprop}
\underline{\lambda}^{(n)}  = (A^{(m,n)})^{-1} \underline{\lambda}^{(m)}.
\end{equation}

Denote by $h_j^{(n)}$ the first return time of any $x \in I_j^{(n)}$ to the induced interval $I^{(n)}$ and by $\underline{h}^{(n)}$ the vector whose components are $h_j^{(n)}$; let $h^{(n)}$ be the maximum $h_j^{(n)}$ for $j=1, \dots, d$. The following result is well-known.
\begin{lemma}\label{lemma:entries}
The $(i,j)$-entry $A^{(n)}_{i,j}$ of $A^{(n)}$ is equal to the number of visits of any point $x \in I^{(n)}_j$ to $I_i$ up to the first return time $h_j^{(n)}$ to $I^{(n)}$. In particular,
$
h_j^{(n)} = \sum_{i=1}^{d} A^{(n)}_{i,j}.
$
\end{lemma}
Let $Z_j^{(n)}$ be the orbit of the interval $I_j^{(n)}$ up to the first return time to $I^{(n)}$, namely 
$$
Z_j^{(n)} := \bigcup_{r=0}^{h_j^{(n)}-1} T^r I_j^{(n)}.
$$
We remark that the above is a disjoint union of intervals by definition of first return time.
For $0 \leq r < h_j^{(n)}$, let $ F^{(n)}_{j,r} := T^{r} (I^{(n)}_{j})$. The intervals $F^{(n)}_{j,r}$ form a partition of $I$, that we will denote $\mathcal{Z}^{(n)} $.

\begin{remark}\label{remark:endpoint}
Because of the definition of the Rauzy-Veech Induction, the partition $\mathcal{Z}^{(n)} = \{ F^{(n)}_{j,r} : 0 \leq r < h_j^{(n)} , 1 \leq j \leq d \}$ is a refinement of the partition $\mathcal{Z}^{(n-1)} $; in particular, for any $n \geq 0$, each point $a_k$ for $0 \leq k \leq d$ belongs to the boundary of some $ F^{(n)}_{j,r} $.  
\end{remark}

We say that any IET for which the result below holds satisfies the \newword{mixing Diophantine condition with integrability power $\tau$}; it was proved by Ulcigrai in \cite{ulcigrai:mixing}. We recall that the \newword{Hilbert distance} $d_H$ on the positive orthant of $\R^{d}$ is defined by $d_H(\underline{a}, \underline{b}) = \log ( \max \{a_i/ b_i\} / \min \{ a_i / b_i\})$ for any positive vectors $\underline{a}, \underline{b} \in \R^{d}$.
\begin{teo}[{\cite[Proposition 3.2]{ulcigrai:mixing}} Mixing DC]\label{th:ulcigrai}
Let $1 < \tau < 2$. For almost every IET there exist a sequence $\{n_l\}_{l \in \N}$ and constants $\nu, \kappa >1$, $0<D<1$, $D\rq{}>0$ and $\overline{l} \in \N$ such that for every $l \in \N$ we have:
\begin{itemize}
\item[(i)] $\nu^{-1} \leq \lambda_i^{(n_l)} / \lambda_j^{(n_l)} \leq \nu$ for all $1\leq i,j \leq d$;
\item[(ii)] $\kappa^{-1} \leq h_i^{(n_l)} / h_j^{(n_l)} \leq \kappa$ for all $1\leq i,j \leq d$;
\item[(iii)] $A^{(n_l, n_{l+ \overline{l}})} > 0$ and, if $d_H$ denotes the Hilbert distance on the positive orthant in $\R^{d}$,
$$
d_H \left( A^{(n_l, n_{l+ \overline{l}})} \underline{a},  A^{(n_l, n_{l+ \overline{l}})} \underline{b}\right) \leq \min \{ D d_H( \underline{a}, \underline{b}), D\rq{}\},
$$
for any vectors $\underline{a}, \underline{b}$ in the positive orthant of $\R^{d}$;
\item[(iv)] $\lim_{l \to \infty} l^{-\tau} \norma{A^{(n_l, n_{l+ 1})} }=0$.
\end{itemize}
Moreover, any IET satisfying these properties is uniquely ergodic.
\end{teo}

\begin{cor}[{\cite[Lemmas 3.1, 3.2 and 3.3]{ulcigrai:mixing}}]\label{cor:thulcigrai}
Consider the sequence $\{n_l\}_{l \in \N}$ given by Theorem \ref{th:ulcigrai}; the following properties hold.
\begin{itemize}
\item[(i)] For each $i,j \in \{ 1, \dots, d\}$, 
$$
\frac{1}{d \nu \kappa  h_j^{(n_l)}} \leq \lambda_i^{(n_l)} \leq \frac{\kappa \nu}{ h^{(n_l)} }.
$$
\item[(ii)] For any fixed $i \in \N$, 
$$
\frac{h^{(n_l)} }{ h^{(n_{l+i\overline{l}})}} \leq \frac{\kappa}{d^{i}}.
$$
\item[(iii)] For any fixed $i \in \N$, $\log \norma{A^{(n_l, n_{l+ i})}} = o(\log h^{(n_l)})$.
\end{itemize}
\end{cor}
\begin{proof}
Kac\rq{}s Theorem implies that $\sum_j h_j^{(n_l)} \lambda_j^{(n_l)} = 1$, from which it follows $\max_j h_j^{(n_l)} \lambda_j^{(n_l)} \geq 1/d$ and $\min_j h_j^{(n_l)} \lambda_j^{(n_l)} \leq 1$. These inequalities together with properties (i) and (ii) in Theorem \ref{th:ulcigrai} yield the first claim (i). 
The matrix $A^{(n_l, n_{l+\overline{l}})}$ has positive integer entries by (iii) in Theorem \ref{th:ulcigrai}, so $\min_j h_j^{(n_{l+i\overline{l}})} \geq d^i \min_j h_j^{(n_l)}$, from which (ii) follows. Finally, (iii) is obtained by combining (iv) in Theorem \ref{th:ulcigrai} and $\log h^{(n_l)} \geq \lfloor l/\overline{l} \rfloor \log d$, which is a consequence of (ii) above.
\end{proof}

%%%%%%%%%%%%%%%%%%%%
%%%%%%%%%%%%%%%%%%%%
%%%%%%%%%%%%%%%%%%%%

\section{The quantitative mixing estimates}\label{section:quantitative}

In order to prove mixing for the suspension flow $\{\phi_t\}_{t \in \R}$, we show that, for a dense set of smooth functions, the correlations tend to zero and we provide an upper bound for the speed of decay, see Theorem \ref{th:decayofcorr} below.

The first step is to estimate the growth of the Birkhoff sums of the derivative $f\rq{}$ of the roof function $f$, see Theorem \ref{th:BS}. For this part (see the Appendix \S\ref{section5bs}), we follow the same strategy used by Ulcigrai in \cite{ulcigrai:mixing}, namely, using the mixing Diophantine condition of Theorem \ref{th:ulcigrai}, we prove that \lq\lq{}most\rq\rq{}\ points in any orbit equidistribute in $I$ and we bound the error given by the other points. 
In the second part (see \S\ref{section6}), we construct a family of partitions of the unit interval following the strategy used by Ulcigrai in \cite[\S4]{ulcigrai:mixing} providing explicit bounds on their size; they are used to define a subset of the phase space of the suspension flow on which we can estimate the shearing of transversal segments. We then use a \newword{bootstrap trick} similar to the one introduced by Forni and Ulcigrai in \cite{forniulcigrai:timechanges} to reduce the study of speed of decay of correlations to the deviations of ergodic averages for IETs and finally we apply the following result by Athreya and Forni~\cite{athreyaforni:iet}.  
\begin{teo}[{\cite[Theorem 1.1]{athreyaforni:iet}}]\label{th:atfo}
Let $S$ be a compact surface and let $\Omega$ be a connected component of a stratum of the moduli space of unit-area holomorphic differentials on $S$. There exists a $\theta >0$ such that the following holds. For all $\omega \in \Omega$, there is a measurable function $K_{\omega} \colon \mathbb{S}^1 \to \R_{>0}$ such that for almost all $\alpha \in \mathbb{S}^1$, for all functions $f$ in the standard Sobolev space $\mathscr{H}^1(S)$ and for all nonsingular $x \in S$,
\begin{equation}\label{eq:deviationforsurf}
\modulo{\int_0^T f \circ \varphi_{\alpha, t}(x) \diff t - T \int f \diff A_{\omega}} \leq K_{\omega}(\alpha) \norma{f}_{\mathscr{H}^1(S)} T^{1-\theta},
\end{equation}
where $\varphi_{\alpha, t}$ is the directional flow on $S$ in direction $\alpha$ and $A_{\omega}$ is the area form on $S$ associated to $\omega$.
\end{teo}
Let $\mathscr{C}^r(\sqcup I_j)$ be the space of functions $h \colon I \to \R$ such that the restriction of $h$ to the interior of each $I_j$ can be extended to a $\mathscr{C}^r$ function on the closure of $I_j$. 
In \cite[\S3]{marmimoussayoccoz:linearization}, Marmi, Moussa and Yoccoz introduced the \newword{boundary operator}\footnote{In their paper, it is denoted by $\partial$.} $\mathcal{B} \colon \mathscr{C}^0(\sqcup I_j) \to \R^s$ to characterize which functions in $\mathscr{C}^1(\sqcup I_j) $ are induced by functions on a suspension over the interval exchange transformation, see \cite[Proposition 8.5]{marmimoussayoccoz:linearization}. We recall their result for the reader's convenience. Given an IET $T = T(\pi, \underline{\lambda})$ of $d$ intervals, define the permutation $\widehat{\pi}$ on $\{1, \dots, d\} \times \{L,R\}$ by 
\begin{equation*}
\begin{split}
&\widehat{\pi}(i,R) = (i+1,L) \text{ for } 1\leq i \leq d-1 \text{ and } \widehat{\pi}(d,R) = (\pi^{-1}(d), R),\\
&\widehat{\pi}(i,L) = (\pi^{-1}(\pi(i)-1), R) \text{ for } i \neq \pi^{-1}(1) \text{ and } \widehat{\pi}(\pi^{-1}(1), L) = (1,L).
\end{split}
\end{equation*}
The cycles of $\widehat{\pi}$ are canonically associated to the singularities of any suspension over $T$ via Veech\rq{}s zippered rectangles. The boundary operator $\mathcal{B}$ is given by
$$
(\mathcal{B} h)_C = \sum_{v \in C} \epsilon(v) h(v),
$$
where $C$ is any cycle in $\widehat{\pi}$, $\epsilon(v) = -1$ if $v = (i,L)$ and $\epsilon(v)=+1$ if $v = (i,R)$ and $h(v)$ is the limit of $h$ at the left (resp., right) endpoint of the $i$-th interval if $v=(i,L)$ (resp., if $v=(i,R)$); see \cite[Definition 3.1]{marmimoussayoccoz:linearization}. They proved the following result.
\begin{prop}[{\cite[Proposition 8.5]{marmimoussayoccoz:linearization}}]\label{th:mamoyo}
Let $S$ be a suspension over $T$ via Veech\rq{}s zippered rectangles and let $\mathscr{C}^r_c(S)$ be the space of $\mathscr{C}^r$ functions over $S$ with compact support in the complement of the singularities. For $ f \in \mathscr{C}^r_c(S)$, define 
$$
\mathcal{I}f(x) = \int_0^{\tau(x)} f \circ \varphi_t(x) \diff t,
$$
where $\tau(x)$ is the first return time of $x$ to the interval $I$ and $\varphi_t(x)$ is the vertical flow on $S$. Then, $\mathcal{I}$ maps $\mathscr{C}^r_c(S)$ continuously into $\mathscr{C}^r(\sqcup I_j)$  and its image is the subspace of functions $h$ satisfying $\mathcal{B}h = \mathcal{B}(\partial_xh) = \cdots = \mathcal{B}(\partial_x^rh) =0$.
\end{prop}
%Denote by $\mathcal{B} \colon \mathscr{C}^1(\sqcup I_j) \to \R^s$ the \newword{boundary operator} defined in \cite[\S3]{marmimoussayoccoz:linearization}.
\begin{cor}\label{th:athreyaforni}
For every permutation $\pi$ of $d$ elements there exists $0 \leq \theta <1$ such that for almost every IET $T = T(\pi, \underline{\lambda})$, for every $h \in \mathscr{C}^1(\sqcup I_j)$ satisfying $\mathcal{B}h = \mathcal{B} (\partial_x h) = 0$, there exists $C_h>0$ for which
$$
\modulo{S_r(h)(x) - r \int_0^1 h(x) \diff x} \leq C_h r^{\theta},
$$
uniformly on $x \in I$.
\end{cor}
\begin{proof}
Since almost every translation surface $S$ has a Veech\rq{}s zippered rectangle presentation (see \cite[Proposition 3.30]{viana:iet2}), Theorem \ref{th:atfo} implies that for almost every IET $T$ there exists a suspension $S$ over $T$ via zippered rectangles such that an estimate like \eqref{eq:deviationforsurf} holds for the vertical flow $\{\varphi_t\}$. Let $h$ be as in the statement of the corollary. By Proposition \ref{th:mamoyo}, there exists a function $f \in \mathscr{C}^1_c(S)$ such that $\mathcal{I}f = h$. The conclusion follows from~\eqref{eq:deviationforsurf}.
\end{proof}

\begin{defin}
We define $\mathscr{M}$ to be the set of IETs which satisfy the mixing Diophantine Condition of Theorem \ref{th:ulcigrai} and $\mathscr{Q}$ to be the set of IETs for which the conclusion of Corollary \ref{th:athreyaforni} holds. We remark that $\mathscr{M} \cap \mathscr{Q}$ has full measure.
\end{defin}

Consider the auxiliary functions $u_k, v_k, \widetilde{u}_k, \widetilde{v}_k \colon I \to \R_{>0}$ obtained by restricting to $I$ the 1-periodic functions defined by
$$
u_k(x) = 1 - \log (x-a_k), \quad \widetilde{u}_k(x) = -u_k\rq{}(x) =  \frac{1}{x-a_k} \quad \text{ for }x \in (a_k,a_k +1], 
$$
and
$$
v_k(x) = 1 - \log (a_k-x), \quad \widetilde{v}_k(x)=v_k\rq{}(x) =\frac{1}{a_k-x} \quad \text{ for } x \in [a_k-1,a_k),
$$ 
for $k = 1, \dots, d-1$. 
It will be convenient to identify functions over $I$ with 1-periodic functions over $\R$.%; any point or interval in the rest of the section must be thought as reduced $ \text{mod }\Z$.

Fix $\tau\rq{}$ such that $\tau/2 < \tau\rq{} < 1$, where $1<\tau<2$ is the integrability power of $T$ of Theorem \ref{th:ulcigrai}, and define the sequence
$$
\sigma_l = \left( \frac{\log \norma{A^{(n_l, n_{l+1})} }}{\log h^{(n_l)}} \right)^{\tau\rq{}}.
$$
The set of points for which we are able to obtain good bounds for the Birkhoff sums of $f\rq{}$ and $f\rq{}\rq{}$ contains those points whose $T$-orbit up to time $\lfloor \sigma_l h^{(n_{l+1})} \rfloor$ stay $\sigma_l \lambda^{(n_l)}$-away from all the singularities, namely the complement of the set
\begin{equation}\label{eq:sigmlk}
\Sigma_l = \bigcup_{k=1}^{d-1} \Sigma_l(k), \text{\ \ where\ \ } \Sigma_l(k)= \bigcup_{i=0}^{\lfloor \sigma_l h^{(n_{l+1})} \rfloor} T^{-i} \{ x \in I : \modulo{a_k-x} \leq  \sigma_l \lambda^{(n_l)} \}.
\end{equation}
We will show in Proposition \ref{th:stretpart} that $\misura(\Sigma_l) \to 0$ as $l$ goes to infinity.
The estimates we need are the following; the proof is given in the Appendix \S\ref{section5bs}. Ulcigrai proved an analogous statement for the case of one singularity at zero, see \cite[Corollaries 3.4, 3.5]{ulcigrai:mixing}; the proof in \S\ref{section5bs} follows her strategy, which is adapted to obtain also uniform bounds on the Birkhoff sums of $f$. 
\begin{teo}\label{th:BS}
Consider $T \in \mathscr{M}$ and let $f$ be a roof function with asymmetric logarithmic singularities; let $C=-C^{+} +C^{-} = - \sum_j C_j^{+} + \sum_j C_j^{-}$.
Define
$$
\widetilde{U}(r,x) := \max_{1\leq k \leq d-1} \max_{0 \leq i < r} \widetilde{u}_k(T^ix), \qquad \widetilde{V}(r,x) := \max_{1 \leq k \leq d-1} \max_{0 \leq i < r} \widetilde{v}_k(T^ix).
$$
For any $\varepsilon > 0$ there exists $\overline{r} >0$ such that for $r \geq \overline{r}$ if $h^{(n_l)} \leq r < h^{(n_{l+1})}$, $x \notin \Sigma_l$
and $x$ is not a singularity of $S_r(f)$, then
\begin{align*}
S_r(f)(x) \leq &2 r + \const \max_{1\leq k \leq d-1} \max_{0 \leq i <r} \modulo{\log \modulo{T^ix_0 - a_k}} \\ 
S_r(f\rq{})(x) \leq &(C+ \varepsilon) r \log r + ( C^{-}+1) ( \lfloor \kappa \rfloor +2) \widetilde{V}(r,x) \\
S_r(f\rq{})(x) \geq &(C- \varepsilon) r \log r - ( C^{+}+1) ( \lfloor \kappa \rfloor +2) \widetilde{U}(r,x) \\
\modulo{S_r(f\rq{}\rq{})(x)} \leq &(2  \max\{ \widetilde{U}(r,x),\widetilde{V}(r,x) \} +1)(C^{+}+C^{-}+\varepsilon) \times \\ %\\
%&\qquad \qquad \qquad \qquad 
& \times \big( r \log r +  (\lfloor \kappa \rfloor+2) (\widetilde{U}(r,x) + \widetilde{V}(r,x)) \big),
\end{align*}
where we recall $\kappa$ is given in Theorem \ref{th:ulcigrai}.
\end{teo}

The previous estimates are interesting in their own right, since they are used by Kanigowski, Kulaga and Ulcigrai in \cite{KKU:multmixing} to strengthen mixing to mixing of all orders for a full-measure set of flows.
In the proof of Theorem \ref{th:decayofcorr} below, we will exploit them only for a fixed $0 < \varepsilon < \modulo{C}$.

We recall from \eqref{eq:suspspace} that $\mathcal{X}$ is the phase space of the suspension flow $\{\phi_t\}$. Let $\Phi \colon \mathcal{X} \to \mathcal{M}\rq{}$ be the measurable isomorphism between $\{\phi_t\}$ and the locally Hamiltonian flow $\{\varphi_t\}$ on the minimal component $\mathcal{M}\rq{}$. We prove a bound on the speed of the decay of correlations for the pull-backs of functions in $\mathscr{C}^1_c(\mathcal{M}\rq{})$.
\begin{teo}\label{th:decayofcorr}
Let $\{\phi_t\}_{t \in \R}$ be a suspension flow over an IET $T \in \mathscr{M} \cap \mathscr{Q}$ with roof function with asymmetric logarithmic singularities.
Then, there exists $0 < \gamma <1$ such that for all $g, h \in \Phi^{\ast}(\mathscr{C}^1_c(\mathcal{M}\rq{}))$ with $\int_{\mathcal{X}}g \diff \misura = 0$ we have
$$
\modulo{ \int_{\mathcal{X}}(g \circ \phi_t) h \diff \misura } \leq \frac{C_{g,h}}{(\log t)^{\gamma}},
$$
for some constant $C_{g,h}>0$.
\end{teo}
Theorem \ref{th:1.2} is an immediate consequence of Theorem \ref{th:decayofcorr}.

\begin{proof}[Proof of Theorem \ref{th:IETgoal}]
We show that Theorem \ref{th:decayofcorr} implies Theorem \ref{th:IETgoal}.
It is sufficient to prove that $\Phi^{\ast}(\mathscr{C}^1_c(\mathcal{M}\rq{}))$ is dense in $L^2(\mathcal{X})$. We claim that $\Phi^{\ast}(\mathscr{C}^1_c(\mathcal{M}\rq{}))$ contains the dense subspace $\mathscr{C}^1_c(\mathcal{X})$ of $\mathscr{C}^1$ functions with compact support on $\mathcal{X}$. Indeed, we show that for any compact set $\mathcal{K} \subset \mathcal{M}\rq{} \setminus \Sigma$ in the complement of the singularities, $\Phi$ is a diffeomorphism between $\Phi^{-1}(\mathcal{K})$ and $\Phi(\Phi^{-1}(\mathcal{K})) \subseteq \mathcal{K}$.

For any $p \in \Phi(\Phi^{-1}(\mathcal{K})) $, choose local coordinates around $p$ such that the vector field generating flow $\{\varphi_t\}$ is $\partial_y$; then, if $\omega = V(x,y) \diff x \wedge \diff y$, we have that $\eta = - V(x,y) \diff x$. On $\mathcal{X}$, the 1-form $\eta$ equals $\diff x$; in these coordinates, $\Phi$ is the solution to the well-defined system of ODEs $\partial_x \Phi = -1/(V \circ \Phi)$ and $\partial_y\Phi = 0$. By compactness, the $\mathscr{C}^{\infty}$-norm of $V$ is uniformely bounded, and so is the $\mathscr{C}^{\infty}$-norm of $\Phi$; thus $\Phi$ is a diffeomorphism. 
\end{proof}

\begin{remark}\label{remark:3}
The argument above shows that any $g \in \Phi^{\ast}(\mathscr{C}^1_c(\mathcal{M}\rq{}))$ is a $\mathscr{C}^1$ function on $\mathcal{X}$. Moreover, define the operator $\mathcal{I}$ as in Proposition \ref{th:mamoyo}, namely
\begin{equation}\label{eq:idig}
(\mathcal{I}g) (x)= \int_0^{f(x)} g(x,y) \diff y.
\end{equation}
The same proof as \cite[Proposition 8.5]{marmimoussayoccoz:linearization} shows that $\mathcal{I}g \in \mathscr{C}^1(\sqcup I_j)$ and $\mathcal{B}(\mathcal{I}g) = \mathcal{B}(\partial_x(\mathcal{I}g)) =0$, in particular $\mathcal{I}g$ satisfies the hypotheses of Corollary \ref{th:athreyaforni}. 
\end{remark}

%%%%%%%%%%%%%%%%%%%%
%%%%%%%%%%%%%%%%%%%%
%%%%%%%%%%%%%%%%%%%%

\section{Proof of Theorem \ref{th:decayofcorr}}\label{section6}

The first part of the proof consists of defining a subset $X(t) \subset \mathcal{X}$ on which we can estimate the shearing of segments transverse to the flow in the flow direction. The construction of $X(t)$ follows the lines of \cite[\S4]{ulcigrai:mixing}, although here we need to make all estimates explicit. In the second part of the proof, we reduce correlations to integrals along long pieces of orbits by a bootstrap trick analogous to \cite{forniulcigrai:timechanges} and we conclude by applying the result by Athreya and Forni on the deviations of ergodic averages in the form of Corollary \ref{th:athreyaforni}.

Within this section, we will always assume that $f$ has asymmetric logarithmic singularities and $T \in \mathscr{M} \cap \mathscr{Q}$. %As usual, the meaning of identities of the form $f_1(t) = O(f_2(t))$ is that there exists some constant $K>0$ satisfying $|f_1(t)| \leq K |f_2(t)|$ for all $t$ sufficiently large.
 
\subsection{Preliminary partitions}

Let $R(t) := \lfloor t/m \rfloor +2$, where $m = \min \{ 1, \min f\}$.
A \newword{partial partition} $\mathcal{P}$ is a collection of pairwise disjoint subintervals $J =[a,b)$ of the unit interval $I = [0,1]$.
\begin{prop}\label{th:prelpart}
Let $0< \alpha < 1$. For each $M >1$ there exists $t_0 >0$ and partial partitions $\mathcal{P}_p(t)$ for $t \geq t_0$ such that $1 - \misura(\mathcal{P}_p(t)) = O\left( (\log t)^{-\alpha} \right)$ and for each $J \in \mathcal{P}_p(t)$ we have 
\begin{itemize}
\item[(i)] $T^j$ is continuous on $J$ for each $0 \leq j \leq R(t)$;
\item[(ii)] $\frac{1}{t(\log t)^{\alpha}} \leq \misura(J) \leq \frac{2}{t (\log t)^{\alpha}}$;
\item[(iii)] $\dist( T^j J, a_k) \geq \frac{M}{t(\log t)^{\alpha}}$ for $0 \leq j \leq R(t)$;
\item[(iv)] $f(T^jx)\leq C_f \log t$ for each $0 \leq j \leq R(t)$ and for all $x \in J$, where $C_f>0$ is a fixed constant.
\end{itemize}
\end{prop}
\begin{proof}
Let $\mathcal{P}_0(t)$ be the partition of $I$ into continuity intervals for $T^{R(t)}$. Consider the set
$$
U_1 = \bigcup_{k=0}^d \bigcup_{j=0}^{R(t)} \left\{ x \in I : \modulo{x - T^{-j}a_k} \leq \frac{2M}{t(\log t)^{\alpha}} \right\},
$$
%which is the union of closed balls of radius $2M/(t(\log t)^{\alpha})$ centred at the discontinuities of $T^{R(t)}$. Let 
and let $\mathcal{P}_1(t)$ be obtained from $\mathcal{P}_0(t)$ by removing all partition elements fully contained in $U_1$. Then 
$$
1 - \misura(\mathcal{P}_1(t)) \leq \misura(U_1) \leq (d+1) \left( \frac{t}{m} +3\right) \frac{4M}{t(\log t)^{\alpha}} = O\left((\log t)^{-\alpha}\right). 
$$
Any $J \in \mathcal{P}_1(t)$ contains at least one point outside $U_1$, therefore, since the endpoints of $J$ are centres of the balls in $U_1$, we have $\misura(J) \geq 4M/(t (\log t)^{\alpha})$. Let
$$
U_2 = \bigcup_{k=0}^d \bigcup_{j=0}^{R(t)}  T^{-j} \left\{ x \in I : \modulo{x - a_k} \leq \frac{M}{t(\log t)^{\alpha}} \right\},
$$
and let $\mathcal{P}_2(t) = \mathcal{P}_1(t) \setminus U_2$. As before we have that 
$$
\misura(\mathcal{P}_1(t)) - \misura(\mathcal{P}_2(t)) \leq \misura(U_2) = O\left( (\log t)^{-\alpha}\right).
$$
By construction, property (iii) is satisfied. Moreover, any interval $J \in \mathcal{P}_2(t)$ is either an interval in $\mathcal{P}_1(t)$ or is obtained from one of them by cutting an interval of length at most $M/(t(\log t)^{\alpha})$ on one or both sides, hence $\misura(J) \geq 2M/(t(\log t)^{\alpha})$. Cut each interval $J \in \mathcal{P}_2(t)$ in such a way that (ii) is satisfied and call $\mathcal{P}_p(t)$ the resulting partition. Finally, there exists a constant $C_f\rq{}$ such that, by (iii), for all $x \in \mathcal{P}_p(t)$ and all $0 \leq j \leq R(t)$ we have $f(T^jx) \leq C_f\rq{} \log(t (\log t)^{\alpha}) \leq (C_f\rq{} +1) \log t$, up to increasing $t_0$. Thus (iv) holds with $C_f = C_f\rq{}+1$.
\end{proof}

\paragraph{Rough lower bound on $r(x,t)$.}
We want to bound the number $r(x,t)$ of iterations of $T$ up to time $t$ (see \eqref{eq:sfj}). From the definition, $r(x,t) \leq R(t)$.
By property (iv) in Proposition~\ref{th:prelpart}, %there exists a constant $C_f>1$ such that for all $x \in \mathcal{P}_p(t)$ we have $f(T^jx) \leq (C_f-1) \log( t(\log t)^{\alpha}) \leq C_f \log t$, up to increase $t_0$. Therefore, 
$$
t < S_{r(x,t) + 1}(f)(x) \leq C_f (r(x,t)+1) \log t,
$$ 
which, up to enlarging $t_0$ if necessary, implies 
\begin{equation}\label{eq:roughlb}
r(x,t) > \frac{t}{2C_f \log t},
\end{equation}
uniformly for $x \in \mathcal{P}_p(t)$.

\subsection{Stretching partitions}
We refine the partitions $\mathcal{P}_p(t)$ in order for Theorem \ref{th:BS} to hold.
Let $l(t) \in \N$ be such that $h^{(n_{l(t)})}\leq R(t) < h^{(n_{l(t)+1})}$.
\begin{lemma}\label{lemmaaa}
If $ \frac{t}{2C_f \log t} \leq r(x,t) \leq R(t)$, then $h^{(n_{l(t) - L(t)})} \leq r(x,t) < h^{(n_{l(t)+1})}$ for all $x \in \mathcal{P}_p(t)$, where $L(t) =O( \log \log t)$.
\end{lemma}
\begin{proof}
By Corollary \ref{cor:thulcigrai}-(ii), for each $\overline{L} \in \N$ we have
$$
h^{(n_{l(t)-\overline{L}\overline{l}})} \leq \frac{\kappa}{d^{\overline{L}}} h^{(n_{l(t)})} \leq \frac{\kappa}{d^{\overline{L}}} R(t)\leq \frac{2 \kappa t}{md^{\overline{L}}}.
$$
It is sufficient to choose $\overline{L}$ minimal such that $ 2\kappa t/(m d^{\overline{L}}) <  t/(2C_f \log t) $; this case is achieved with an $L(t) = \overline{L}\overline{l} =O( \log \log t)$.
\end{proof}

\begin{lemma}\label{th:stimelt}
We have that $l(t) = O(\log t)$ and, for any $\varepsilon >0$, $l(t)^{-1} = O\left((\log t)^{-\frac{1}{1+\varepsilon}}\right)$. 
\end{lemma}
\begin{proof}
By Corollary \ref{cor:thulcigrai}-(ii) we have 
$$
d^{\lfloor l(t)/ \overline{l} \rfloor} \leq \kappa h^{(n_{l(t)})} \leq \kappa R(t) \leq \frac{2\kappa t}{m}, 
$$
so that $l(t) = O(\log t)$. For the other inequality, we use the Diophantine condition (iv) in Theorem \ref{th:ulcigrai} to get
 \begin{equation*}
\begin{split}
\log h^{(n_{l(t)+1})} &\leq \log ( \norma{A^{(n_0, n_{l(t)+1})}}) \leq \log (\norma{A^{(n_{l(t)}, n_{l(t)+1})}} \cdots \norma{A^{(n_0, n_1)}}) \\
& = \sum_{i=0}^{l(t)} \log (\norma{A^{(n_i,n_{i+1})}}) = O \left( \sum_{i=1}^{l(t)} \log(i^{\tau}) \right) \\
&= O \left( \int_1^{l(t)+1} \log x \diff x \right) =O( l(t) \log l(t) ) =O( l(t)^{1+\varepsilon}).
\end{split}
\end{equation*}
%whenever $t \geq t_1$, for $t_1$ depending on $\varepsilon$.
The conclusion follows from $\log h^{(n_{l(t)+1})} \geq \log R(t) \geq \log t$.
\end{proof}

We now assume $C^{+}>C^{-}$; the proof in the other case is analogous.
\begin{prop}\label{th:stretpart}
Suppose $C^{+} > C^{-}$. There exist $t_1 \geq t_0$, constants $C\rq{}, \widetilde{C}\rq{}, C\rq{}\rq{} >0 $ and a family of refined partitions $\mathcal{P}_s(t) \subset \mathcal{P}_p(t)$ for all $t \geq t_1$, with $1-\misura(\mathcal{P}_s(t))= O( (\log t)^{-\alpha\rq{}})$ for some $0<\alpha\rq{}<1$, such that for all $x \in \mathcal{P}_s(t)$
\begin{itemize}
\item[(i)] $S_{r(x,t)}(f)(x) \leq 3t$,
\item[(ii)] $S_{r(x,t)}(f\rq{})(x) \leq -C\rq{} t \log t$,\\
\item[(iii)] $\modulo{S_{r(x,t)}(f\rq{})(x) }\leq \widetilde{C}\rq{} t \log t$,\\
\item[(iv)] $S_{r(x,t)}(f\rq{}\rq{})(x) \leq \frac{C\rq{}\rq{}}{M} t^2 (\log t)^{1+\alpha}$.
\end{itemize}
\end{prop}
\begin{proof}
Recall the definition of $\Sigma_l$ in \eqref{eq:sigmlk} and that $r(x,t)$ is the number of iterations of $T$ applied to $x$ up to time $t$. Theorem \ref{th:BS} provides bounds for the Birkhoff sums $S_{r(x,t)}(f)(x)$ and $S_{r(x,t)}(f\rq{})(x)$ for all $x \notin \Sigma_l$, where $l$ is such that $h^{(n_l)} \leq r(x,t) < h^{(n_{l+1})}$.
By Lemma \ref{lemmaaa} we know that $h^{(n_{l(t) - L(t)})} \leq r(x,t) < h^{(n_{l(t)+1})}$ for all $x \in \mathcal{P}_p(t)$, hence to make sure we can apply Theorem \ref{th:BS}, it is sufficient to remove all sets $\Sigma_l$, with $l(t)-L(t) \leq l \leq l(t)$. Thus, we define
$$
\widehat{\Sigma}(t) = \bigcup_{k=1}^{d-1} \ \bigcup_{l = l(t)-L(t)}^{l(t)} \Sigma_l(k).
$$
Let $\mathcal{P}_s(t)$ be obtained from $\mathcal{P}_p(t)$ by removing all intervals which intersect $\widehat{\Sigma}(t)$. We estimate the total measure of $\mathcal{P}_s(t)$. 
If $J \in \mathcal{P}_p(t)$ intersects $\widehat{\Sigma}(t)$, then either $ J \subset \widehat{\Sigma}(t)$ or $T^jJ$ contains some point of the form $a_k \pm \sigma_l \lambda^{(n_l)}$ for some $0 \leq j \leq R(t)$ and $ l(t)-L(t) \leq l \leq l(t)$. Therefore, by Lemma \ref{lemmaaa},
%\begin{equation*}
%\begin{split}
%&
\begin{multline*}
\misura(\mathcal{P}_p(t))-\misura(\mathcal{P}_s(t)) \leq  \misura(\widehat{\Sigma}(t)) + \frac{2}{t(\log t)^{\alpha}} (R(t)+1) 2d(L(t)+1)\\
%& \qquad 
= \misura(\widehat{\Sigma}(t)) + O\left( \frac{\log \log t}{(\log t)^{\alpha}}\right) = \misura(\widehat{\Sigma}(t)) + O\left( (\log t)^{-\alpha_1}\right),
\end{multline*}
%\end{split}
%\end{equation*}
for some $\alpha_1 < \alpha$. From Corollary \ref{cor:thulcigrai} we get
\begin{multline*}
\misura(\widehat{\Sigma}(t))  =O\left( L(t) \sigma_{l(t)}^2  \lambda^{(n_{l(t)})}h^{(n_{l(t)+1})} \right) =O\left( L(t) \sigma_{l(t)}^2  \frac{h^{(n_{l(t)+1})}}{h^{(n_{l(t)})}} \right)\\
=O\left( L(t) \sigma_{l(t)}^2 \norma{A^{(n_{l(t)},n_{l(t)+1})}} \right) =O\left( L(t) \frac{(\log l(t))^{2\tau\rq}}{l(t)^{2 \tau\rq -\tau}} \right) =O\left(\frac{L(t)}{l(t)^{\alpha_2}}\right),
\end{multline*}
for some $\alpha_2 > 0$, since $2 \tau\rq > \tau$. 

From Lemma \ref{th:stimelt}, we deduce that 
$$
\misura(\widehat{\Sigma}(t)) =O\left( \frac{\log \log t}{(\log t)^{\frac{\alpha_2}{1+\varepsilon}}} \right) =O\left((\log t)^{-\alpha_3}\right),
$$
for some $\alpha_3 >0$, so that 
$$
1- \misura(\mathcal{P}_s(t)) \leq (1- \misura(\mathcal{P}_p(t))) + (\misura(\mathcal{P}_p(t))-\misura(\mathcal{P}_s(t))) =O\left( (\log t)^{-\alpha\rq{}}\right),
$$
for some $0 < \alpha\rq{} \leq \min \{\alpha_1, \alpha_3\}$.

Fix  $0< \varepsilon < - C = C^{+}-C^{-}$. By \eqref{eq:roughlb}, we have $r(x,t) \geq t/ (2C_f \log t ) \geq t_1/ (2C_f \log t_1)$; let us choose $t_1$ such that the latter is greater than $\overline{r}$ in Theorem \ref{th:BS}. By construction, the estimates on the Birkhoff sums of $f$ and $f\rq{}$ hold for all $x \in \mathcal{P}_s(t)$.

\begin{lemma} \label{th:boundsonr}
For all $x \in \mathcal{P}_s(t)$ we have that $t/3 \leq r(x,t) \leq R(t) \leq 2t/m$.
\end{lemma}
\begin{proof}
We only have to prove the lower bound. By definition and by the uniform estimates on the Birkhoff sums of $f$ in Theorem \ref{th:BS} we have
$$
t < S_{r(x,t) + 1}(f)(x) \leq 2 (r(x,t)+1) + \const \max_{0 \leq i \leq r(x,t)} f(T^ix).
$$
Since $f(T^ix) \leq C_f \log t$ for all $x \in \mathcal{P}_s(t)$ by Proposition \ref{th:prelpart}-(iv), the conclusion follows up to increasing $t_1$.
\end{proof}

Let us show (ii). From the fact that $\modulo{x-a_k}^{-1} \leq t(\log t)^{\alpha}/M$, we have that 
$$
S_{r(x,t)}(f\rq{})(x) \leq (C + \varepsilon) r(x,t) \log r(x,t) \left( 1 + O\left( \frac{t (\log t)^{\alpha}}{r(x,t)\log r(x,t)}\right)\right).
$$
By Lemma \ref{th:boundsonr}, 
$$
O \left( \frac{t (\log t)^{\alpha}}{r(x,t)\log r(x,t)} \right) = O \left( (\log t)^{\alpha-1}\right);
$$
therefore we deduce (ii) with $-C\rq{} = (C+\varepsilon)/4 <0$. Proceeding in an analogous way, one gets (i), (iii) and (iv).
\end{proof}

\subsection{Final partition and mixing set}

\begin{prop}\label{th:mixpart}
There exist $\alpha\rq{}\rq{} >0$ and $t_2 \geq t_1$ such that for all $t \geq t_2$ there exists a family of refined partitions $\mathcal{P}_f(t) \subset \mathcal{P}_s(t)$ with $1- \misura(\mathcal{P}_f(t) ) = O( (\log t)^{-\alpha\rq{}\rq{}})$ such that for all $x \in J = [a,b) \in \mathcal{P}_f(t) $ we have
\begin{equation}\label{eq:mixpart}
\min_{1 \leq k \leq d} \modulo{T^r x -a_k} \geq \frac{1}{(\log t)^2},
\end{equation}
for all $r(a,t) \leq r \leq r(a,t) + \frac{2C_f}{m} \log t$.
\end{prop}
\begin{proof}
Let $K(t) = \lfloor \frac{2C_f}{m} \log t \rfloor +1$ and define
$$
U_3 = \bigcup_{k=1}^{d-1} \bigcup_{i=- K(t)}^{K(t)} T^i \left\{x \in I : \modulo{x-a_k} \leq \frac{1}{(\log t)^{2}}\right\}.
$$
Since $T^{\pm K(t)}$ is an IET of at most $d(K(t)+1)$ intervals, the set $U_3$ consists of at most $O\left( K(t)^2 \right)$ intervals. 
Let 
$$
U_4 = \left\{ x \in I : \text{dist}(x, U_3) \leq \frac{2}{t (\log t)^{\alpha}} \right\}, \text{\ \ \ and \ \ \ } U_5 = T_t^{-1}U_4,
$$
where $T_t (x) = T^{r(t,x)}x$. The measure of $U_4$ is bounded by the measure of $U_3$ plus the number of intervals in $U_3$ times $4/(t (\log t)^{\alpha})$, namely
\begin{equation*}
\begin{split}
\misura(U_4) & \leq \misura(U_3) + O\left( \frac{K(t)^2}{t(\log t)^{\alpha}} \right) \leq \frac{d(2K(t) +1)}{(\log t)^2} + O\left( \frac{(\log t)^{2-\alpha}}{t} \right)\\
&= O\left( (\log t)^{-1}\right).
\end{split}
\end{equation*}
We apply the following lemma by Kochergin.
\begin{lemma}[{\cite[Lemma 1.3]{kochergin:lemma}}]
For any measurable set $U \subset I$,
$$
\misura(T_t^{-1}U) \leq \int_U \left(\frac{f(x)}{m}+1 \right)\diff x.
$$
\end{lemma}
The previous result and the Cauchy-Schwarz inequality give us
$$
\misura(U_5) \leq \int_{U_4} \left(\frac{f(x)}{m}+1 \right)\diff x \leq \left(1+\frac{\norma{f}_2}{m}\right) \misura(U_4)^{1/2} =O\left( (\log t)^{-1/2}\right),
$$
since $f \in L^2(I)$.

Let $\mathcal{P}_f(t)$ be obtained from $\mathcal{P}_s(t)$ by removing all intervals $J \in \mathcal{P}_s(t)$ such that $J \subset U_5$. Then $1 - \misura(\mathcal{P}_f(t)) \leq 1 -  \misura(\mathcal{P}_s(t)) + O((\log t)^{-1/2}) = O((\log t)^{-\alpha\rq{}\rq{}})$ for some $\alpha\rq{}\rq{} >0$. 

We show that the conclusion holds for all $J = [a,b) \in \mathcal{P}_f(t)$. By construction, there exists $y \in J$ such that $T^{r(y,t)}y \notin U_4$, therefore, using Proposition \ref{th:prelpart}-(ii), $T^{r(y,t)}x \notin U_3$ for all $x \in J$. In particular, for all $x \in J$, the inequality \eqref{eq:mixpart} is satisfied for all $r(y,t)-K(t) \leq r \leq r(y,t) + K(t)$. To conclude, we notice that, arguing as in \cite[Corollary 4.2]{ulcigrai:mixing}, we have 
\begin{multline*}
r(a,t) \leq r(y,t) \leq r(a,t) + \sup_{z \in J} \frac{S_{r(z,t)}(f\rq{})(z)}{t(\log t)^{\alpha}} \\
\leq r(a,t) + O \left( (\log t)^{1-\alpha} \right) \leq r(a,t) + K(t),
\end{multline*}
for $t \geq t_2$, for some $t_2 \geq t_1$. Hence $r(y,t) - K(t) \leq r(a,t)$ and $r(a,t)+K(t) \leq r(y,t) + K(t)$.
\end{proof}

We now define the subset $X(t)$ of $\mathcal{X}$ on which we can estimate the correlations. It consists of full vertical translates of intervals $J\in \mathcal{P}_f(t)$, namely we consider
$$
X(t)= \bigcup_{J \in \mathcal{P}_f(t)} \{ (x,y) : x \in J, 0 \leq y \leq \inf_{x \in J} f(x) \}.
$$
We can bound the measure of $X(t)$ by
$$
\misura(X(t)) \geq 1- \int_{I \setminus \mathcal{P}_f(t)} f(x) \diff x - \sum_{J \in \mathcal{P}_f(t)} \int_J (f(x) - \inf_J f) \diff x.
$$
Since $f \in L^2(I)$, Cauchy-Schwarz inequality yields
$$
\int_{I \setminus \mathcal{P}_f(t)} f(x) \diff x  \leq  \norma{f}_2 \misura(I \setminus \mathcal{P}_f(t))^{1/2} = O \left((\log t)^{-\alpha\rq{}\rq{}/2}\right).
$$
On the other hand, by the Mean-Value Theorem and Proposition \ref{th:prelpart}-(ii),
\begin{multline*}
\sum_{J \in \mathcal{P}_f(t)} \int_J (f(x) - \inf_J f) \diff x = \sum_{J \in \mathcal{P}_f(t)} \misura(J)( f(x_J) - \inf_J f ) \\
\leq \frac{2}{t (\log t)^{\alpha}}  \sum_{J \in \mathcal{P}_f(t)} \modulo {f(x_J) - \inf_J f } \leq \frac{2}{t (\log t)^{\alpha}} \cdot \text{Var}(f |_{\mathcal{P}_f(t)}); % \leq  \const \frac{\log (t (\log t)^{\alpha})}{t (\log t)^{\alpha}} ;
\end{multline*}
where $ \text{Var}(f |_{\mathcal{P}_f(t)})$ denotes the variation of $f$ restricted to $ \mathcal{P}_f(t)$. Since $f$ has logarithmic singularities at the points $a_k$ and $\text{dist}(\mathcal{P}_f(t), a_k) \geq 1/(t(\log t)^{\alpha})$, the variation is of order $ \text{Var}(f |_{\mathcal{P}_f(t)}) = O\left( \log (t(\log t)^{\alpha}) \right)$. Hence,
$$
1- \misura(X(t)) = O \left( (\log t)^{-\beta}\right),
$$
for some $0 < \beta\leq \alpha\rq{}\rq{}$.

\subsection{Decay of correlations}
In this proof of mixing, shearing is the key phenomenon. We show that the speed of decay of correlations can be reduced to the speed of equidistribution of the flow by an argument in the spirit of Marcus \cite{marcus:horocycle}, using a bootstrap trick inspired by \cite{forniulcigrai:timechanges}.
The geometric mechanism is the following: each horizontal segment $\{(x,y) : x \in J \in \mathcal{P}_f(t)\}$ in $ X(t)$ gets sheared along the flow direction and approximates a long segment of an orbit of the flow $\phi_t$, see Figure \ref{fig:2}.

Consider an interval $J = [a,b) \in \mathcal{P}_f(t)$ and let $\xi_J (s) = (s,0)$ for $a \leq s < b$. 
On $J$ the function $r(\cdot,t)$ is non-decreasing (non-increasing, if $C^{-}>C^{+}$). To see this, let $x<y$; then, since $S_{r(x,t)}(f\rq{}) <0$, the function $S_{r(x,t)}(f)$ is decreasing, hence $S_{r(x,t)}(f)(y)<S_{r(x,t)}(f)(x) \leq t$. By definition of $r(\cdot,t)$, it follows that $r(y,t) \geq r(x,t)$. 
Moreover, $r(\cdot, t)$ assumes finitely many different values $r(a,t), r(a,t)+1, \dots, r(a,t)+N(J)$; more precisely there exist $u_0 = a < u_1 < \cdots < u_{N(J)} < u_{N(J)+1} = b$ such that $r(x,t) = r(a,t) + i$ for all $x \in [u_i, u_{i+1})$. Denote $\xi_i = \xi_J|_{[u_i,u_{i+1})}$.
For $a < u < b$, define also $\xi_{[a,u)} = \xi_J|_{[a,u)}$ and let $N(u)$ be the maximum $i$ such that $u_i < u$.  

For all $a< u < b$ the curve $\phi_t \circ \xi_{[a,u)} $ splits into $N(u)$ distinct curves $\phi_t \circ \xi_i$ on which the value of $r(x,t)$ is constant. The tangent vector is given by
\begin{equation}\label{eq:tgtvect}
\frac{\diff}{\diff s} \phi_t \circ \xi_{[a,u)} (s)= \frac{\diff}{\diff s} (T^{r(s,t)}(s), t- S_{r(s,t)}(f)(s)) = (1, - S_{r(s,t)}(f\rq{})(s)).
\end{equation}
In particular, for any $(x,y) \in X(t)$ we have 
\begin{equation}\label{eq:pushforwder}
[(\phi_t)_{\ast}(\partial_x)]\negthickspace\upharpoonright_{(x,y)} = \partial_x\negthickspace\upharpoonright_{(x,y)} - S_{r(x,t+y)}(f\rq{})(x) \partial_y\negthickspace\upharpoonright_{(x,y)}. 
\end{equation}
The total \lq\lq vertical stretch\rq\rq\ $\Delta f(u)$ of $\phi_t \circ \xi_{[a,u)} $ is the sum of all the vertical stretches of the curves  $\phi_t \circ \xi_i$; by definition, it equals
$$
\Delta f(u) = \int_{\phi_t \circ \xi_{[a,u)} }\modulo{\diff y} = \int_a^{u} \modulo{S_{r(s,t)}(f\rq{})(s)} \diff s,
$$
and, by Proposition \ref{th:stretpart}-(iii),
\begin{equation}\label{eq:vertstretch}
\Delta f(u) \leq (u-a) \sup_{a\leq s <u}\modulo{S_{r(s,t)}(f\rq{})(s)} \leq \widetilde{C}\rq{} (t \log t) (u-a) \leq 2 \widetilde{C}\rq{} (\log t)^{1-\alpha};
\end{equation}
in particular we get
\begin{equation}\label{eq:Nu}
N(u) \leq \left\lfloor \frac{\Delta f(u)}{m} \right\rfloor+2 \leq \frac{4 \widetilde{C}\rq{}}{m} (\log t)^{1-\alpha}.
\end{equation}
Let also $\Delta t(u) = S_{r(u,t)}(f)(a)-S_{r(u,t)}(f)(u)$ be the delay accumulated by the endpoints $a$ and $u$. In Figure \ref{fig:2}, $\Delta f(u)$ is the sum of the vertical lengths of the curves $\phi_t \circ \xi_i$, whence $\Delta t(u)$ equals the length of the orbit segment $\gamma$. By the Mean-Value Theorem, there exists $z \in [a,u]$ such that $\Delta t(u) = -S_{r(u,t)}(f\rq{})(z)(u-a)$. Theorem \ref{th:BS} and Lemma \ref{th:boundsonr} yield
\begin{equation}\label{eq:deltatu}
\Delta t(u) = O\left( (t \log t) \frac{2}{t(\log t)^{\alpha}} \right) = O \left( (\log t)^{1-\alpha} \right).
\end{equation}

We estimate the decay of correlations 
$$
\langle g \circ \phi_t, h \rangle = \int_{\mathcal{X}} (g \circ \phi_t) h \diff \misura,
$$
for $g,h$ as in the statement of the theorem. We have that 
\begin{equation}
\begin{split}
\modulo{\int_{\mathcal{X}} (g \circ \phi_t) h \diff \misura } & \leq \modulo{ \int_{X(t)} (g \circ \phi_t) h \diff \misura }+ \misura({\mathcal{X}} \setminus X(t)) \norma{g}_{\infty}\norma{h}_{\infty} \\
& = \modulo{\int_{X(t)} (g \circ \phi_t) h \diff \misura} + O \left( (\log t)^{-\beta} \right). \label{eq:dcxt}
\end{split}
\end{equation}
By Fubini\rq{}s Theorem
\begin{equation}\label{eq:fubini1}
\int_{X(t)} (g \circ \phi_t) h \diff \misura = \sum_{J \in \mathcal{P}_f(t)} \int_0^{y_J} \int_a^{b} (g \circ \phi_{t+y} \circ \xi_J(s)) (h \circ \phi_y \circ \xi_J(s)) \diff s \diff y,
\end{equation}
where $J=[a,b)$ and $y_J = \inf_J f$.

Fix any $0 \leq \overline{y} \leq y_J$ and let $\overline{g} = g \circ \phi_{\overline{y}}$ and $ \overline{h} = h \circ \phi_{\overline{y}}$. Integration by parts gives
\begin{equation*}
\begin{split}
&\modulo{\int_a^{b} (\overline{g} \circ \phi_{t} \circ \xi_J(s)) (\overline{h} \circ \xi_J(s)) \diff s} =\\
& = \modulo{\Big( \int_a^{b} \overline{g} \circ \phi_{t} \circ \xi_J(s) \diff s \Big) \overline{h}(b,y) - \int_a^{b}\Big( \int_a^{u} \overline{g} \circ \phi_{t} \circ \xi_J(s) \diff s \Big) (\partial_x\overline{h}\circ \xi_J(u)) \diff u} \\
& \leq \norma{\overline{h}}_{\infty} \modulo{\int_a^{b} \overline{g} \circ \phi_{t} \circ \xi_J(s) \diff s} + \norma{\partial_x \overline{h}}_{\infty} \misura(J) \sup_{a \leq u \leq b} \modulo{ \int_a^{u} \overline{g}\circ \phi_{t} \circ \xi_J(s) \diff s } \\
%& \qquad \leq (\norma{h}_{\infty}  +  \norma{\partial_x h}_{\infty} \misura(J)) \sup_{a \leq u \leq b} \modulo{ \int_a^{u} \overline{g} \circ \phi_{t} \circ \sigma_J(s) \diff s }. 
\end{split}
\end{equation*}
We have that $\norma{\overline{h}}_{\infty} = \norma{h}_{\infty}$ and, by \eqref{eq:pushforwder}, Theorem \ref{th:BS} and Proposition \ref{th:prelpart}-(iv), 
\begin{equation}\label{eq:dxhbar}
\begin{split}
\norma{\partial_x \overline{h}}_{\infty} &\leq \max_{(x,y) \in X(t)} \modulo{S_{r(x,\overline{y}+y)}(f\rq{})(x)} \norma{h}_{\mathscr{C}^1} \\
&= O \left( \max_{(x,y) \in X(t)} r(x, \overline{y}+y) \log r(x, \overline{y}+y) \right) = O( \log t \log \log t). 
\end{split}
\end{equation}
Since $\misura(J) \leq 2/(t (\log t)^{\alpha})$, we obtain
$$
\modulo{\int_a^{b} (\overline{g} \circ \phi_{t} \circ \xi_J(s)) (\overline{h} \circ \xi_J(s)) \diff s} = ( \norma{\overline{h}}_{\infty} + 1 )  \sup_{a \leq u \leq b} \modulo{ \int_a^{u} \overline{g}\circ \phi_{t} \circ \xi_J(s) \diff s }.
$$

The following is our bootstrap trick.

\begin{lemma}\label{th:bootstrap}
There exists $C>0$ such that 
$$
 \sup_{a \leq u \leq b} \modulo{ \int_a^{u} \overline{g} \circ \phi_{t} \circ \xi_J(s) \diff s } \leq \frac{C}{t \log t}  \sup_{a \leq u \leq b} \modulo{ \int_{\phi_t \circ \xi_{[a,u)}} \overline{g} \diff y}. 
$$
\end{lemma}
\begin{proof}
Fix $\varepsilon>0$ and let $a \leq \ell \leq b$,
\begin{equation*}
\begin{split}
\int_a^{\ell}  \overline{g} \circ \phi_{t} \circ \xi_J(s) \diff s =&\int_a^{\ell} ( \overline{g} \circ \phi_{t} \circ \xi_J(s) ) \Big( - \frac{S_{r(s,t)}(f\rq{})(s)}{(C\rq{} +\varepsilon)t \log t} \Big) \diff s  \\
&+ \int_0^{\ell} ( \overline{g} \circ \phi_{t} \circ \xi_J(s) ) \Big( 1+ \frac{S_{r(s,t)}(f\rq{})(s)}{(C\rq{} +\varepsilon)t \log t} \Big) \diff s. 
\end{split}
\end{equation*}
By \eqref{eq:tgtvect}, the first summand equals 
$$
\int_a^{\ell} ( \overline{g} \circ \phi_{t} \circ \xi_J(s) ) \Big( - \frac{S_{r(s,t)}(f\rq{})(s)}{(C\rq{} +\varepsilon)t \log t} \Big) \diff s=  \frac{1}{(C\rq{} +\varepsilon)t \log t} \int_{\phi_t \circ \xi_{[a,\ell)}} \overline{g} \diff y.
$$
%where, we recall, $\sigma_{\ell}$ is the path $\sigma_J$ restricted to $[a,\ell)$. 
Integration by parts of the second summand gives
\begin{equation*}
\begin{split}
& \int_a^{\ell} (\overline{g} \circ \phi_{t}  \circ \xi_J(s) ) \Big( 1+ \frac{S_{r(s,t)}(f\rq{})(s)}{(C\rq{} +\varepsilon)t \log t} \Big) \diff s \\
& = \Big( 1+ \frac{S_{r(\ell,t)}(f\rq{})(\ell)}{(C\rq{} +\varepsilon)t \log t}  \Big) \int_a^{\ell} \overline{g} \circ \phi_{t}  \circ \xi_J(s) \diff s  \\
& \qquad \qquad -  \int_a^{\ell} \frac{\diff}{\diff s}  \Big( 1+ \frac{S_{r(s,t)}(f\rq{})(s)}{(C\rq{} +\varepsilon)t \log t}  \Big) \Big( \int_a^{s} \overline{g} \circ \phi_{t}  \circ \xi_J(u) \diff u \Big)\diff s \\
& =\Big( 1+ \frac{S_{r(\ell,t)}(f\rq{})(\ell)}{(C\rq{} +\varepsilon)t \log t}  \Big)  \int_a^{\ell} \overline{g} \circ \phi_{t}  \circ \xi_J(s) \diff s  \\
& \qquad \qquad -  \int_a^{\ell} \Big( \frac{S_{r(s,t)}(f\rq{}\rq{})(s)}{(C\rq{} +\varepsilon)t \log t}  \Big) \Big( \int_a^{s} \overline{g} \circ \phi_{t}  \circ \xi_J(u) \diff u \Big) \diff s\\
%& \qquad -  \int_a^{\ell} \Big( \frac{S_{r(s,t)}(f\rq{}\rq{})(s)}{(C\rq{} +\varepsilon)t \log t}  \Big) \Big( \int_a^{s} \overline{g} \circ \phi_{t}  \circ \sigma_J(u) \diff u \Big) \diff s\\
\end{split}
\end{equation*}
Thus
\begin{equation*}
\begin{split}
& \modulo{ \int_a^{\ell} \overline{g} \circ \phi_{t}  \circ \xi_J(s) \diff s } \leq \frac{1}{(C\rq{} +\varepsilon)t \log t} \modulo{ \int_{\phi_t \circ \xi_{[a,\ell)}}\overline{g} \diff y} \\
 & \qquad + \modulo{1+ \frac{S_{r(\ell,t)}(f\rq{})(\ell)}{(C\rq{} +\varepsilon)t \log t}} \modulo{ \int_a^{\ell} \overline{g} \circ \phi_{t}  \circ \xi_J(s) \diff s } \\
& \qquad + \modulo{ \max_{a \leq u \leq \ell} \frac{S_{r(u,t)}(f\rq{}\rq{})(u)}{(C\rq{} +\varepsilon)t \log t} \cdot (\ell-a)} \sup_{a \leq u \leq \ell}\modulo{\int_a^{u} \overline{g} \circ \phi_{t} \circ \xi_J(s) \diff s}
\end{split}
\end{equation*}
By Proposition \ref{th:stretpart}-(ii),(iv) and $\ell -a\leq b-a \leq 2/(t(\log t)^{\alpha})$, we get 
\begin{multline*}
 \modulo{ \int_a^{\ell} \overline{g} \circ \phi_{t}  \circ \xi_J(s) \diff s }  \leq \frac{1}{(C\rq{} +\varepsilon)t \log t} \modulo{ \int_{\phi_t \circ \xi_{[a,\ell)}} g \circ \phi_y \diff y}  \\
+ \Big( 1-\frac{C\rq{}}{C\rq{} +\varepsilon} + \frac{C\rq{}\rq{}}{(C\rq{} +\varepsilon)M} \Big) \sup_{a \leq u \leq \ell}\modulo{\int_a^{u} \overline{g} \circ \phi_{t}  \circ \xi_J(s) \diff s}.
\end{multline*}
Since this is true for any $a \leq \ell \leq b$, we can consider the supremum on both sides and, after rearranging the terms, 
$$
\Big( C\rq{} - \frac{C\rq{}\rq{}}{M} \Big) \sup_{a \leq u \leq b}\modulo{\int_a^{u} \overline{g} \circ \phi_{t}  \circ \xi_J(s) \diff s}\leq \frac{1}{t \log t} \sup_{a \leq u \leq b} \modulo{ \int_{\phi_t \circ \xi_{[a,u)}} \overline{g} \diff y}. 
$$
The conclusion follows by choosing $M>1$ so that $C^{-1} = C\rq{} - C\rq{}\rq{}/M >0$.
\end{proof}

We now compare the integral of $\overline{g}$ along the curve $\phi_t \circ \xi_{[a,u)}$ with the integral of $\overline{g}$ along the orbit segment starting from $\phi_t(a,0)$ of length $\Delta t(u)$.

\begin{figure}[h]
\centering
\begin{tikzpicture}[scale=2.8]
\clip (0,-0.2) rectangle (4,3);
\draw (0,0) -- (4,0);
\draw [-] (0,0.8) to [bend right = 17] (0.56,2.8);
%\draw (0.62,0) node [anchor=north] {$a_1$};
\draw [-] (1,1.4) to [bend left = 10] (0.67,2.8);
\draw [-] (1,1.4) to [bend right = 30] (1.2,1.3);
\draw [-] (1.2,1.3) to [bend right = 25] (1.4,1.5);
\draw [-] (1.4,1.5) to [bend right = 10] (1.62,2.8);
%\draw (1.65,0) node [anchor=north] {$a_2$};
\draw [-] (1.67,2.8) to [bend right = 14] (1.8,2.2);
\draw [-] (1.8,2.2) to [bend right = 35] (2.3,2);
\draw [-] (2.3,2) to [bend right = 25] (2.4,2.1);
\draw [-] (2.4,2.1) to [bend right = 10] (2.5,2.8);
%\draw (2.5,0) node [anchor=north] {$a_3$};
\draw [-] (2.55,2.8) to [bend right = 5] (2.7,1.5);
\draw [-] (2.7,1.5) to [bend right = 25] (2.8,1.4);
\draw [-] (2.8,1.4) to [bend right = 20] (3,1.4);
\draw [-] (3,1.4) to [bend right = 30] (3.3,1.7);
\draw [-] (3.3,1.7) to [bend right = 5] (3.4,2.8);
%\draw (3.4,0) node [anchor=north] {$a_4$};
\draw [-] (3.5,2.8) to [bend right = 10] (4,1);
\draw [-] (1,0.6) to [bend right = 10] (1.15,1.3);
\draw [-] (1.9,0) to [bend right = 5] (2.1,1.98);
\draw [-] (2.9,0) to [bend right = 5] (3.1,1);
\draw [-] (3.7,0) to [bend right = 5] (3.9,1.22);
\draw (1,0.6) node [anchor=north] {$\phi_t(a,0)$};
\draw (1.1,0.8) node [anchor=west] {$\phi_t \circ \xi_1$};
\draw [red] [-] (1,0.6) to (1,1.4);
\draw [red] (1,1) node [anchor=east] {$\gamma_1$};
\draw [red] [-] (1.8,0) to (1.8,2.2);
\draw [red] (1.8,1.5) node [anchor=east] {$\gamma_i$};
\draw [red] [-] (3.6,0) to (3.6,2.1);
\draw [red] [-] (2.7,0) to (2.7,1);
\draw [red] (2.7,0.5) node [anchor=east] {$\gamma_{N(u)}$};
\draw (2.95,0.2) node [anchor=west] {$\phi_t \circ \xi_{N(u)}$};
\draw (2.1,1.5) node [anchor=west] {$\phi_t \circ \xi_i$};
\draw (3.2,1) node [anchor=south] {$\phi_t(b,0)$};
\draw [red] (2.6,1) node [anchor=south] {$\phi_{t+\Delta t(u)}(a,0)$};
\draw [red] (1.55,0) node [anchor=north] {$T^{r(a,t)+i}a$};
\draw (2.15,0) node [anchor=north] {$T^{r(a,t)+i}u_i$};
\end{tikzpicture}
\caption{The curve $\phi_t\circ \xi_{[a,u)}$ splits into $N(u)$ curves $\phi_t \circ \xi_i$. In red, the orbit segment $\gamma$.}
\label{fig:2}
\end{figure}
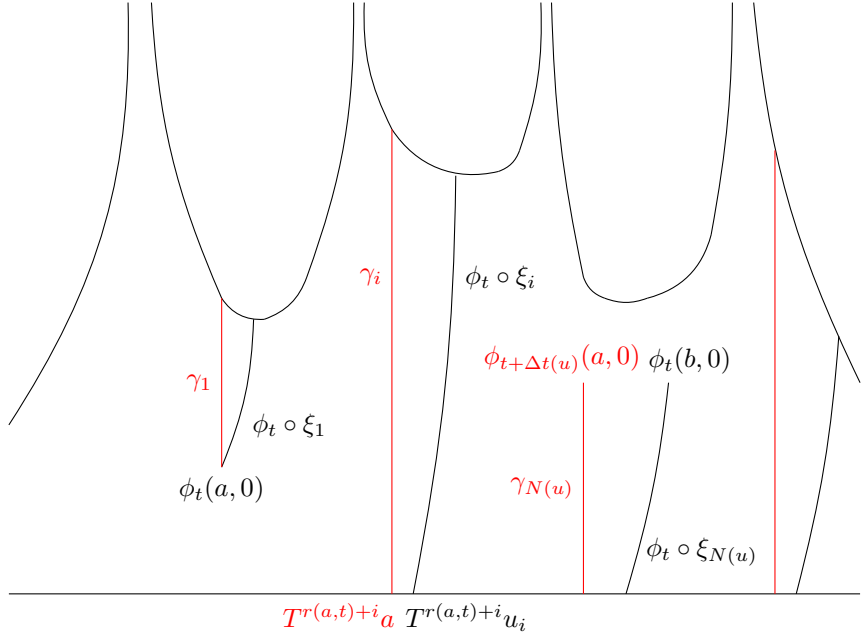

\begin{lemma}
Let $\gamma (s) = \phi_{t+s}(a,0)$, $0 \leq s < \Delta t(u)$, be the orbit segment of length $\Delta t(u)$ starting from $\phi_t(a,0)$. We have 
\begin{equation}\label{eq:approx}
\modulo{\int_{\phi_t \circ \xi_{[a,u)}}\overline{g} \diff y} \leq \modulo{\int_{\gamma} \overline{g} \diff y} + O\left( (\log t)^{-1} \right).
\end{equation}
\end{lemma} 
\begin{proof}
For all $1 \leq i \leq N(u)$, we compare the integral of $\overline{g}$ along the curve $\phi_t \circ \xi_i$ with the integral of $\overline{g}$ along an appropriate orbit segment. If $i \neq 1, N(u)$, consider $\gamma_i (s) = \phi_s(T^{r(a,t)+i}a,0)$, for $0 \leq s <  f(T^{r(a,t)+i}a)$; define also $\gamma_1(s) = \phi_{t+s}(a,0)$, for $0\leq s < S_{r(a,t)+1}(f)(a)-t$ and $\gamma_{N(u)}(s) = \phi_s(T^{r(a,t)+N(u)}a,0)$, for $0 \leq s <  t-S_{r(u,t)}(f)(u)$. Fix $0 \leq i \leq N(u)$ and join the starting points of $\phi_t \circ \xi_i$ and $\gamma_i$ by an horizontal segment and the end points by the curve $\zeta_i(s) = (T^{r(a,t)+i}s, f(T^{r(a,t)+i}s))$, $a \leq s \leq u_{i+1}$, if $i\neq N(u)$ and by another horizontal segment, if $i=N(u)$. See Figure \ref{fig:2}.

We remark that the integral over any horizontal segment of $\overline{g} \diff y$ is zero. By Green\rq{}s Theorem, 
\begin{equation}\label{eq:gthm}
\modulo{\int_{\phi_t \circ \xi_i} \overline{g} \diff y - \int_{\gamma_i} \overline{g} \diff y} \leq \modulo{\int_{\zeta_i} \overline{g} \diff y} + \norma{\partial_x \overline{g}}_{\infty} \int_{T^{r(a,t)+i}a}^{T^{r(a,t)+i}u_{i+1}} f(x) \diff x.
\end{equation}
Since $r(a,t) + i \leq r(b,t) \leq R(t)$, by Proposition \ref{th:prelpart}-(i), $T^{r(a,t)+i}$ is an isometry, hence
\begin{equation*}
\begin{split}
\int_{T^{r(a,t)+i}a}^{T^{r(a,t)+i}u_{i+1}} f(x) \diff x &\leq \norma{f}_2 \misura([T^{r(a,t)+i}a, T^{r(a,t)+i}u_{i+1} ])^{1/2}\\
& \leq \frac{2\norma{f}_2 }{(t (\log t)^{\alpha})^{1/2}}.
\end{split}
\end{equation*}
Reasoning as in \eqref{eq:dxhbar}, $\norma{\partial_x \overline{g}}_{\infty} =O(\log t \log \log t)$, thus the second term in \eqref{eq:gthm} is $O \left( (\log t)^{2-\alpha/2}/t^{1/2} \right)$.
Moreover, by \eqref{eq:Nu} we can apply Proposition \ref{th:mixpart} to deduce $f\rq{}(T^{r(a,t)+i}x) =O\left( (\log t)^2 \right)$, so that
$$
\modulo{\int_{\zeta_i} \overline{g} \diff y} \leq  \norma{ \overline{g}}_{\infty} \int_{a}^{u_{i+1}} \modulo{f\rq{}(T^{r(a,t)+i}x)} \diff x  = O \left( \frac{(\log t)^2}{t (\log t)^{\alpha}} \right).
$$
Summing over all $i=0,\dots,N(u)$ we conclude using \eqref{eq:Nu}
\begin{equation*}
\begin{split}
&\modulo{\int_{\phi_t \circ \xi_{[a,u)}} \overline{g} \diff y - \int_{\gamma} \overline{g} \diff y} \leq \sum_{i=0}^{N(u)} \left(\modulo{\int_{\zeta_i} \overline{g} \diff y} + \norma{\partial_x \overline{g}}_{\infty} \int_{T^{r(a,t)+i}a}^{T^{r(a,t)+i}u_{i+1}} f(x) \diff x\right)\\
& \qquad \qquad  = N(u) O\left( \frac{(\log t)^2}{t (\log t)^{\alpha}} + \frac{ (\log t)^{2-\alpha/2}}{t^{1/2}}  \right) = O\left((\log t)^{-1} \right).
\end{split}
\end{equation*}
\end{proof}

By definition, the integral of $\overline{g}$ along the orbit segment $\gamma$ equals the integral of $g$ along $\phi_{\overline{y}}\circ \gamma$. 
The latter can be expressed as a Birkhoff sum of $\mathcal{I}g= \int_0^{f(x)} g(x,y) \diff y$ (see \eqref{eq:idig}) plus an error term arising from the initial and final point of the orbit segment $\phi_{\overline{y}}\circ \gamma$, namely, recalling the definition $T_t(x) = T^{r(x,t)}x$, 
\begin{multline*}
 \modulo{ \int_{\gamma} \overline{g} \diff y} = \modulo{ \int_{\phi_{\overline{y}} \circ \gamma} g \diff y} \leq S_{r(T_{t+\overline{y}}(a),\Delta t(u))}(\mathcal{I}g)(T_{t+\overline{y}}(a)) \\
 + \norma{g}_{\infty} (f(T_{t+\overline{y}}a) + f(T_{t+\overline{y} + \Delta t(u)}a) ).
\end{multline*}
We recall from Remark \ref{remark:3} that $\mathcal{I}g$ satisfies the hypotheses of Corollary \ref{th:athreyaforni}. We claim that 
\begin{equation}\label{miclaim}
f(T^{r(a,t+\overline{y})}a) + f(T^{r(a,t+\overline{y} + \Delta t(u))}a) = O(\log \log t).
\end{equation}
Indeed, by the cocycle relation for Birkhoff sums we have
\begin{equation*}
\begin{split}
&S_{r(a,t)+ \lfloor (\overline{y} + \Delta t(u))/m \rfloor + 2}(f)(a) \\
& \qquad \qquad  = S_{r(a,t)+1}(f)(a) + S_{ \lfloor (\overline{y} + \Delta t(u))/m \rfloor + 1}(f)(T^{r(a,t)+1}a)\\
& \qquad \qquad > t + ( \lfloor (\overline{y} + \Delta t(u))/m \rfloor + 1)m > t + \overline{y} + \Delta t(u);
\end{split}
\end{equation*}
hence, 
$$
r(a,t) \leq r(a, t + \overline{y}) \leq r(a, t + \overline{y} + \Delta t(u) ) \leq r(a,t)+ \lfloor (\overline{y} + \Delta t(u))/m \rfloor + 2.
$$
By Proposition \ref{th:prelpart}-(iv), $\overline{y} \leq C_f \log t$; hence, by \eqref{eq:deltatu}, the latter summand above is bounded by $r(a,t) + \frac{2C_f}{m} \log t$, up to enlarging $t_2$.
Proposition \ref{th:mixpart} yields the claim~\eqref{miclaim}.

Therefore, by \eqref{miclaim}, Corollary \ref{th:athreyaforni} and \eqref{eq:vertstretch},
\begin{equation}
\begin{split}
 \modulo{ \int_{\gamma} \overline{g} \diff y}\leq & S_{r(T_{t+\overline{y}}(a),\Delta t(u))}(\mathcal{I}g)(T_{t+\overline{y}}(a))  + O( \log \log t )\\
 = & O \left( (r(T_{t+\overline{y}}(a),\Delta t(u)))^{\theta} + \log \log t \right) = O\left((\Delta t(u))^{\theta} + \log \log t \right)\\
  = & O \left((\log t)^{\theta(1-\alpha)} +\log \log t \right) =O \left( (\log t)^{\theta(1-\alpha)} \right). \label{eq:atfor}
\end{split}
\end{equation}
From Lemma \ref{th:bootstrap}, \eqref{eq:approx} and \eqref{eq:atfor}, we obtain
\begin{equation*}
\begin{split}
&\sup_{a \leq u \leq b}\modulo{\int_a^{u} \overline{g} \circ \phi_{t}  \circ \xi_J(s) \diff s} \leq \frac{C}{t \log t} \sup_{a \leq u \leq b} \modulo{ \int_{\phi_t \circ \xi_{[a,u)}} \overline{g} \diff y}\\
& \qquad \qquad  \leq \frac{C}{t \log t} \left(\modulo{ \int_{\gamma} \overline{g}  \diff y} + O \left( (\log t)^{-1} \right) \right) = O \left( \frac{(\log t)^{\theta(1-\alpha)}}{t \log t}\right).% \\
%&\leq  \frac{\const}{t \log t} \left((\log t)^{\theta(1-\alpha)} + \frac{1}{(\log t)^{\alpha}} \right) \leq \const \frac{(\log t)^{\theta(1-\alpha)}}{t \log t}.
\end{split}
\end{equation*}
From \eqref{eq:fubini1}, we deduce
\begin{equation*}
\begin{split}
\modulo{\int_{X(t)} (g \circ \phi_t)h \diff \misura } &= O \left( \frac{(\log t)^{\theta(1-\alpha)}}{t \log t} \right) \sum_{J \in \mathcal{P}_f(t)} \int_0^{y_J} \frac{\misura(J)}{\misura(J)} \diff y \\
& = O \left( \frac{(\log t)^{\theta(1-\alpha)}}{t \log t} (t(\log t)^{\alpha})  \right) \sum_{J \in \mathcal{P}_f(t)} \int_0^{y_J}  \misura(J) \diff y \\
& = O \left(\frac{1}{(\log t)^{(1-\theta)(1-\alpha)}} \right),
\end{split}
\end{equation*}
which, combined with \eqref{eq:dcxt}, concludes the proof.

%%%%%%%%%%%%%%%%%%%%
%%%%%%%%%%%%%%%%%%%%
%%%%%%%%%%%%%%%%%%%%

\section{Appendix: estimates of Birkhoff sums}\label{section5bs}

In this appendix we will prove the bounds on the Birkhoff sums of the roof function $f$ and of its derivatives $f\rq{}$ and $f\rq{}\rq{}$ in Theorem \ref{th:BS}. The proof is a generalization to the case of finitely many singularities of a result by Ulcigrai \cite[Corollaries 3.4, 3.5]{ulcigrai:mixing}.

We first consider the auxiliary functions $u_k,v_k, \widetilde{u}_k, \widetilde{v}_k$ introduced in \S\ref{section:quantitative}.

\subsection{Special Birkhoff sums}\label{gaperror}

Fix $\varepsilon\rq{}>0$ and $w$ and $\widetilde{w}$ to be either $u_k$ or $v_k$ and either $\widetilde{u}_k$ or $\widetilde{v}_k$ respectively for fixed $k$. Let $\overline{l}, D, D\rq{}$ be given by Theorem \ref{th:ulcigrai}; for $\varepsilon>0$ (which will be determined later) choose $L_1, L_2 \in \N$ such that $D^{L_1} D\rq{} < \varepsilon$ and $ \nu (d-1)^{-L_2} < \varepsilon$. Assume $l_0 \geq \overline{l}(1+L_1+L_2)$ and introduce the past steps
$$
l_{-1} := l_0 - L_1 \overline{l}, \qquad l_{-2} = l_0 - (L_1+L_2) \overline{l}.
$$
Consider a point $x_0 \in I^{(n_{l_{0}})}_{j_0} \subset I^{(n_{l_{0}})}$; we want to estimate the Birkhoff sums of $w$ and $\widetilde{w}$ at $x_0$ along $Z_{j_0}^{(n_{l_{0}})}$, namely the sums
$$
S_{r_0}(w)(x_0) = \sum_{i=0}^{r_0-1} w(T^ix_0), \quad \text{ and } \quad S_{r_0}(\widetilde{w})(x_0) = \sum_{i=0}^{r_0-1} \widetilde{w}(T^ix_0), 
$$
where $r_0:=h_{j_0}^{(n_{l_{0}})}$. Sums of this type will be called \newword{special Birkhoff sums}. We will prove that 
\begin{equation}\label{eq:sbsweakforf}
S_{r_0}(w)(x_0) \leq  (1+\varepsilon\rq{}) r_0 \int_0^1w(x) \diff x + \max_{0 \leq i < r_0}w(T^ix_0).
\end{equation}
and
\begin{equation}\label{eq:sbsweak}
 (1-\varepsilon\rq{}) r_0 \log h^{(n_{l_{0}})} \leq S_{r_0}(\widetilde{w})(x_0) \leq  (1+\varepsilon\rq{}) r_0 \log h^{(n_{l_{0}})} + \max_{0 \leq i < r_0} \widetilde{w}(T^ix_0),
\end{equation}
where, we recall, $h^{(n_{l_{0}})} = \max \{ h_{j}^{(n_{l_{0}})} : 1 \leq j \leq d \}$.

By Remark \ref{remark:endpoint}, at each step $n$ the singularity $a_k$ of $w$ and of $\widetilde{w}$ belongs to the boundary of two adjacent elements of the partition $\mathcal{Z}^{(n)}$ defined in \S\ref{section:def}. Denote by $F^{(n)}_{\text{sing}}$ the element of $\mathcal{Z}^{(n)}$ which has $a_k$ as \newword{left} endpoint if $w=u_k$ or as \newword{right} endpoint if $w=v_k$, and similarly when we consider $\widetilde{w}$ instead of $w$. 
Outside $F^{(n)}_{\text{sing}}$ the value of $w$ is bounded by $1-\log \lambda^{(n)}_{\text{sing}} $ and the value of $\widetilde{w}$ is bounded by $1/\lambda^{(n)}_{\text{sing}}$, where $\lambda^{(n)}_{\text{sing}}$ is the length of $F^{(n)}_{\text{sing}}$.
Remark that, by construction, $F^{(n)}_{\text{sing}} \subset F^{(m)}_{\text{sing}}$ for $n >m$; decompose the initial interval $I=I^{(0)}$ into the three pairwise disjoint sets $I^{(0)} = A \sqcup B \sqcup C$, with
\begin{equation*}
A=  F^{(n_{l_0})}_{\text{sing}}, \quad B = F^{(n_{l_{-2}})}_{\text{sing}} \setminus  F^{(n_{l_0})}_{\text{sing}}, \quad C= I^{(0)} \setminus  F^{(n_{l_{-2}})}_{\text{sing}}. \label{eq:decomposition}
\end{equation*}
Using the partition above, we can write
\begin{equation}\label{eq:Birkhoffsum}
S_{r_0}(w)(x_0) = \sum_{T^ix_0 \in A} w(T^ix_0) + \sum_{T^ix_0 \in B} w(T^ix_0)+ \sum_{T^ix_0 \in C} w(T^ix_0), 
\end{equation}
and similarly for $\widetilde{w}$.
Notice that the first summand is not zero if and only if there exists $r \leq r_0$ such that $T^rx_0 \in F^{(n_{l_0})}_{\text{sing}}$, i.e.~if and only if $F^{(n_{l_0})}_{\text{sing}} \subset Z^{(n_{l_0})}_{j_0}$; in this case it equals $w(T^rx_0)$.

We refer to the summands in \eqref{eq:Birkhoffsum} as \newword{singular term}, \newword{gap error} and \newword{main contribution} respectively.

\medskip

\paragraph{Gap error} We first consider $\widetilde{w}$.
Let $b = \# \{T^ix_0 \in B\}$; we will approximate the gap error with the sum of $\widetilde{w}$ over an arithmetic progression of length $b$. 
For any $T^ix_0 \in B$ we have $\widetilde{w}(T^ix_0) \leq 1/ \lambda_{\text{sing}}^{(n_{l_{0}})}$ and, since $T^ix_0$ and $T^jx_0$ belong to different elements of $\mathcal{Z}^{(n_{l_0})}$ when $i \neq j$, for $i,j \leq r_0$ also $\modulo{T^ix_0 - T^jx_0} \geq \lambda^{(n_{l_0})}_{j_0} \geq (d \kappa \nu r_0 )^{-1} $ by Corollary \ref{cor:thulcigrai}-(i). Up to rearranging the sequence $\{T^ix_0 \in B : 0 \leq  i < r_0 \}$ in increasing order of $T^ix_0 - a_k$ if $\widetilde{w}=\widetilde{u}_k$ (decreasing, if $\widetilde{w}=\widetilde{v}_k$) and calling it $x_i$, we have
$$
x_i \geq \lambda^{(n_{l_0})}_{\text{sing}} + \frac{i}{d  \kappa \nu r_0 }.
$$
By monotonicity of $\widetilde{w}$ it follows that
$$
0 \leq \sum_{T^ix_0 \in B} \widetilde{w}(T^ix_0) = \sum_{T^ix_0 \in B} \frac{1}{x_i} \leq \sum_{i=0}^{b} \Bigg( \lambda^{(n_{l_0})}_{\text{sing}} + \frac{i}{d \kappa \nu r_0} \Bigg)^{-1}.
$$
Using the trivial fact that for any continuous and decreasing function $h$, $\sum_{i=0}^b h(i) \leq h(0) + \int_0^b h(x) \diff x$ and $ d \kappa \nu r_0 \lambda^{(n_{l_0})}_{\text{sing}} \geq 1$ by Corollary \ref{cor:thulcigrai}-(i), we get
\begin{equation*}
\begin{split}
0 &\leq \sum_{T^ix_0 \in B} \widetilde{w}(T^ix_0) \leq \frac{1}{\lambda_{\text{sing}}^{(n_{l_{0}})}} + \int_0^b  \Bigg( \lambda^{(n_{l_0})}_{\text{sing}} + \frac{x}{d  \kappa \nu r_0} \Bigg)^{-1} \diff x\\
& \leq d \kappa \nu r_0 + d  \kappa \nu r_0 \log \Bigg( 1+ \frac{b}{d \kappa \nu r_0 \lambda^{(n_{l_0})}_{\text{sing}}} \Bigg) \leq d \kappa \nu r_0 (1+ \log(b+1)).
\end{split}
\end{equation*}
Since $B \subset F_{\text{sing}}^{(n_{l_{-2}})}$, we have that $b \leq \# \{ T^ix_0 \in Z_{j_0}^{(n_{l_0})} \cap F_{\text{sing}}^{(n_{l_{-2}})} \}$. 
Let $\alpha \in \{1, \dots, d\}$ be such that $ F_{\text{sing}}^{(n_{l_{-2}})} \subset  Z_{\alpha}^{(n_{l_{-2}})}$; the number of $T^ix_0 \in Z_{j_0}^{(n_{l_0})}$ contained in $F_{\text{sing}}^{(n_{l_{-2}})}$ equals the number of those contained in $I_{\alpha}^{(n_{l_{-2}})}$.
Thus, by Lemma \ref{lemma:entries},
\begin{equation}\label{eq:bnormaA}
b \leq \# \{ T^ix_0 \in Z_{j_0}^{(n_{l_0})} \cap I_{\alpha}^{(n_{l_{-2}})} \} = A_{\alpha, j_0}^{(n_{l_{-2}}, n_{l_0})} \leq \norma{A^{(n_{l_{-2}}, n_{l_0})} }.
\end{equation}
From the asymptotic behavior (iii) in Corollary \ref{cor:thulcigrai}, we obtain
\begin{equation*}
\frac{ \sum_{T^ix_0 \in B} \widetilde{w}(T^ix_0)}{ r_0 \log  h^{(n_{l_0})}} \leq \frac{ d \kappa \nu r_0 (1+ \log( \norma{A^{(n_{l_{-2}}, n_{l_0})} }+1)) }{ r_0 \log  h^{(n_{l_0})}} \to 0,
\end{equation*}
so, for $l_0$ large enough, we conclude
\begin{equation}\label{eq:gerr}
0 \leq \sum_{T^ix_0 \in B} \widetilde{w}(T^ix_0) \leq \varepsilon ( r_0 \log  h^{(n_{l_0})}) .
\end{equation}

We can carry out analogous computations for $w$. In this case,
\begin{equation*}
\begin{split}
0 &\leq  \sum_{T^ix_0 \in B} w(T^ix_0) = \sum_{T^ix_0 \in B}(1 -\log T^ix_0) \leq b(1-\log \lambda_{\text{sing}}^{(n_{l_{0}})}) = O ( b \log r_0).
\end{split}
\end{equation*}
Corollary \ref{cor:thulcigrai}-(ii) implies that $l_0 = O( \log r_0)$; hence by \eqref{eq:bnormaA}, the Diophantine condition in Theorem \ref{th:ulcigrai}-(iv) and the definition of $l_{-2}$ we obtain
$$
b \leq \norma{A^{(n_{l_{-2}}, n_{l_0})} } \leq l_0^{(L_1+L_2)\overline{l} \tau} = O \left( (\log r_0)^{(L_1+L_2)\overline{l} \tau} \right). 
$$
In particular, for $l_0$ large enough we conclude
\begin{equation}\label{eq:gerrforw}
0 \leq \sum_{T^ix_0 \in B} w(T^ix_0) \leq \varepsilon r_0.
\end{equation}

\medskip

\paragraph{Main contribution.} Consider the partition $\mathcal{Z}^{(n_{l_{-1}})}$ restricted to the set $C$. We will exploit the fact that the partition elements are nicely distributed in $\mathcal{Z}^{(n_{l_0})}$ to approximate the special Birkhoff sum of $w$ and $\widetilde{w}$ by the respective integrals over $C$, and then bound the latters.

For any $F_{\alpha} \in \mathcal{Z}^{(n_{l_{-1}})} \cap C$, $F_{\alpha} \subset Z_{j_{\alpha}}^{(n_{l_{-1}})}$ with $j_{\alpha} \in \{1, \dots, d\}$, choose points $\overline{x}_{\alpha}, \widetilde{x}_{\alpha} \in F_{\alpha}$ given by the Mean-Value Theorem, namely such that 
$$
w(\overline{x}_{\alpha}) = \frac{1}{\lambda_{\alpha}^{(n_{l_{-1}})}} \int_{F_{\alpha}}w(x) \diff x, \qquad \widetilde{w}(\widetilde{x}_{\alpha}) = \frac{1}{\lambda_{\alpha}^{(n_{l_{-1}})}} \int_{F_{\alpha}}\widetilde{w}(x) \diff  x,
$$
with $\lambda_{\alpha}^{(n_{l_{-1}})} = \misura (F_{\alpha})$. We now show that for any $T^ix_0 \in F_{\alpha}$,
\begin{equation}
1-\varepsilon \leq \frac{w(T^ix_0)}{w(\overline{x}_{\alpha})} \leq 1+ \varepsilon, \qquad 1-\varepsilon \leq \frac{\widetilde{w}(T^ix_0)}{\widetilde{w}(\widetilde{x}_{\alpha})} \leq 1+ \varepsilon. \label{eq:meanvalue}
\end{equation}
Since $w \geq 1$ and for all $x \in F_{\alpha}\subset C$ we have $\modulo{x-a_k} \geq \lambda_{\text{sing}}^{(n_{l_{-2}})}$, again by the Mean-Value Theorem we have
$$
\modulo{\frac{w(T^ix_0)}{w(\overline{x}_{\alpha})}-1 } \leq \modulo{ \max_C w\rq{}} \lambda_{\alpha}^{(n_{l_{-1}})} \leq \frac{\lambda_{\alpha}^{(n_{l_{-1}})}}{\lambda_{\text{sing}}^{(n_{l_{-2}})}}.
$$
Considering $\widetilde{w}$, up to replacing $F_{\alpha}$ with $F_{\alpha}+1$ or $F_{\alpha}-1$, we can suppose that $\widetilde{w}(x) = 1/ \modulo{x-a_k}$ for $x \in F_{\alpha}$. Then, 
\begin{equation*}
\frac{\widetilde{w}(T^ix_0)}{\widetilde{w}(\widetilde{x}_{\alpha})} = \modulo{ \frac{\widetilde{x}_{\alpha}-a_k}{T^ix_0-a_k} } \leq \frac{\sup_{x \in F_{\alpha}} \modulo{x-a_k}}{\inf_{x \in F_{\alpha}} \modulo{x-a_k}} = 1+ \frac{\lambda_{\alpha}^{(n_{l_{-1}})}}{\inf_{x \in F_{\alpha}} \modulo{x-a_k}} \leq 1+ \frac{\lambda_{\alpha}^{(n_{l_{-1}})}}{\lambda_{\text{sing}}^{(n_{l_{-2}})}},
\end{equation*}
and similarly
\begin{equation*}
\frac{\widetilde{w}(T^ix_0)}{\widetilde{w}(\widetilde{x}_{\alpha})} = \modulo{ \frac{\widetilde{x}_{\alpha}-a_k}{T^ix_0-a_k} } \geq \frac{\inf_{x \in F_{\alpha}} \modulo{x-a_k}}{\sup_{x \in F_{\alpha}} \modulo{x-a_k}} = 1 - \frac{\lambda_{\alpha}^{(n_{l_{-1}})}}{\sup_{x \in F_{\alpha}} \modulo{x-a_k}} \geq 1- \frac{\lambda_{\alpha}^{(n_{l_{-1}})}}{\lambda_{\text{sing}}^{(n_{l_{-2}})}}.
\end{equation*}
Thus, it is sufficient to prove that ${\lambda_{\alpha}^{(n_{l_{-1}})}}/{\lambda_{\text{sing}}^{(n_{l_{-2}})}}< \varepsilon$. The length vectors are related by the cocycle property \eqref{eq:cocycleprop}, namely, by the definition of $l_{-2}$,
$$
\underline{\lambda}^{(n_{l_{-2}})}=A^{(n_{l_{-2}}, n_{l_{-1}})} \underline{\lambda}^{(n_{l_{-1}})} = \prod_{j=0}^{L_2-1} A^{(n_{l_{-2}+j \overline{l}} , n_{l_{-2}+(j+1) \overline{l}} )} \underline{\lambda}^{(n_{l_{-1}})},
$$
and each of those $d \times d$ matrices is strictly positive with integer coefficients by (iii) in Theorem \ref{th:ulcigrai}. Therefore
$$
{\lambda}^{(n_{l_{-2}})}_{\text{sing}} \geq d^{L_2} \min_j {\lambda}^{(n_{l_{-1}})}_j \geq \frac{d^{L_2}}{\nu}{\lambda}^{(n_{l_{-1}})}_{\alpha},
$$
which implies ${\lambda}^{(n_{l_{-1}})}_{{\alpha}}/ {\lambda}^{(n_{l_{-2}})}_{\text{sing}} \leq \nu d^{-L_2} <\varepsilon$ by the choice of $L_2$. Hence the claim \eqref{eq:meanvalue} is now proved.

\medskip 

Rewriting
$$
\sum_{T^ix_0\in C} w(T^ix_0) = \sum_{Z_{\alpha} \subset C}\sum_{T^ix_0 \in F_{\alpha}} w(T^ix_0),
$$
we get from \eqref{eq:meanvalue}
\begin{multline*}
(1-\varepsilon) \sum_{F_{\alpha} \subset C} \# \{ T^ix_0 \in F_{\alpha} \} w(\overline{x}_{\alpha}) \leq \sum_{T^ix_0 \in C} w(T^ix_0)\\
 \leq (1+\varepsilon) \sum_{F_{\alpha} \subset C} \# \{ T^ix_0\in F_{\alpha} \} w(\overline{x}_{\alpha}).
\end{multline*}
Exactly as in the previous paragraph, $ \# \{ T^ix_0 \in F_{\alpha} \}  =  \# \{ T^ix_0 \in I_{j_{\alpha}}^{(n_{l_{-1}})}\} = A_{j_{\alpha}, j_0}^{(n_{l_{-1}}, n_{l_{0}})}$. 
We apply the following lemma by Ulcigrai.
\begin{lemma}[{\cite[Lemma 3.4]{ulcigrai:mixing}}]
For each $1 \leq i,j \leq d$,
$$
e^{-2D^{L_1}D\rq{}} \lambda_{i}^{(n_{l_{-1}})} \leq \frac{ A_{i,j}^{(n_{l_{-1}}, n_{l_{0}})} }{ h_{j}^{(n_{l_{0}})} } \leq e^{2D^{L_1}D\rq{}} \lambda_{i}^{(n_{l_{-1}})}.
$$
\end{lemma}
By the initial choice of $L_1$, this implies that $e^{-2 \varepsilon}  \lambda_{j_{\alpha}}^{(n_{l_{-1}})} r_0 \leq A_{j_{\alpha}, j_0}^{(n_{l_{-1}}, n_{l_{0}})} \leq e^{2 \varepsilon}  \lambda_{j_{\alpha}}^{(n_{l_{-1}})} r_0$. We get
\begin{equation}
\begin{split}
&\sum_{T^ix_0 \in C} w(T^ix_0 ) \leq  (1+\varepsilon) \sum_{F_{\alpha} \subset C}  A_{j_{\alpha}, j_0}^{(n_{l_{-1}}, n_{l_{0}})} w(\overline{x}_{\alpha}) \\
& \qquad \leq e^{2\varepsilon} (1+\varepsilon) \sum_{F_{\alpha} \subset C}  \lambda_{j_{\alpha}}^{(n_{l_{-1}})} r_0 w(\overline{x}_{\alpha}) = e^{2\varepsilon} (1+\varepsilon)  r_0 \sum_{F_{\alpha} \subset C} \int_{F_{\alpha}} w(x) \diff x \\
&\qquad  = e^{2\varepsilon} (1+\varepsilon) r_0 \int_{C} w(x) \diff x.\label{eq:upperboundforw}
\end{split}
\end{equation}

The same computations can be carried out for $\widetilde{w}$, obtaining
\begin{equation}
e^{-2\varepsilon} (1-\varepsilon) r_0 \int_{C} \widetilde{w}(x) \diff x \leq \sum_{T^ix_0 \in C} \widetilde{w}(T^ix_0) \leq e^{2\varepsilon} (1+\varepsilon) r_0 \int_{C} \widetilde{w}(x) \diff x  .\label{eq:lowerbound}
\end{equation}
Recalling $C= I^{(0)} \setminus Z_{\text{sing}}^{(n_{l_{-2}})}$, we have to estimate the integral
$$
 \int_{I^{(0)} \setminus Z_{\text{sing}}^{(n_{l_{-2}})} } \widetilde{w}(x) \diff x =\log  \frac{1}{\lambda_{\text{sing}}^{(n_{l_{-2}})}}.
$$
Since $\lambda_{\text{sing}}^{(n_{l_{-2}})} \geq \lambda_{\text{sing}}^{(n_{l_{0}})}  \geq 1/(d\kappa \nu h^{(n_{l_{0}})})$ by Corollary \ref{cor:thulcigrai}-(i), we have the upper bound
\begin{equation}
\log \frac{1}{\lambda_{\text{sing}}^{(n_{l_{-2}})}} \leq \log (d\kappa \nu h^{(n_{l_{0}})}) = \left( 1+ \frac{\log (d \kappa \nu)}{\log h^{(n_{l_{0}})}} \right) \log h^{(n_{l_{0}})} \leq (1+ \varepsilon) \log h^{(n_{l_{0}})}, \label{eq:integral1}
\end{equation}
for $l_0$ sufficiently large. On the other hand, adding and subtracting $\log h^{(n_{l_{0}})}$, we obtain the lower bound
\begin{equation}
\begin{split}
\log  \frac{1}{\lambda_{\text{sing}}^{(n_{l_{-2}})}} \pm \log h^{(n_{l_{0}})} & = \log h^{(n_{l_{0}})} \left( 1- \frac{\log (h^{(n_{l_{0}})}\lambda_{\text{sing}}^{(n_{l_{-2}})})}{ \log h^{(n_{l_{0}})} } \right) \\
& \geq \log h^{(n_{l_{0}})} \left( 1- \frac{\log (\kappa \nu h^{(n_{l_{0}})} / h^{(n_{l_{-2}})})}{ \log h^{(n_{l_{0}})} } \right) \\
& \geq \log h^{(n_{l_{0}})} \left( 1- \frac{\log (\kappa \nu \norma{A^{(n_{l_{-2}}, n_{l_{0}})}})}{ \log h^{(n_{l_{0}})} } \right),\label{eq:integral2}
\end{split}
\end{equation}
where we used the cocycle relation $ \underline{h}^{(n_{l_{0}})} = (A^{(n_{l_{-2}}, n_{l_{0}})})^T  \underline{h}^{(n_{l_{-2}})}$ to obtain ${h}^{(n_{l_{0}})} \leq \norma{A^{(n_{l_{-2}}, n_{l_{0}})}} {h}^{(n_{l_{-2}})}$. The term in brackets goes to 1 as $l_0$ goes to infinity because of Corollary \ref{cor:thulcigrai}-(iii), thus for $l_0$ sufficiently large we have obtained $ \log  1/{\lambda_{\text{sing}}^{(n_{l_{-2}})}} \geq (1-\varepsilon) \log h^{(n_{l_{0}})}$.

Combining the bounds \eqref{eq:lowerbound} with the estimates \eqref{eq:integral1} and \eqref{eq:integral2}, we deduce
\begin{equation}\label{eq:maic}
 e^{-2\varepsilon} (1-\varepsilon)^2 r_0 \log h^{(n_{l_{0}})} \leq \sum_{T^ix_0 \in C} \widetilde{w}(T^ix_0) \leq  e^{2\varepsilon} (1+\varepsilon)^2 r_0 \log h^{(n_{l_{0}})}.
\end{equation}

\medskip

\paragraph{Final estimates.}
Choose $\varepsilon >0$ such that $e^{2\varepsilon}(1+\varepsilon)^2 + \varepsilon < 1 + \varepsilon\rq{}$ and $e^{-2\varepsilon}(1-\varepsilon)^2 > 1 - \varepsilon\rq{}$.
As we have already remarked, the singular terms are nonzero if and only if $F^{(n_{l_0})}_{\text{sing}} \subset Z^{(n_{l_0})}_{j_0}$, in which case it equals $\max_{0\leq i < r_0} w(T^ix_0)$ and $\max_{0\leq i < r_0} \widetilde{w}(T^ix_0)$ respectively. Together with the estimates of the gap error \eqref{eq:gerrforw} and \eqref{eq:gerr} and of the main contribution \eqref{eq:upperboundforw} and \eqref{eq:maic}, this proves the estimates \eqref{eq:sbsweakforf} and \eqref{eq:sbsweak} for the special Birkhoff sums.

%%%%%%%%%%

\subsection{General case}\label{resonantterm}

Fix $\varepsilon\rq{}\rq{}>0$, $r \in \N$ and take $l$ such that $  h^{( n_{l})} \leq r <  h^{( n_{l+1})}$. 
In this section we want to estimate Birkhoff sums $S_r(w)(x_0)$ and $S_r(\widetilde{w})(x_0)$ for any orbit length $r$; namely we will prove that for any $r$ sufficiently large and for any $ x \notin \Sigma_l(k)$,
\begin{equation}\label{eq:conclusionforw}
S_r(w)(x_0) \leq (1+\varepsilon\rq{}\rq{}) r \int_0^1 w(x) \diff x + (\lfloor \kappa \rfloor +2) \max_{0 \leq i < r} w(T^ix_0),
\end{equation}
and
\begin{equation}
(1-\varepsilon\rq{}\rq{}) r \log r \leq S_r(\widetilde{w})(x_0) \leq (1+\varepsilon\rq{}\rq{}) r \log r + (\lfloor \kappa \rfloor +2) \max_{0 \leq i < r} \widetilde{w}(T^ix_0). \label{eq:conclusion}
\end{equation} 
The idea is to decompose $S_r(w)$ and $S_r(\widetilde{w})$ into special Birkhoff sums of previous steps $n_{l_i}$. To have control of the sum, however, we have to throw away the set $\Sigma_l(k)$ of points which go too close to the singularity, whose measure is small, see Proposition \ref{th:stretpart}.

\begin{defin}\label{definx}
Let $\orbita{r}{x} = \{T^ix: 0 \leq i < r\}$. We introduce the following notation: if $x \in I_j^{(n)}$, denote by $x_{j}^{(n)}$ and $\widetilde{x}_{j}^{(n)}$ the points in $\orbita{h_j^{(n)}}{x} \cap Z_j^{(n)}$ at which the functions $w$ and $\widetilde{w}$ attain their respective maxima, and by $x_r$ and $\widetilde{x}_r$ the points such that $w(x_r) = \max_{0 \leq i < r} w(T^ix_0)$ and $\widetilde{w}(\widetilde{x}_r) = \max_{0 \leq i < r} \widetilde{w}(T^ix_0)$.
\end{defin}

Suppose $x_0 \in Z_{j_0}^{(n)}$. By definition of the sets $Z_j^{(n)}$, there exist 
$$
0 \leq Q=Q(n) \leq r/ \min_j h^{(n)}_j \text{\ \ and\ \ }y_0^{(n)} \in I_{i_0}^{(n)}, y_1^{(n)} \in I_{i_1}^{(n)}, \dots, %y_Q^{(n)} \in I_{i_Q}^{(n)}, 
y_{Q+1}^{(n)} \in I_{i_{Q+1}}^{(n)},
$$ 
such that the orbit $\orbita{r}{x_0}$ can be decomposed as the disjoint union 
\begin{equation}\label{eq:decomposeorbit}
\bigsqcup_{\alpha = 1}^{Q(n)} \orbita{h_{i_\alpha}^{(n)}}{y_{\alpha}^{(n)}} \subset \orbita{r}{x_0} \subset \bigsqcup_{\alpha = 0}^{Q(n)+1} \orbita{h_{i_\alpha}^{(n)}}{y_{\alpha}^{(n)}}.
\end{equation}
This expression shows that we can approximate the Birkhoff sum along $\orbita{r}{x_0}$ with the sum of special Birkhoff sums.
We will need three levels of approximation $n_{l-L} < n_l < n_{l+1}$. Fix $L \in \N$ such that $2 \kappa d^{-L/\overline{l}} < \varepsilon$ and let $y_{\alpha}^{(n_{l-L})} \in I_{i_{\alpha}}^{(n_{l-L})}$ for $0 \leq \alpha \leq Q(n_{l-L})+1$, $I_{j_{\beta}}^{(n_{l})}$ for $0 \leq \beta \leq Q(n_{l})+1$ and $I_{q_{\gamma}}^{(n_{l+1})}$ for $0 \leq \gamma \leq Q(n_{l+1})+1$ be defined as above.

By the positivity of $w$ and \eqref{eq:decomposeorbit}, it follows
$$
\sum_{\alpha = 1}^{Q(n_{l-L})} S_{h_{i_{\alpha}}^{(n_{l-L})}}(w)(y_{\alpha}^{(n_{l-L})}) \leq S_r(w)(x_0) \leq  \sum_{\alpha =0}^{Q(n_{l-L})+1} S_{h_{i_{\alpha}}^{(n_{l-L})}}(w)(y_{\alpha}^{(n_{l-L})}),
$$
and similarly for $\widetilde{w}$. 
Let $\varepsilon\rq{} >0$ (to be determined later); each term is a special Birkhoff sum, so, by applying the estimates \eqref{eq:sbsweakforf} and \eqref{eq:sbsweak}, we get 
\begin{equation}\label{eq:genericbsforw}
S_r(w)(x_0) \leq (1+\varepsilon\rq{}) \Big( \int_0^1 w(x)\diff x \Big)\sum_{\alpha =0}^{Q(n_{l-L})+1} h_{i_{\alpha}}^{(n_{l-L})} + \sum_{\alpha =0}^{Q(n_{l-L})+1} w(x_{i_{\alpha}}^{(n_{l-L})}),
\end{equation}
and 
\begin{align}
&S_r(\widetilde{w})(x_0) \geq (1-\varepsilon\rq{}) \sum_{\alpha =1}^{Q(n_{l-L})} h_{\alpha}^{(n_{l-L})} \log h^{(n_{l-L})}, \label{eq:genericbs1}\\
&S_r(\widetilde{w})(x_0) \leq (1+\varepsilon\rq{}) \sum_{\alpha =0}^{Q(n_{l-L})+1} h_{\alpha}^{(n_{l-L})} \log h^{(n_{l-L})} + \sum_{\alpha =0}^{Q(n_{l-L})+1} \widetilde{w}(\widetilde{x}_{i_{\alpha}}^{(n_{l-L})}) \label{eq:genericbs2},
\end{align}
where $x_{i_{\alpha}}^{(n_{l-L})}$ and $\widetilde{x}_{i_{\alpha}}^{(n_{l-L})}$ are the points defined in Notation \ref{definx} at which the corresponding special Birkhoff sums of $w$ and $\widetilde{w}$ attain their respective maxima.
We refer to the first terms in the right-hand side of \eqref{eq:genericbsforw}, \eqref{eq:genericbs1} and \eqref{eq:genericbs2} as the \newword{ergodic terms} and to the second terms in the right-hand side of \eqref{eq:genericbsforw} and \eqref{eq:genericbs2} as the \newword{resonant terms}.

\medskip

\paragraph{Ergodic terms.}
The estimates of the ergodic terms for $\widetilde{w}$ are identical to \cite[pp.~1016-1017]{ulcigrai:mixing} and the estimate for $w$ can be deduced from the same proof. Explicitly, the ergodic term for $w$ is bounded above by $(1 + \varepsilon\rq{})^2 r \int w$, whence the ergodic terms for $\widetilde{w}$ are bounded below and above by $(1- \varepsilon\rq{})^2 r \log r$ and by $(1+ \varepsilon\rq{})^2 r \log r$ respectively.

\medskip

\paragraph{Resonant terms.}
We want to estimate the resonant terms $ \sum_{\alpha} w(x_{i_{\alpha}}^{(n_{l-L})})$ and $ \sum_{\alpha} \widetilde{w}(\widetilde{x}_{i_{\alpha}}^{(n_{l-L})})$. First, we reduce to consider the maxima over sets $Z$ of step $n_l$ instead of step $n_{l-L}$ by comparing the sum with an arithmetic progression, as we did in the estimates for the gap error in~\S\ref{gaperror}. 

Let $\varepsilon>0$. Again, we first consider $\widetilde{w}$.
Group the summands according to the decomposition as in \eqref{eq:decomposeorbit} of step $n_l$, so that
$$
%\sum_{\alpha =0}^{Q(n_{l-L})+1} \widetilde{w}(\widetilde{x}_{i_{\alpha}}^{(n_{l-L})}) = \sum_{\beta = 0}^{Q(n_{l})+1}\  \sum_{0\leq \alpha \leq Q(n_{l-L})+1\ :\ \widetilde{x}_{i_{\alpha}}^{(n_{l-L})} \in  Z_{j_{\beta}}^{(n_{l})} } \widetilde{w}(\widetilde{x}_{i_{\alpha}}^{(n_{l-L})}).
\sum_{\alpha =0}^{Q(n_{l-L})+1} \widetilde{w}(\widetilde{x}_{i_{\alpha}}^{(n_{l-L})}) = \sum_{\beta = 0}^{Q(n_{l})+1}\  \sum_{ \alpha \ :\ y_{\alpha}^{(n_{l-L})} \in  \orbita{h_{j_{\beta}}^{(n_{l})}}{ y_{\beta}^{(n_l)} } } \widetilde{w}(\widetilde{x}_{i_{\alpha}}^{(n_{l-L})}).
$$
For any fixed $\beta = 0, \dots, Q(n_{l})+1$, each of the points $\widetilde{x}_{i_{\alpha}}^{(n_{l-L})} \in  \orbita{h_{i_{\alpha}}^{(n_{l-L})}}{ y_{\alpha}^{(n_{l-L})} } $ appearing in the second sum in the right-hand side above
belongs to a different interval of $ Z_{j_{\beta}}^{(n_{l})}$, hence the distance between any two of them is at least $\lambda_{j_{\beta}}^{(n_l)} \geq (d \kappa \nu h_{j_{\beta}}^{(n_l)})^{-1}$. Moreover, the number of the points $\widetilde{x}_{i_{\alpha}}^{(n_{l-L})}$ contained in $ Z_{j_{\beta}}^{(n_{l})}$ is bounded by $\norma{A^{(n_{l-L},n_l)}}$.

Fix $0 \leq \beta \leq Q(n_{l})+1$; we separate the point $\widetilde{x}_{j_{\beta}}^{(n_{l})}$ corresponding to the maximum of $\widetilde{w}$ in $Z_{j_{\beta}}^{(n_{l})}$ from the others,
\begin{multline*}
\sum_{ \alpha \ :\ y_{\alpha}^{(n_{l-L})} \in  \orbita{h_{j_{\beta}}^{(n_{l})}}{ y_{\beta}^{(n_l)} } } \widetilde{w}(\widetilde{x}_{i_{\alpha}}^{(n_{l-L})}) = \widetilde{w}(\widetilde{x}_{j_{\beta}}^{(n_{l})}) \\
+  \sum_{ \alpha \ :\ y_{\alpha}^{(n_{l-L})} \in  \orbita{h_{j_{\beta}}^{(n_{l})}}{ y_{\beta}^{(n_l)} } ,\ \widetilde{x}_{i_{\alpha}}^{(n_{l-L})} \neq \widetilde{x}_{j_{\beta}}^{(n_{l})}} \widetilde{w}(\widetilde{x}_{i_{\alpha}}^{(n_{l-L})}).
\end{multline*}
If $\widetilde{x}_{i_{\alpha}}^{(n_{l-L})} \neq \widetilde{x}_{j_{\beta}}^{(n_{l})}$, then $\widetilde{x}_{i_{\alpha}}^{(n_{l-L})}$ does not belong to the interval of $\mathcal{Z}^{(n_l)}$ containing $a_k$ as left endpoint if $\widetilde{w}= \widetilde{u}_k$ or right endpoint  if $\widetilde{w}= \widetilde{v}_k$. 
Since $\widetilde{w}$ has only a one-side singularity and is monotone, the value $\widetilde{w}(\widetilde{x}_{i_{\alpha}}^{(n_{l-L})}) $ is bounded by the inverse of the distance between $a_k$ and the second closest return to the right of $a_k$ if $\widetilde{w}= \widetilde{u}_k$ or to the left if $\widetilde{w}= \widetilde{v}_k$; in both cases we have that $\widetilde{w}(\widetilde{x}_{i_{\alpha}}^{(n_{l-L})}) \leq 1/\lambda_{j_{\beta}}^{(n_l)}$. Moreover, $\modulo{\widetilde{x}_{i_{\alpha}}^{(n_{l-L})} - \widetilde{x}_{i_{\alpha\rq{}}}^{(n_{l-L})}} \geq (d \kappa \nu h_{j_{\beta}}^{(n_l)})^{-1}$ thus 
we can bound the second sum above with an arithmetic progression of length $\norma{A^{(n_{l-L},n_l)}}$. Reasoning as in~\S\ref{gaperror} we obtain
\begin{equation*}
\begin{split}
& \sum_{  \alpha : y_{\alpha}^{(n_{l-L})} \in  \orbita{h_{j_{\beta}}^{(n_{l})}}{ y_{\beta}^{(n_l)} }  } \widetilde{w}(\widetilde{x}_{i_{\alpha}}^{(n_{l-L})})  \leq   \widetilde{w}(\widetilde{x}_{j_{\beta}}^{(n_{l})}) +  \sum_{i= 1}^{\norma{A^{(n_{l-L},n_l)}}} \left(  \lambda_{j_{\beta}}^{(n_l)} + \frac{i}{ d \kappa \nu h_{j_{\beta}}^{(n_l)} }  \right)^{-1}\\
& \qquad \qquad  \leq \widetilde{w}(\widetilde{x}_{j_{\beta}}^{(n_{l})}) +  d \kappa \nu  \log h_{j_{\beta}}^{(n_l)} (1+ \log (\norma{A^{(n_{l-L},n_l)}} +1)).
\end{split}
\end{equation*}
Therefore
\begin{equation}
\begin{split}
\sum_{\alpha =0}^{Q(n_{l-L})+1} \widetilde{w}(\widetilde{x}_{i_{\alpha}}^{(n_{l-L})}) & \leq  \sum_{\beta = 0}^{Q(n_{l})+1} d \kappa \nu  h_{j_{\beta}}^{(n_l)}( 1+\log(\norma{A^{(n_{l-L},n_l)}}+1)) \\
& \qquad  +  \sum_{\beta = 0}^{Q(n_{l})+1}  \widetilde{w}(\widetilde{x}_{j_{\beta}}^{(n_{l})}) . \label{eq:resonantterm}
 \end{split}
\end{equation}

The first term on the right-hand side in \eqref{eq:resonantterm} has the desired asymptotic behavior. Indeed, from \eqref{eq:decomposeorbit} we obtain
$$
\sum_{\beta = 1}^{Q(n_{l})}  h_{j_{\beta}}^{(n_l)} \leq r \leq \sum_{\beta = 0}^{Q(n_{l})+1}  h_{j_{\beta}}^{(n_l)} \leq \sum_{\beta = 1}^{Q(n_{l})}  h_{j_{\beta}}^{(n_l)} + 2 h^{(n_l)} \leq r + 2 h^{(n_l)},
$$ 
so that $ (\sum_{\beta}  h_{j_{\beta}}^{(n_l)} )/r \leq1+ 2h^{(n_l)}/r \leq 3 $. Moreover $\log(\norma{A^{(n_{l-L},n_l)}}+1) / \log r \to 0$, by Corollary \ref{cor:thulcigrai}-(iii); for $l$ sufficiently big we then have
\begin{equation}
d \kappa \nu \left(  \sum_{\beta= 0}^{Q(n_{l})+1}   h_{j_{\beta}}^{(n_l)} \right)( 1+\log(\norma{A^{(n_{l-L},n_l)}}+1)) \leq \varepsilon {r \log r}. \label{eq:resonantterm1}
\end{equation}
Therefore, \eqref{eq:resonantterm} becomes
\begin{equation}\label{eq:nonso1}
\sum_{\alpha =0}^{Q(n_{l-L})+1} \widetilde{w}(\widetilde{x}_{i_{\alpha}}^{(n_{l-L})}) \leq  \varepsilon {r \log r} +  \sum_{\beta = 0}^{Q(n_{l})+1}  \widetilde{w}(\widetilde{x}_{j_{\beta}}^{(n_{l})}). 
\end{equation}

The analogous approach for $w$ yields
\begin{equation*}
\begin{split}
\sum_{\alpha =0}^{Q(n_{l-L})+1} w(x_{i_{\alpha}}^{(n_{l-L})}) & \leq \sum_{\beta = 0}^{Q(n_{l})+1} w(x_{j_{\beta}}^{(n_{l})}) +  \sum_{\beta = 0}^{Q(n_{l})+1}\norma{A^{(n_{l-L},n_l)}} \left( 1- \log ( \lambda_{j_{\beta}}^{(n_l)} ) \right) \\
& \leq \sum_{\beta = 0}^{Q(n_{l})+1} w(x_{j_{\beta}}^{(n_{l})}) + 2 \norma{A^{(n_{l-L},n_l)}} (Q(n_{l})+2) \log h^{(n_l)}.
\end{split}
\end{equation*}
Recalling that $Q(n_{l})$ is the number of special Birkhoff sums of level $n_l$ needed to approximate the original Birkhoff sum along $\orbita{r}{x_0}$ as in \eqref{eq:decomposeorbit}, it follows that $Q(n_{l}) \leq r/ \min_j h_j^{(n_l)} \leq \kappa r / h^{(n_l)}$. By Corollary \ref{cor:thulcigrai}-(ii), $\norma{A^{(n_{l-L},n_l)}} \leq l^{L\tau} = O\left( (\log h^{(n_l)})^{L\tau} \right)$; hence we conclude
\begin{equation}\label{eq:nonso2}
\begin{split}
\sum_{\alpha =0}^{Q(n_{l-L})+1} w(x_{i_{\alpha}}^{(n_{l-L})}) & = O \left( \Big( \frac{r}{h^{(n_l)}}\Big) (\log h^{(n_l)})^{1+L\tau} \right) +\sum_{\beta = 0}^{Q(n_{l})+1} w(x_{j_{\beta}}^{(n_{l})}) \\
&\leq \varepsilon r +\sum_{\beta = 0}^{Q(n_{l})+1} w(x_{j_{\beta}}^{(n_{l})}).
\end{split}
\end{equation}

Thus, it remains to bound the second summands in \eqref{eq:nonso1} and \eqref{eq:nonso2}. To do that, we proceed in two different ways depending on $r$ being closer to $h^{(n_{l+1})}$ or to $h^{(n_l)}$.
Recalling the definitions of $\sigma_l$ and of $\Sigma_l(k)$ introduced in~\S\ref{section:quantitative}, we distinguish two cases.

\emph{Case 1.} Suppose that $\sigma_l h^{(n_{l+1})} \leq r < h^{(n_{l+1})}$. We compare the second summand in \eqref{eq:nonso1} with an arithmetic progression and the second summand in \eqref{eq:nonso2} in the same way as above, considering $n_l$ and $n_{l+1}$ instead of $n_{l-L}$ and $n_l$: we obtain
\begin{equation}\label{eq:remarkallforw}
 \sum_{\beta = 0}^{Q(n_{l})+1} w(x_{j_{\beta}}^{(n_{l})}) \leq 2 \norma{A^{(n_{l},n_{l+1})}} \sum_{\gamma = 0}^{Q(n_{l+1})+1} \log  h^{(n_{l+1})} +  \sum_{\gamma = 0}^{Q(n_{l+1})+1}  w(x_{q_{\gamma}}^{(n_{l+1})}),
\end{equation}
and
\begin{equation}\label{eq:remarkall}
\begin{split}
 \sum_{\beta = 0}^{Q(n_{l})+1} \widetilde{w}(\widetilde{x}_{j_{\beta}}^{(n_{l})}) &\leq \sum_{\gamma = 0}^{Q(n_{l+1})+1} d \kappa \nu h_{q_{\gamma}}^{(n_{l+1})}( 1+\log(\norma{A^{(n_{l},n_{l+1})}}+1)) \\
 & \qquad + \sum_{\gamma = 0}^{Q(n_{l+1})+1} \widetilde{w}(\widetilde{x}_{q_{\gamma}}^{(n_{l+1})}).
\end{split}
\end{equation}
Since $r < h^{(n_{l+1})} \leq \kappa \min_j h_j^{(n_{l+1})}$, as before we have that $Q(n_{l+1}) \leq r/ \min_j h_j^{(n_{l+1})} \leq \lfloor \kappa \rfloor$; therefore the second terms on the right-hand side of \eqref{eq:remarkallforw} and \eqref{eq:remarkall} are bounded by $(\lfloor \kappa \rfloor +2) w( x_r)$ and $(\lfloor \kappa \rfloor +2) \widetilde{w}( \widetilde{x}_r)$ respectively. We now bound the first summand in the right-hand side of \eqref{eq:remarkallforw}. We have that $\norma{A^{(n_{l},n_{l+1})}} \leq l^{\tau} = O\left( (\log  h^{(n_{l})})^{\tau} \right) = O\left( (\log r)^{\tau} \right)$ as in the proof of Lemma \ref{th:stimelt}. Moreover, we use the estimate $h^{(n_{l+1})} / r \leq 1/ \sigma_l$ to get
$$
\norma{A^{(n_{l},n_{l+1})}} \sum_{\gamma = 0}^{Q(n_{l+1})+1} \log  h^{(n_{l+1})}  = O \left( (\log r)^{1+\tau} - \log r \log \sigma_l\right) \leq \varepsilon r,
$$
since $|\log \sigma_l| = O(\log \log h^{(n_l)}) = o(\log r)$, which is easy to check from the definition of $\sigma_l$.
On the other hand, as regards the first summand in the right-hand side of \eqref{eq:remarkall}, we have
\begin{multline*}
d \kappa \nu \left( \frac{ \sum_{\gamma}  h_{q_{\gamma}}^{(n_{l+1})} }{r} \right) \frac{( 1+\log(\norma{A^{(n_{l},n_{l+1})}}+1))}{ \log r}\\
 \leq d \kappa \nu \frac{( \lfloor \kappa \rfloor +2)}{\sigma_l} \frac{ ( 1+\log(\norma{A^{(n_{l},n_{l+1})}}+1)) }{\log \left( \sigma_l h^{(n_{l+1})}\right)},
\end{multline*}
which can be made arbitrary small by enlarging $l$. Therefore, 
\begin{equation}\label{eq:resonantterm2forw}
\sum_{\beta = 0}^{Q(n_{l})+1} w(x_{j_{\beta}}^{(n_{l})}) \leq \varepsilon r +  (\lfloor \kappa \rfloor +2) w( x_r)
\end{equation}
and 
\begin{equation}
\sum_{\beta = 0}^{Q(n_{l})+1} \widetilde{w}( \widetilde{x}_{j_{\beta}}^{(n_{l})}) \leq \varepsilon r \log r + (\lfloor \kappa \rfloor +2)  \widetilde{w}(  \widetilde{x}_r). \label{eq:resonantterm2}
\end{equation}

\emph{Case 2.} Now suppose $h^{(n_l)} \leq r < \sigma_l h^{(n_{l+1})}$. 
If the initial point $x_0 \notin \Sigma_l(k)$, for any $0 \leq i \leq \lfloor \sigma_l h^{(n_{l+1})} \rfloor$ we know that $\modulo{T^ix_0 - a_k} \geq \sigma_l \lambda^{(n_l)} \geq  \sigma_l/ h^{(n_l)}$, since $1 = \sum_j h_j^{(n_l)} \lambda_j^{(n_l)} \leq  h^{(n_l)} \sum_j \lambda_j^{(n_l)}= h^{(n_l)} \lambda^{(n_l)}$. In particular, we have that $w(x_r) \leq 1+\log h^{(n_l)}$ and  $\widetilde{w}(\widetilde{x}_r) \leq h^{(n_l)}/\sigma_l$.

Obviously, 
$$
\sum_{\beta = 0}^{Q(n_{l})+1} w(x_{j_{\beta}}^{(n_{l})}) \leq (Q(n_{l})+2) w(x_r), \qquad \sum_{\beta = 0}^{Q(n_{l})+1} \widetilde{w}(\widetilde{x}_{j_{\beta}}^{(n_{l})}) \leq (Q(n_{l})+2) \widetilde{w}(\widetilde{x}_r),
$$ 
and we recall $Q(n_{l}) \leq r/ \min_j h_j^{(n_l)} \leq \kappa r / h^{(n_l)}$. Therefore,
\begin{equation}\label{eq:resonantterm3forw}
\sum_{\beta = 0}^{Q(n_{l})+1} w(x_{j_{\beta}}^{(n_{l})}) \leq \left( \frac{\kappa r}{h^{(n_l)}}+2  \right) (1+\log h^{(n_l)}) \leq \varepsilon r
\end{equation}
and
$$
\sum_{\beta = 0}^{Q(n_{l})+1} \widetilde{w}(\widetilde{x}_{j_{\beta}}^{(n_{l})}) \leq  \left( \frac{\kappa r}{h^{(n_l)}}+2  \right) \frac{h^{(n_l)}}{\sigma_l} = \frac{\kappa r + 2 h^{(n_l)}}{\sigma_l}.
$$
Since $h^{(n_l)} \leq r$ and $\log r / \log h^{(n_l)} \geq 1$ we can write
\begin{equation}
\sum_{\beta = 0}^{Q(n_{l})+1} \widetilde{w}(\widetilde{x}_{j_{\beta}}^{(n_{l})}) \leq  \left( \frac{\kappa + 2}{\sigma_l \log h^{(n_l)}}  \right) r \log r, \label{eq:resonantterm3}
\end{equation}
and the term in brackets can be made smaller than $\varepsilon$ by choosing $l$ big enough \cite[Lemma~3.9]{ulcigrai:mixing}. 

\medskip

\paragraph{Final estimates.}
For any $r$ as in Case 1, for any $x_0$, by combining \eqref{eq:nonso2} with \eqref{eq:resonantterm2forw} and \eqref{eq:nonso1} with \eqref{eq:resonantterm2},
\begin{equation*}
\begin{split}
&\sum_{\alpha =0}^{Q(n_{l-L})+1} w(x_{i_{\alpha}}^{(n_{l-L})})  \leq 2 \varepsilon r +  (\lfloor \kappa \rfloor +2) w( x_r), \\  
&\sum_{\alpha =0}^{Q(n_{l-L})+1} \widetilde{w}(\widetilde{x}_{i_{\alpha}}^{(n_{l-L})}) \leq  2 \varepsilon r \log r + (\lfloor \kappa \rfloor +2)  \widetilde{w}(  \widetilde{x}_r);
\end{split}
\end{equation*}
whence, for any $r$ as in Case 2 and for all $x \notin \Sigma_l(k)$,  by combining \eqref{eq:nonso2} with \eqref{eq:resonantterm3forw} and \eqref{eq:nonso1} with \eqref{eq:resonantterm2},
$$
\sum_{\alpha =0}^{Q(n_{l-L})+1} w(x_{i_{\alpha}}^{(n_{l-L})})  \leq 2 \varepsilon r, \quad 
\sum_{\alpha =0}^{Q(n_{l-L})+1} \widetilde{w}(\widetilde{x}_{i_{\alpha}}^{(n_{l-L})}) \leq  2 \varepsilon r \log r.
$$
These estimates together with those for the ergodic terms prove \eqref{eq:conclusionforw} and \eqref{eq:conclusion}, choosing $\varepsilon, \varepsilon\rq{}>0$ appropriately.

%%%%%%%%%%

\subsection{Proof of Theorem \ref{th:BS}}
By the hypothesis on the roof function $f$ we can write
\begin{equation}\label{eq:writef}
\begin{split}
&f(x) = \sum_{k=1}^{d-1} (C_k^{+} u_k(x) + C_k^{-} v_k(x)) + e(x), \\
&f\rq{}(x) = \sum_{k=1}^{d-1} (-C_k^{+} \widetilde{u}_k(x) + C_k^{-} \widetilde{v}_k(x)) + e\rq{}(x),
\end{split}
\end{equation}
for a smooth function $e$.
Fix $\epsilon < \varepsilon/(C^{+}+C^{-})$ and choose $\overline{r} \geq 1$ such that if $r \geq \overline{r}$ the estimates \eqref{eq:conclusionforw} and \eqref{eq:conclusion} hold with respect to $\epsilon$. By unique ergodicity of $T$, up to enlarging $\overline{r}$, we have that $S_r(e)(x) \leq (1 + \epsilon)r \int e$.

The estimates \eqref{eq:conclusionforw} imply
\begin{equation*}
\begin{split}
S_r(f)(x_0) &\leq (1+\epsilon) r \sum_{k=1}^{d-1} \left(C_k^{+} \int_0^1 u_k(x) \diff x+ C_k^{-} \int_0^1 v_k(x) \diff x\right) \\
& \qquad + (1 + \epsilon)r \int_0^1 e(x) \diff x \\
&\leq (1+\epsilon)r \int_0^1 f(x) \diff x \\
& \qquad + 2 (d-1) ( \lfloor \kappa \rfloor +2) \max_{1\leq k \leq d-1} \max_{0 \leq i <r} \modulo{\log \modulo{T^ix_0 - a_k}}\\
&\leq 2 r + \const \max_{1\leq k \leq d-1} \max_{0 \leq i <r} \modulo{\log \modulo{T^ix_0 - a_k}}.
\end{split}
\end{equation*}

Considering the derivative $f\rq{}$, from the estimates \eqref{eq:conclusion} we get
\begin{equation*}
\begin{split}
S_r(f\rq{})(x_0) & \leq -C^{+} (1-\epsilon) r \log r + C^{-} (1 + \epsilon) r \log r + C^{-} ( \lfloor \kappa \rfloor +2)  \widetilde{V}(r,x) \\
& \leq (-C^{+} + C^{-}+\varepsilon) r \log r + C^{-} ( \lfloor \kappa \rfloor +2) \widetilde{V}(r,x),
\end{split}
\end{equation*}
and similarly
\begin{equation*}
\begin{split}
S_r(f\rq{})(x_0) & \geq -C^{+} (1+\epsilon) r \log r - C^{+} ( \lfloor \kappa \rfloor +2)  \widetilde{U}(r,x)+ C^{-} (1 - \epsilon) r \log r  \\
& \leq (-C^{+} + C^{-}-\varepsilon) r \log r - C^{+} ( \lfloor \kappa \rfloor +2) \widetilde{U}(r,x).
\end{split}
\end{equation*}

Let us estimate the Birkhoff sum of the second derivative $f\rq{}\rq{}$.
By deriving \eqref{eq:writef}, if $x_0$ is not a singularity of $S_r(f)$, we have
\begin{equation*}
%\begin{split}
\modulo{S_r(f\rq{}\rq{})(x_0)} \leq \sum_{k=1}^d \left( C^{+}_k S_r(\widetilde{u}_k^2)(x_0) + C^{-}_k S_r(\widetilde{v}_k^2)(x_0) \right) + r \max_{x \in I}\modulo{e\rq{}\rq{}(x)}. %\\
%& \leq 2 \widetilde{U}(r,x) \sum_{k=1}^d C_k^{+}  S_r(\widetilde{u}_k)(x) + 2 \widetilde{V}(r,x) \sum_{k=1}^d C_k^{-}  S_r(\widetilde{v}_k)(x) + r \max_{x \in K}f\rq{}\rq{}(x).
%\end{split}
\end{equation*}
Since $S_r(\widetilde{u}_k^2)(x_0) \leq \left(\max_{0 \leq i < r} \widetilde{u}_k(T^ix_0) \right)S_r(\widetilde{u}_k)(x_0)$ and similarly for $\widetilde{v}_k$, we get
\begin{equation*}
\begin{split}
\modulo{S_r(f\rq{}\rq{})(x_0)} &\leq \widetilde{U}(r,x) \sum_{k=1}^d C_k^{+}  S_r(\widetilde{u}_k)(x_0) + \widetilde{V}(r,x) \sum_{k=1}^d C_k^{-}  S_r(\widetilde{v}_k)(x_0) \\
& \qquad+  r \max_{x \in I}\modulo{e\rq{}\rq{}(x)},
\end{split}
\end{equation*}
where we recall 
$$
\widetilde{U}(r,x) := \max_{1\leq k \leq d-1} \max_{0 \leq i < r} \widetilde{u}_k(T^ix), \qquad \widetilde{V}(r,x) := \max_{1 \leq k \leq d-1} \max_{0 \leq i < r} \widetilde{v}_k(T^ix).
$$
Up to increasing $\overline{r}$, we have that $ \max_{x \in I}\modulo{e\rq{}\rq{}(x)}  \leq \varepsilon \log r$; thus one can proceed as before to get the desired estimate.

%%%%%%%%%%%%%%%%%%%%
%%%%%%%%%%%%%%%%%%%%
%%%%%%%%%%%%%%%%%%%%


\begin{thebibliography}{10}

\bibitem{arnold:torus}
V.I. Arnold.
\newblock Topological and ergodic properties of closed 1-forms with rationally
  independent periods.
\newblock {\em Functional Analysis and Its Applications}, 25(2):81--90, 1991.

\bibitem{athreyaforni:iet}
J.~S. Athreya and G.~Forni.
\newblock Deviation of ergodic averages for rational polygonal billiards.
\newblock {\em Duke Math. J.}, 144(2):285--319, 08 2008.

\bibitem{chaika:example}
J.~Chaika and A.~Wright.
\newblock A smooth mixing flow on a surface with non-degenerate fixed points.
\newblock {\em preprint arXiv:1501.02881}, 2015.

\bibitem{fayad:quantmix}
B.~Fayad.
\newblock Polynomial decay of correlations for a class of smooth flows on the
  two torus.
\newblock {\em Bull. Soc. Math. France}, 129(4):487--503, 2001.

\bibitem{forniulcigrai:timechanges}
G.~Forni and C.~Ulcigrai.
\newblock Time-changes of horocycle flows.
\newblock {\em Journal of Modern Dynamics}, 6(2):251--273, 4 2012.

\bibitem{fraczek:infinite}
K.~Frączek and C.~Ulcigrai.
\newblock Ergodic properties of infinite extensions of area-preserving flows.
\newblock {\em Mathematische Annalen}, 354(4):1289--1367, 2012.

\bibitem{KKU:multmixing}
A.~Kanigowski, J.~Kulaga-Przymus, and C.~Ulcigrai.
\newblock Multiple mixing and parabolic divergence in minimal components of
  smooth area-preserving flows on higher genus surfaces.
\newblock {\em preprint arXiv:1606.09189}, 2016.

\bibitem{kochergin:absence}
A.V. Kochergin.
\newblock The absence of mixing in special flows over a rotation of the circle
  and in flows on a two-dimensional torus.
\newblock {\em Soviet Mathematics. Doklady}, 13:949--952, 1972.

\bibitem{kochergin:lemma}
A.V. Kochergin.
\newblock Mixing in special flows over a shifting of segments and in smooth
  flows on surfaces.
\newblock {\em Sbornik: Mathematics}, 96(138):471--502, 1975.

\bibitem{kochergin:degenerate}
A.V. Kochergin.
\newblock On mixing in special flows over a rearrangement of segments and in
  smooth flows on surfaces.
\newblock {\em Sbornik: Mathematics}, 25(3):441--469, 1975.

\bibitem{kochergin:absence2}
A.V. Kochergin.
\newblock Nonsingular saddle points and the absence of mixing.
\newblock {\em Mathematical Notes of the Academy of Sciences of the USSR},
  19(3):277--286, 1976.

\bibitem{kochergin:non}
A.V. Kochergin.
\newblock Non-degenerate fixed points and mixing in flows on a 2-torus.
\newblock {\em Sbornik: Mathematics}, 194(8):83--112, 2003.

\bibitem{kochergin:non2}
A.V. Kochergin.
\newblock Non-degenerate fixed points and mixing in flows on a 2-torus ii.
\newblock {\em Sbornik: Mathematics}, 195(3):317--346, 2004.

\bibitem{kochergin:torus}
A.V. Kochergin.
\newblock Some generalizations of theorems on mixing for flows with
  nondegenerate saddles on a two-dimensional torus.
\newblock {\em Sbornik: Mathematics}, 195(9):19--36, 2004.

\bibitem{kochergin:torus2}
A.V. Kochergin.
\newblock Well-approximable angles and mixing for flows on $\mathbb{T}^2$ with
  nonsingular fixed points.
\newblock {\em Electronic Research Announcements of the American Mathematical
  Society}, 10(13):113--121, 2004.

\bibitem{levitt:feuilletages}
G.~Levitt.
\newblock Feuilletages des surfaces.
\newblock {\em Annales de l'institut Fourier}, 32(2):179--217, 1982.

\bibitem{marcus:horocycle}
B.~Marcus.
\newblock Ergodic properties of horocycle flows for surfaces of negative
  curvature.
\newblock {\em Annals of mathematics}, 105(1):81--105, 1977.

\bibitem{marmimoussayoccoz:linearization}
S.~Marmi, P.~Moussa, and J.-C. Yoccoz.
\newblock Linearization of generalized interval exchange maps.
\newblock {\em Annals of mathematics}, 176(3):1583--1646, 2012.

\bibitem{masur:ergodic}
H.~Masur.
\newblock Interval exchange transformations and measured foliations.
\newblock {\em Annals of Mathematics}, 115(1):169--200, 1982.

\bibitem{mayer:minimal}
A.A. Mayer.
\newblock Trajectories on the closed orientable surfaces.
\newblock {\em Matematicheskii Sbornik}, 54(1):71--84, 1943.

\bibitem{milnor:morse}
J.W. Milnor.
\newblock {\em Morse Theory}.
\newblock Annals of mathematics studies. Princeton University Press, 1963.

\bibitem{moser:trick}
J.~Moser.
\newblock On the volume elements on a manifold.
\newblock {\em Transactions of the American Mathematical Society}, pages
  286--294, 1965.

\bibitem{novikov:hamiltonian}
S.P. Novikov.
\newblock The hamiltonian formalism and a many-valued analogue of morse theory.
\newblock {\em Russian mathematical surveys}, 37(5):1--56, 1982.

\bibitem{pajitnov:morse}
A.V. Pajitnov.
\newblock {\em Circle-valued Morse Theory}.
\newblock De Gruyter Studies in Mathematics. De Gruyter, 2006.

\bibitem{sinai:mixing}
Ya.G. Sinai and K.M. Khanin.
\newblock Mixing for some classes of special flows over rotations of the
  circle.
\newblock {\em Functional Analysis and Its Applications}, 26(3):155--169, 1992.

\bibitem{ulcigrai:mixing}
C.~Ulcigrai.
\newblock Mixing of asymmetric logarithmic suspension flows over interval
  exchange transformations.
\newblock {\em Ergodic Theory and Dynamical Systems}, 27(3):991--1035, 2007.

\bibitem{ulcigrai:weakmixing}
C.~Ulcigrai.
\newblock Weak mixing for logarithmic flows over interval exchange
  transformations.
\newblock {\em Journal of Modern Dynamics}, 3(1):35--49, 2009.

\bibitem{ulcigrai:absence}
C.~Ulcigrai.
\newblock Absence of mixing in area-preserving flows on surfaces.
\newblock {\em Annals of mathematics}, 173(3):1743--1778, 2011.

\bibitem{veech:ergodic}
W.A. Veech.
\newblock Gauss measures for transformations on the space of interval exchange
  maps.
\newblock {\em Annals of Mathematics}, 115(2):201--242, 1982.

\bibitem{viana:iet2}
M.~Viana.
\newblock Dynamics of interval exchange transformations and teichmueller flows.
\newblock Lecture notes available at the author\rq{}s webpage.

\bibitem{viana:iet}
M.~Viana.
\newblock Ergodic theory of interval exchange maps.
\newblock {\em Revista Matem{\'a}tica Complutense}, 19(1):7--100, 2006.

\bibitem{zorich:hyperplane}
A.~Zorich.
\newblock How do the leaves of a closed 1-form wind around a surface?
\newblock {\em Pseudoperiodic topology}, (197):135--181, 1999.

\end{thebibliography}
\end{document}